\documentclass[11pt]{amsart}

 \usepackage{amsfonts,graphics,amsmath,amsthm,amsfonts,amscd,amssymb,amsmath,latexsym,multicol,
 mathrsfs}
\usepackage{epsfig,url}
\usepackage{flafter}
\usepackage{fancyhdr}
\usepackage{hyperref}
\hypersetup{colorlinks=true, linkcolor=black}

 \usepackage[matrix, arrow]{xy}

\DeclareMathOperator{\Div}{Div}

\DeclareMathOperator{\lct}{lct}

\DeclareMathOperator{\Pic}{Pic}

\DeclareMathOperator{\Supp}{Supp}

\DeclareMathOperator{\vol}{vol}

 \numberwithin{equation}{subsection}
 \numberwithin{footnote}{subsection}

 \newtheorem{cor}[subsection]{Corollary}
 \newtheorem{lem}[subsection]{Lemma}
 \newtheorem{prop}[subsection]{Proposition}
 \newtheorem{thm}[subsection]{Theorem}
 \newtheorem{conj}[subsection]{Conjecture}

{
\theoremstyle{upright}

 \newtheorem{defn}[subsection]{Definition}

 \newtheorem{rem}[subsection]{Remark}

}

 \newcommand{\N}{\mathbb N}
 \newcommand{\PP}{\mathbb P}
 \newcommand{\A}{\mathbb A}
 \newcommand{\Q}{\mathbb Q}
 \newcommand{\R}{\mathbb R}
 \newcommand{\Z}{\mathbb Z}
  \newcommand{\C}{\mathbb C}
 \newcommand{\bir}{\dashrightarrow}
 \newcommand{\rddown}[1]{\left\lfloor{#1}\right\rfloor} 

\title{\large L\MakeLowercase{og} C\MakeLowercase{alabi}-Y\MakeLowercase{au fibrations}}
\thanks{2010 MSC:
14J45, 
14J32,  
14J10,  
14E30, 
14J17, 
14C20, 
14E05. 
}
\author{\large C\MakeLowercase{aucher} B\MakeLowercase{irkar}}
\date{\today}
\begin{document}
\maketitle

\begin{abstract}
In this paper we study boundedness properties and singularities of log Calabi-Yau fibrations, particularly those 
admitting Fano type structures. A log Calabi-Yau fibration roughly consists of a pair $(X,B)$ with good singularities 
and a  projective morphism $X\to Z$  such that $K_X+B$ is numerically trivial 
over $Z$. This class includes many central ingredients of birational geometry such as 
Calabi-Yau and Fano varieties and also fibre spaces of such varieties, 
flipping and divisorial contractions, crepant models, germs of singularities, etc.  

\end{abstract}

\tableofcontents


\section{\bf Introduction}

We work over a fixed algebraically closed field $k$ of characteristic zero unless stated otherwise.
According to the minimal model  program (including the abundance conjecture) 
every variety $W$ is birational to a projective variety $X$ 
with good singularities such that either 
\begin{itemize}
\item $X$ is canonically polarised (i.e. $K_X$ is ample), or

\item  $X$ admits a Mori-Fano fibration $X\to Z$ (i.e. $K_X$ is anti-ample over $Z$), or 

\item $X$ admits a Calabi-Yau fibration $X\to Z$ (i.e. $K_X$ is numerically trivial over $Z$).
\end{itemize} 
This reduces the birational classification 
of algebraic varieties to classifying such  $X$. From the point of view of moduli theory 
it makes perfect sense to focus on such $X$ as they have a better chance of having a 
reasonable moduli theory due to the special geometric structures they carry. For this 
and other reasons Fano and Calabi-Yau varieties and their fibrations are central to birational geometry. 
They are also of great importance in many other parts of mathematics such as arithmetic geometry, differential 
geometry, mirror symmetry, and mathematical physics.

Boundedness properties of canonically polarised varieties and Fano varieties have been extensively studied  
in the literature leading to recent advances [\ref{HMX-moduli}][\ref{B-compl}][\ref{B-BAB}] but much less is known 
about Calabi-Yau varieties. With the above philosophy of the minimal model program in mind, 
there is a natural urge to extend such studies to the more general framework of Fano and Calabi-Yau fibrations. 
It is also more fruitful and more flexible to discuss this in the context of pairs. 

Now we introduce the notion which unifies many central ingredients of birational geometry. 
A \emph{log Calabi-Yau fibration} consists of a pair $(X,B)$ with log canonical singularities and 
a contraction $f\colon X\to Z$ (i.e. a surjective projective morphism 
with connected fibres) such that $K_X+B\sim_\R 0$ relatively over $Z$. 
We usually denote the fibration by $(X,B)\to Z$. 
Note that we allow the two extreme cases: when $f$ is birational and when $f$ is constant.
When $f$ is birational such a fibration is a crepant model of $(Z,f_*B)$ (see below).
When $f$ is constant, that is, when $Z$ is a point, we just say $(X,B)$ is a \emph{log Calabi-Yau pair}.  
In general, if $F$ is a general fibre of $f$ and if we let $K_F+B_F=(K_X+B)|_F$, then $K_F+B_F\sim_\R 0$, 
hence $(F,B_F)$ is a log Calabi-Yau pair justifying the terminology. 

The class of log Calabi-Yau fibrations includes all log Fano and log Calabi-Yau varieties 
 and much more. For example, if $X$ is a variety which is Fano over a base $Z$, then 
we can easily find $B$ so that $(X,B)\to Z$ is a log Calabi-Yau fibration. This includes all Mori fibre spaces. 
Since we allow birational contractions, it also includes all divisorial and flipping contractions. 
Another interesting example of log Calabi-Yau fibrations $(X,B)\to Z$ is when  
$X\to Z$ is the identity morphism; the set of such fibrations simply coincides with the set of 
pairs with log canonical singularities.
On the other hand, a surface with a minimal elliptic fibration over a curve is another instance of 
a log Calabi-Yau fibration. 

Besides the classification problem mentioned above, there are other motivations for considering 
log Calabi-Yau fibrations. Indeed, they are very useful for inductive 
treatment of various problems in algebraic geometry. For example they are used to treat 
the minimal model and abundance and Iitaka conjectures, to construct complements on Fano varieties, 
etc. They appear in the literature with other names, e.g. lc-trivial fibrations [\ref{ambro-lc-trivial}].

The following are general guiding questions which are the focus of this paper:

\smallskip

{\bf Questions.} 
{\it Under what conditions do log Calabi-Yau fibrations form bounded families?} 

\smallskip

{\it How do singularities behave on the total space and base of log Calabi-Yau fibrations?}  
\smallskip

{\it When do bounded (klt or lc) complements exist for log Calabi-Yau fibrations?}

\smallskip
{\flushleft These} questions are naturally related to many problems in birational geometry. 
In this paper we investigate these questions giving particular attention to those log Calabi-Yau fibrations 
which in some sense carry full or partial Fano type structures.

A log Calabi-Yau fibration $(X,B)\to Z$ is of \emph{Fano type} if $X$ is of Fano type over $Z$, 
that is, if $-(K_X+C)$ is ample over $Z$ and $(X,C)$ is klt for some boundary $C$. When 
$(X,B)$ is klt, this is equivalent to saying that $-K_X$ is big over $Z$.
 We introduce some notation, somewhat similar to [\ref{Jiang}], 
to simplify the statements of our results below. 

\begin{defn}\label{d-FT-fib}
Let $d,r$ be natural numbers and $\epsilon$ be a positive real number. 
A \emph{$(d,r,\epsilon)$-Fano type (log Calabi-Yau) fibration} consists of a 
pair $(X,B)$ and a contraction $f\colon X\to Z$ such that we have the following:
\begin{itemize}
\item $(X,B)$ is a projective $\epsilon$-lc pair of dimension $d$,

\item $K_X+B\sim_\R f^* L$ for some $\R$-divisor $L$, 

\item $-K_X$ is big over  $Z$, i.e. $X$ is of Fano type over $Z$,

\item $A$ is a very ample divisor on $Z$ with $A^{\dim Z}\le r$, and 

\item $A-L$ is ample.
\end{itemize}
\end{defn}

That is, a $(d,r,\epsilon)$-Fano type fibration is a log Calabi-Yau fibration which is of Fano type and with certain 
geometric and numerical data bounded by the numbers $d,r,\epsilon$. The condition $A^{\dim Z}\le r$ 
means that $Z$ belongs to a bounded family of varieties. Ampleness of $A-L$ means that 
the ``degree" of $K_X+B$ is in some sense bounded (this degree is measured with respect to $A$).
When $Z$ is a point the last two conditions in the definition are vacuous: in this case the fibration is 
simply a Fano type $\epsilon$-lc  Calabi-Yau pair of dimension $d$.

In the rest of this introduction we will state some of the main results of this paper. To keep the introduction 
as simple as possible we have moved further results and remarks to Section 2.\\

{\textbf{\sffamily{Boundedness of  log Calabi-Yau fibrations with Fano type structure.}}}
Our first result concerns the boundedness of Fano type fibrations as defined above. 
This  maybe considered as a relative version 
of the so-called BAB conjecture [\ref{B-BAB}, Theorem 1.1] which is about boundedness of Fano varieties 
in the global setting. 

\begin{thm}\label{t-bnd-cy-fib}
Let $d,r$ be natural numbers and $\epsilon$ be a positive real number. Consider the set of all
$(d,r,\epsilon)$-Fano type fibrations $(X,B)\to Z$ as in \ref{d-FT-fib}. 
Then the $X$ form a bounded family.
\end{thm}

The theorem also holds in the 
more general setting of generalised pairs, see \ref{t-log-bnd-cy-gen-fib}. 
A key ingredient of the proof is the theory of complements. Indeed 
in the process of proving the theorem we show that there is $\Lambda\ge 0$ 
such that $(X,\Lambda)$ is klt, $K_X+\Lambda\sim_\Q 0/Z$ having bounded Cartier index, 
and $(X,\Lambda)$ is log bounded. 

Jiang [\ref{Jiang}, Theorem 1.4] considers the setting of the theorem and proves birational boundedness 
of $X$ modulo several conjectures. We use some of his arguments to get the birational boundedness 
 but we need to do a lot more work to get boundedness.

The boundedness statement of Theorem \ref{t-bnd-cy-fib} does not say anything about boundedness of 
$\Supp B$. This is because in general we have no control over $\Supp B$, e.g.  
when $X=\PP^2$ and $Z$ is a point, $\Supp B$ can contain arbitrary 
hypersufaces. However, if the coefficients of $B$ are bounded away from zero, then indeed 
$\Supp B$ would also be bounded. More generally we have: 
 
\begin{thm}\label{t-log-bnd-cy-fib}
Let $d,r$ be natural numbers and $\epsilon,\delta$ be positive real numbers. Consider the set of all
$(d,r,\epsilon)$-Fano type fibrations $(X,B)\to Z$ as in \ref{d-FT-fib} and $\R$-divisors   
$0\le \Delta\le B$ whose non-zero coefficients are $\ge \delta$.
Then the set of such $(X,\Delta)$ is log bounded.
\end{thm}

For applications it is important to consider variants of the above results by 
replacing the Fano type assumption with a more flexible notion. We say that a contraction $X\to Z$ of normal varieties 
\emph{factors as a tower of Fano fibrations of length $l$} if $X\to Z$ factors as a sequence of contractions 
$$
X=X_1\to  X_2 \to \cdots \to X_l=Z 
$$ 
where  $-K_{X_i}$ is ample over $X_{i+1}$, for each $1\le i\le l-1$. For practical convenience we allow 
$X_i\to X_{i+1}$ to be an isomorphism and allow $\dim X_i=0$, so $l$ is not uniquely determined by $X\to Z$.
The next result replaces Fano type with existence of a tower of Fano fibrations. 
It will be a crucial ingredient of the proof of \ref{cor-bnd-cy-fib-non-product} below.

\begin{thm}\label{t-towers-of-Fanos}
Let $d,r,l$ be natural numbers and $\epsilon,\tau$ be positive real numbers. 
Consider pairs $(X,B)$ and contractions $f\colon X\to Z$ such that 
\begin{itemize}
\item $(X,B)$ is projective $\epsilon$-lc of dimension $d$,

\item the non-zero coefficients of $B$ are $\ge \tau$,

\item $K_X+B\sim_\R f^*L$ for some $\R$-divisor $L$,

\item $X\to Z$ factors as a tower of Fano fibrations of length $l$,

\item there is a very ample divisor $A$ on $Z$ with $A^{\dim Z}\le r$, and 

\item $A-L$ is ample. 

\end{itemize}
Then the set of such $(X,B)$ forms a log bounded family.
\end{thm}

Special cases of this are proved in [\ref{DiCerbo-Svaldi}, 1.9 and 1.10].\\

{\textbf{\sffamily{Boundedness of crepant models.}}}
It is interesting to look at the special cases of Theorem \ref{t-bnd-cy-fib}  when $f$ is birational 
and when it is constant. In the latter case, the theorem is equivalent to the 
BAB conjecture [\ref{B-BAB}, Theorem 1.1 and Corollary 1.2] but 
in the former case, which says something about \emph{crepant models}, a lot work is needed to 
derive it from the BAB. 

Given a pair $(Z,B_Z)$ and a birational contraction $\phi \colon X\to Z$, we can write  
$K_X+B=\phi^*(K_Z+B_Z)$ for some uniquely determined $B$. We say $(X,B)$ is a 
crepant model of $(Z,B_Z)$ if $B\ge 0$. The birational case of Theorem \ref{t-bnd-cy-fib} 
then essentially says that if $Z$ belongs to a bounded family, if the ``degree" of $B_Z$ is bounded 
with respect to some very ample divisor, and if $(Z,B_Z)$ is $\epsilon$-lc, then   
the underlying varieties of all the crepant models of such pairs form a bounded family; 
this is quite non-trivial even in the case $Z=\PP^3$ (actually it is already challenging for $Z=\PP^2$ 
if we do not use BAB). 
Special cases of 
boundedness of crepant models have appeared in the literature assuming that $\Supp B_Z$ is bounded, see 
[\ref{Mc-Pr}, Lemma 10.5][\ref{HX}, Propositions 2.5, 2.9][\ref{CDHJS}, Proposition 4.8]. The key 
point here is that we remove such assumptions on the support of $B_Z$.  

Note that the $\epsilon$-lc condition and 
boundedness of  ``degree" of $B_Z$ are both necessary. Indeed if we replace $\epsilon$-lc by 
lc, then the crepant models will not be bounded, e.g. 
considering $(Z=\PP^2,B_Z)$ where $B_Z$ is the union of three lines intersecting transversally and  
successively blowing up intersection points in the boundary gives an infinite sequence of 
crepant models with no bound on their Picard number. On the other hand, if 
$B_Z$ can have arbitrary degree, then we can easily choose it so that  
$(Z=\PP^2,B_Z)$ is $\frac{1}{2}$-lc having a crepant model of arbitrarily large Picard number.\\

{\textbf{\sffamily{Boundedness of log Calabi-Yau pairs.}}}
Without appropriate restrictions, the set of log Calabi-Yau pairs of a fixed dimension is far from being bounded. 
For example, it is well-known that the set of K3 surfaces is not bounded although they are topologically 
bounded. In this respect log Calabi-Yau pairs with non-zero 
boundary behave better as the next result illustrates. 

\begin{thm}\label{cor-bnd-cy-fib-non-product}
Let $d$ be a natural number and $\epsilon,\tau$ be positive real numbers. 
Consider pairs $(X,B)$ with the following properties: 
\begin{itemize}
\item $(X,B)$ is projective $\epsilon$-lc of dimension $d$,

\item $K_X+B\sim_\R 0$, 

\item $B\neq 0$ and its coefficients are $\ge \tau$, and

\item $(X,B)$ is not of product type.   
\end{itemize}
Then the set of such $(X,B)$ is log bounded up to isomorphism in codimension one.
\end{thm}

Boundedness up to isomorphism in codimension one means that there is a bounded family of projective 
varieties $Y$ such that for each $X$ in the theorem we can find some $Y$ together with a birational map $X\bir Y$ 
which is an isomorphism in codimension one; a similar definition applies to the case of pairs. 

Di Cerbo and Svaldi [\ref{DiCerbo-Svaldi}, Theorem 1.3] studied this statement and proved it in dimension $\le 4$ 
when the coefficients of $B$ belong to a fixed DCC set; in this case we can replace $\epsilon$-lc with klt 
as $\epsilon$-lc would then follow from the other assumptions. 
They use MMP to obtain a tower of Mori fibre spaces on a birational model of $X$ (see \ref{t-tower-of-Mfs}) 
and use this tower to prove the claimed boundedness. We will use this strategy and 
\ref{t-log-bnd-cy-fib} and \ref{t-towers-of-Fanos} to prove the theorem.

The condition of $(X,B)$ not being of product type means that if there is a birational map 
$\phi\colon X\bir Y$ to a normal projective 
variety whose inverse does not contract divisors and if $g\colon Y\to Z$ is a contraction with 
$\dim Y>\dim Z>0$, then $K_Z \not \equiv 0$. 
This is a technical condition related to moduli and Hodge 
theory but without it the statement fails. We will see that $K_Y+B_Y$ is numerically equivalent to 
$g^*(K_Z+B_Z+M_Z)$ where $B_Z,M_Z$ are the discriminant and moduli divisors of adjunction, 
and the condition says that $B_Z+M_Z$ is not numerically trivial.

When $X$ is of Fano type the theorem was already known [\ref{B-compl}, Theorem 1.4] (or [\ref{HX}] 
when the coefficients of $B$ are in a fixed DCC set) in which case  
we can remove the non-product type assumption as it follows from the Fano type property. 
But the main point of the theorem is that 
we have replaced Fano type with the weaker property of not being of product type.
The property allows one to reduce the theorem  
to boundedness of certain towers of Mori fibre spaces which turns out to be a 
special case of Theorem \ref{t-towers-of-Fanos}. It it worth mentioning that the theorem also holds in 
the relative setting similar to \ref{t-towers-of-Fanos} but for simplicity we prove the above version.\\

{\textbf{\sffamily{Boundedness of singularities on log Calabi-Yau fibrations.}}}
Understanding singularities on log Calabi-Yau fibrations 
is very important as it naturally appears in inductive arguments. 
The next statement gives a lower bound for lc thresholds on Fano type fibrations. 
In particular, it implies [\ref{Jiang}, Conjecture 1.13] as a special case. 

\begin{thm}\label{t-sing-FT-fib-totalspace}
Let $d,r$ be natural numbers and $\epsilon$ be a positive real number. 
 Then there is a positive real number $t$ depending only on $d,r,\epsilon$ satisfying the following. 
Let $(X,B)\to Z$ be any $(d,r,\epsilon)$-Fano type fibration as in \ref{d-FT-fib}. If 
$P\ge 0$ is any $\R$-Cartier divisor on $X$ such that either  
\begin{itemize}
\item $f^*A+B-P$ is pseudo-effective, or 

\item $f^*A-K_X-P$ is pseudo-effective, 
\end{itemize}
then  
$(X,B+t P)$ is klt.
\end{thm}

In particular, the theorem can be applied to any $0\le P\sim_\R f^*A+B$ or any $0\le P\sim_\R f^*A-K_X$ 
assuming $P$ is $\R$-Cartier. To get a feeling for what the theorem says consider the case 
when $Z$ is a curve; in this case the theorem implies that the multiplicities of each fibre of $f$ 
are bounded from above (compare with the main result of [\ref{Mori-Prokhorov-del-pezzo}] 
for del Pezzo fibrations over curves): 
indeed, for any closed point $z\in Z$ we can find $0\le Q\sim A$ so that 
$z$ is a component of $Q$; then applying the theorem to $P:=f^*Q$ implies that the multiplicities of 
the fibre of $f$ over $z$ are bounded.

One can derive the theorem from \ref{t-log-bnd-cy-fib} and the results of [\ref{B-BAB}]. 
However, in practice the theorem is proved together along with \ref{t-log-bnd-cy-fib} in an intertwining inductive process.\\

On the other hand, a fundamental problem on singularities is a conjecture of 
Shokurov [\ref{B-sing-fano-fib}, Conjecture 1.2] (a special case of which is due to M$^{\rm c}$Kernan) 
which roughly says that the singularities on 
the base of a Fano type fibration are controlled by those on the total space.
We will prove Shokurov's conjecture under some boundedness assumptions on the 
base (\ref{t-sh-conj-bnd-base}) and prove a generalisation of the conjecture under some 
boundedness assumptions on the general fibres (\ref{t-cb-conj-sing-bnd-fib}). 
These results are of independent interest but also closely related to the other results of this paper. 

First we recall adjuction for fibrations also known as canonical bundle formula. If $(X,B)$ is an lc pair
and $f\colon X\to Z$ is a contraction with $K_X+B\sim_\R 0/Z$, then 
by [\ref{kaw-subadjuntion}][\ref{ambro-adj}] we can define a 
discriminant divisor $B_Z$ and a moduli divisor $M_Z$ so that we have 
$$
K_X+B\sim_\R f^*(K_Z+B_Z+M_Z).
$$
This is a generalisation of the Kodaira canonical bundle formula. 
Let $D$ be a prime divisor on $Z$. Let $t$ be the lc threshold of $f^*D$ with respect to $(X,B)$ 
over the generic point of $D$. 
We then put the coefficient of  $D$ in $B_Z$ to be $1-t$. 
Having defined $B_Z$, we can find $M_Z$ giving 
$$
K_{X}+B+M\sim_\R f^*(K_Z+B_Z+M_Z)
$$
where $M_Z$ is determined up to $\R$-linear equivalence. 
We call $B_Z$ the \emph{discriminant divisor of adjunction} for $(X,B)$ over $Z$. 

For any birational contraction $Z'\to Z$  from a normal variety there is 
a birational contraction $X'\to X$ from a normal variety so that the induced map $X'\bir Z'$ is a morphism.  
Let $K_{X'}+B'$ be the 
pullback of $K_{X}+B$. We can similarly define $B_{Z'},M_{Z'}$ for $(X',B')$ over $Z'$. In this way we 
get the \emph{discriminant b-divisor ${\bf{B}}_Z$ of adjunction} for $(X,B)$ over $Z$. 

The conjecture of Shokurov then can be stated as:

\begin{conj}\label{conj-sh-sing-fib}
Let $d$ be a natural number and $\epsilon$ be a positive real number. 
 Then there is a positive real number $\delta$ depending only on $d,\epsilon$ satisfying the following. 
 Assume that $(X,B)$ is a pair and $f\colon X\to Z$ is a contraction such that
\begin{itemize}
\item $(X,B)$ is $\epsilon$-lc of dimension $d$,

\item $K_X+B\sim_\R 0/Z$, and  

\item $-K_X$ is big over  $Z$.
\end{itemize}
Then the discriminant b-divisor ${\bf B}_Z$ has coefficients in $(-\infty, 1-\delta]$.
\end{conj}

The next result says that Shokurov conjecture holds in the setting of Fano type fibrations.
The strength of this result is in the fact that (similar to some of the other results above, e.g. \ref{t-bnd-cy-fib}) we are 
not assuming any boundedness condition on support of $B$ along the general fibres of $f$.

\begin{thm}\label{t-sh-conj-bnd-base}
Let $d,r$ be natural numbers and $\epsilon$ be a positive real number. 
 Then there is a positive real number $\delta$ depending only on $d,r,\epsilon$ satisfying the following. 
Let $(X,B)\to Z$ be any  $(d,r,\epsilon)$-Fano type fibration as in \ref{d-FT-fib}. 
Then  the discriminant b-divisor ${\bf B}_Z$ has coefficients in $(-\infty, 1-\delta]$.
\end{thm}

This can be viewed as a relative version of Ambro's conjecture [\ref{B-BAB}, Theorem 1.4] which is closely related 
to the BAB conjecture.

Next we prove a variant of Conjecture \ref{conj-sh-sing-fib} which is weaker in the sense that we assume 
certain boundedness along the general fibres but it is stronger in the sense that we replace 
bigness of $-K_X$ over $Z$ with a less restrictive condition. This is important for applications, eg 
proofs of \ref{t-towers-of-Fanos} and \ref{cor-bnd-cy-fib-non-product}. 

\begin{thm}\label{t-cb-conj-sing-bnd-fib}
Let $d,v$ be natural numbers and $\epsilon$ be a positive real number. 
 Then there is a positive real number $\delta$ depending only on $d,v,\epsilon$ satisfying the following. 
 Assume that $(X,B)$ is a pair and $f\colon X\to Z$ is a contraction such that
\begin{itemize}
\item $(X,B)$ is $\epsilon$-lc of dimension $d$,

\item $K_X+B\sim_\R 0/Z$, and

\item there is an integral divisor $G\ge 0$ with 
$$
0<\vol((\Supp B+G)|_F)<v
$$ 
for the general fibres $F$ of $f$.
\end{itemize}
Then the discriminant b-divisor ${\bf B}_Z$ has coefficients in $(-\infty, 1-\delta]$.
\end{thm}

Consider a minimal elliptic surface, that is, a smooth projective surface $X$ endowed with an elliptic 
fibration $f\colon X\to Z$, that is, a contraction with $K_X\sim_\Q 0/Z$ (the general fibres are then elliptic curves). 
In general, $f$ can have singular fibres of arbitrarily large multiplicity (cf. [\ref{PSh-II}, Example 7.17]), 
hence the discriminant divisor $B_Z$ can have coefficients arbitrarily close to $1$. Here $(X,B)\to Z$ 
satisfies all the assumptions of the theorem with $B=0$ and $\epsilon=1$ except the last condition 
involving $G$. This illustrates the role of $G$ in the theorem. Indeed, if we additionally assume that 
$f$ has a multi-section of fixed degree $l$ (that is, if there is a horizontal$/Z$ curve $G$ with degree 
of $G\to Z$ being $l$), then the discriminant divisor $B_Z$ would have 
coefficients bounded away from $1$.\\

{\textbf{\sffamily{Boundedness of relative-global complements.}}}
One of the key tools used in this paper is the theory of complements. 
We need to prove a more general version of the bounded ``complements" 
statement of [\ref{B-BAB}, Theorem 1.7]. Most likely this will be useful elsewhere.
These complements are relative but are controlled globally.

\begin{thm}\label{t-bnd-comp-lc-global}
Let $d$ be a natural number and 
$\mathfrak{R}\subset [0,1]$ be a finite set of rational numbers. Then there exists a natural number 
$n$ depending only on $d,\mathfrak{R}$ satisfying the following. Assume 

\begin{itemize}
\item $(X,B)$ is a projective lc pair of dimension $d$,

\item the coefficients of $B$ are in $\mathfrak{R}$, 

\item $M$ is a semi-ample Cartier divisor on $X$ defining a contraction $f\colon X\to Z$,

\item $X$ is of Fano type over $Z$, 

\item $M-(K_X+B)$ is nef and big, and 

\item $S$ is a non-klt centre of $(X,B)$ with $M|_{S}\equiv 0$.\\ 

\end{itemize}
Then there is a $\Q$-divisor $\Lambda\ge B$ such that  
\begin{itemize}
\item $(X,\Lambda)$ is lc over a neighbourhood of $z:=f(S)$, and

\item $n(K_X+\Lambda)\sim (n+2)M$.
\end{itemize}
\end{thm}

We prove a stronger statement when singularities are milder which is proved along with the 
other results above.

\begin{thm}\label{t-bnd-comp-lc-global-cy-fib}
Let $d,r$ be natural numbers, $\epsilon$ be a positive real number,  and 
$\mathfrak{R}\subset [0,1]$ be a finite set of rational numbers. Then there exists natural numbers 
$n,m$ depending only on $d,r,\epsilon,\mathfrak{R}$ satisfying the following. 
Assume that $(X,B)\to Z$ is a $(d,r,\epsilon)$-Fano type fibration and that 
\begin{itemize}
\item  we have  $0\le \Delta\le B$ with coefficients in $\mathfrak{R}$, and 

\item $-(K_X+\Delta)$ is big over $Z$. 
\end{itemize}
Then there is a $\Q$-divisor $\Lambda\ge \Delta$ such that  
\begin{itemize}
\item $(X,\Lambda)$ is klt, and 

\item $n(K_X+\Lambda)\sim mf^*A$.
\end{itemize}
\end{thm}

For example we can apply the theorem under the assumptions of \ref{t-bnd-cy-fib} by taking 
$\Delta=0$.\\

{\textbf{\sffamily{Announcement of results of a sequel paper.}}
Here we state several results whose proofs will appear in a sequel paper joint with Di Cerbo and Svaldi. 
For simplicity we do not state them in their most general form.

\begin{thm}
Let $d,p$ be natural numbers. 
Consider pairs $(X,B)$ with the following properties: 
\begin{itemize}
\item $(X,B)$ is projective klt of dimension $d$,

\item $p(K_X+B)\sim 0$, and

\item $X$ is rationally connected.   
\end{itemize}
Then the set of such $(X,B)$ is log bounded up to isomorphism in codimension one.
\end{thm}

In low dimension we have a stronger statement. 

\begin{thm}
Let $\epsilon$ be a positive real number. 
Consider pairs $(X,B)$ with the following properties: 
\begin{itemize}
\item $(X,B)$ is projective $\epsilon$-lc of dimension $3$,

\item $-(K_X+B)$ is nef, 

\item $K_X\not\equiv 0$, and

\item $X$ is rationally connected.   
\end{itemize}
Then the set of such $X$ is bounded up to isomorphism in codimension one.
\end{thm}

The above results have applications to Calabi-Yau varieties with elliptic fibrations.

\begin{thm}
Let $d,p$ be natural numbers. 
Consider varieties $X$ with the following properties: 
\begin{itemize}
\item $X$ is projective klt of dimension $d$,

\item $pK_X\sim 0$, 

\item there is an elliptic fibration $X\to Y$ admitting a rational section, and 

\item $Y$ is rationally connected.   
\end{itemize}
Then the set of such $X$ is bounded up to isomorphism in codimension one.
\end{thm}

This generalises [\ref{DiCerbo-Svaldi}, Theorem 1.1] to every dimension.

\bigskip


{\textbf{\sffamily{Acknowledgements.}}
This work was supported by a grant of the Leverhulme Trust. Part of it was done while  visiting the 
University of Tokyo in March 2018 which was arranged by Yujiro Kawamata and Yoshinori Gongyo, 
and I would like to thank them for their hospitality. Thanks to Christopher Hacon for answering our 
questions regarding moduli part of adjunction.  Thanks to Yifei Chen for numerous valuable comments.
Thanks to Jingjun Han, Roberto Svaldi, and Yanning Xu for useful comments on an earlier version of this paper.

\section{\bf Further results and remarks}

In this section we state further results and remarks working mostly in the context of generalised pairs.\\

{\textbf{\sffamily{A framework for classification of Fano fibrations.}}
Here we illustrate how the results above can be used towards classification of 
Fano fibrations such as Mori fibre spaces  
in the context of birational classification of algebraic varieties.
Suppose that we are given a normal projective variety $Z$ with a very ample divisor $A$ on it. The aim is 
to somehow classify a given set of Fano fibrations over $Z$. We naturally want to fix or bound 
certain invariants. Let $d$ be a natural number and $\epsilon$ be a positive real number. 
Assume $\mathcal{P}$ is a set of contractions $f\colon X\to Z$ such that 
\begin{itemize}
\item $X$ is projective of dimension $d$ with $\epsilon$-lc singularities, and 

\item $-K_X$ is ample over $Z$. 
\end{itemize}
For each non-negative integer $l$, let $\mathcal{P}_l$ be the set of all $X \to Z$ in $\mathcal{P}$ 
such that 
$$
l=\min\{a\in \Z^{\ge 0} \mid \mbox{$af^*A-K_X$ is ample}\}.
$$ 
For example, $X\to Z\in \mathcal{P}_0$ means that $-K_X$ is ample, hence $X$ is globally a Fano variety; 
 $X\to Z\in \mathcal{P}_1$ means that $-K_X$ is not ample but $f^*A-K_X$ is ample. Thus we have a disjoint union 
$$
\mathcal{P}=\bigcup_{l\in \Z^{\ge 0}} \mathcal{P}_l.
$$
For each $X\to Z$ in $\mathcal{P}_l$, we can choose a general $0\le B\sim_\Q lf^*A-K_X$, so that  
$(X,B)\to Z$ is a 
$$
\mbox{$(d,((l+1)A)^{\dim Z},{\epsilon})$-Fano type fibration} 
$$
perhaps after a slight decrease of $\epsilon$ (we need to decrease only if $\epsilon=1$).
Thus the set of such  
$X$ forms a bounded family, by Theorem \ref{t-bnd-cy-fib}. That is, we can write $\mathcal{P}$ as a disjoint union of 
bounded sets. The next step is to study each set $\mathcal{P}_l$ more closely to get a finer classification.

Lets look at the simplest non-trivial case of surfaces, that is, consider the set $\mathcal{P}$ of 
Mori fibre spaces $X\to Z=\PP^1$ where $d=\dim X=2$ and $X$ is smooth.  In this case it is well-known that 
$\mathcal{P}$ coincides with the sequence of Hirzebruch surfaces, that is, $\PP^1$-bundles $f_i\colon X_i\to Z$ having a section 
$E_i$ satisfying $E_i^2=-i$, for $i=0,1,\dots$. Applying the divisorial 
adjunction formula gives $K_{X_i}\cdot E_i=i-2$. 
Letting $A$ be a point on $Z$ and using the fact that the Picard group of $X_i$ is generated by $E_i$ and 
a fibre of $f_i$, it is easy to check that $X_0$ and $X_1$ are Fano, and $(i-1)f_i^*A-K_{X_i}$ is ample
but $(i-2)f_i^*A-K_{X_i}$ is not ample for $i\ge 2$. Therefore,  under the notation introduced above, we have 
$$
\mathcal{P}_0=\{X_0\to Z,X_1\to Z\}, ~~\mbox{and}~~ \mathcal{P}_l=\{X_{l+1}\to Z\}~~ \mbox{for}~~ l\ge 1.
$$
Of course we have used the classification of ruled surfaces over $\PP^1$ but the point we want to make is that 
conversely studying $\mathcal{P}$ and each subset $\mathcal{P}_l$ closely will naturally lead us  
to the classification of ruled surfaces over $\PP^1$.\\

{\textbf{\sffamily{Boundedness of generalised Fano type fibrations.}}}
We will prove some of the results stated above in the context of generalised pairs.
For the basics of generalised pairs see [\ref{BZh}] and [\ref{B-compl}] and the preliminaries below. 
A \emph{generalised log Calabi-Yau fibration} consists of a generalised pair $(X,B+M)$ 
with generalised lc singularities and a contraction $X\to Z$ such that $K_X+B+M\sim_\R 0/Z$. 
We define generalised Fano type fibrations  similar to \ref{d-FT-fib}. 

\begin{defn}\label{d-gen-FT-fib}
Let $d,r$ be natural numbers and $\epsilon$ be a positive real number. 
A generalised $(d,r,\epsilon)$-Fano type (log Calabi-Yau) fibration consists of a projective generalised pair 
$(X,B+M)$ with data $X'\to X$ and $M'$, and a contraction $f\colon X\to Z$ such that we have: 
\begin{itemize}
\item $(X,B+M)$ is generalised $\epsilon$-lc of dimension $d$,

\item $K_X+B+M\sim_\R f^* L$ for some $\R$-divisor $L$,

\item $-K_X$ is big over  $Z$, i.e. $X$ is of Fano type over $Z$,

\item $A$ is a very ample divisor on $Z$ with $A^{\dim Z}\le r$, and 

\item $A-L$ is ample.   
\end{itemize}
\end{defn}
Note that $M'$ is assumed to be nef globally. We usually write $(X,B+M)\to Z$ to denote the fibration.
Theorem \ref{t-log-bnd-cy-fib} can be extended to the case of generalised pairs, that is:

\begin{thm}\label{t-log-bnd-cy-gen-fib}
Let $d,r$ be natural numbers and $\epsilon,\tau$ be positive real numbers. Consider the set of all
generalised $(d,r,\epsilon)$-Fano type fibrations $(X,B+M)\to Z$ such that  
\begin{itemize}
\item we have  $0\le \Delta\le B$ whose non-zero coefficients are $\ge \tau$, and

\item $-(K_X+\Delta)$ is big over $Z$. 
\end{itemize} 
Then the set of such $(X,\Delta)$ is log bounded.
\end{thm}\

{\textbf{\sffamily{Singularities on generalised log Calabi-Yau fibrations.}}}
Theorem \ref{t-sing-FT-fib-totalspace} also holds for generalised pairs, that is: 

\begin{thm}\label{t-sing-gen-FT-fib-totalspace}
Let $d,r$ be natural numbers and $\epsilon$ be a positive real number. 
 Then there is a positive real number $t$ depending only on $d,r,\epsilon$ satisfying the following. 
Let $(X,B+M)\to Z$ be any generalised $(d,r,\epsilon)$-Fano type fibration as in \ref{d-gen-FT-fib}. If 
$P\ge 0$ is any $\R$-Cartier divisor on $X$ such that either  
\begin{itemize}
\item $f^*A+B+M-P$ is pseudo-effective, or 

\item $f^*A-K_X-P$ is pseudo-effective, 
\end{itemize}
then  
$(X,B+t P+M)$ is generalised klt.
\end{thm}

In particular, the theorem can be applied to any $0\le P\sim_\R f^*A+B+M$ or any $0\le P\sim_\R f^*A-K_X$ 
assuming $P$ is $\R$-Cartier. 

Adjunction for fibrations also makes sense for generalised pairs. That is, if $(X,B+M)\to Z$ is a 
generalised log Calabi-Yau pair, then we can define a discriminant divisor $B_Z$ and a moduli divisor 
$M_Z$ giving 
$$
K_X+B+M\sim_\R f^*(K_Z+B_Z+M_Z).
$$
Moreover, for any birational contraction $Z'\to Z$ from a normal variety we can define the discriminant divisor $B_{Z'}$ 
whose pushdown to $Z$ is just $B_Z$. In this way we get the discriminant b-divisor  ${\bf B}_Z$.
See \ref{fib-adj-setup} for more details.

Now we state a generalised version of Shokurov's conjecture \ref{conj-sh-sing-fib}.

\begin{conj}\label{conj-sh-sing-gen-fib}
Let $d$ be a natural number and $\epsilon$ be a positive real number. 
 Then there is a positive real number $\delta$ depending only on $d,\epsilon$ satisfying the following. 
 Assume that $(X,B+M)$ is a generalised pair with data $X'\to X\overset{f}\to Z$ and $M'$ 
where $f$ is a contraction such that
\begin{itemize}
\item $(X,B+M)$ is generalised $\epsilon$-lc of dimension $d$,

\item $K_X+B+M\sim_\R 0/Z$, and  

\item $-K_X$ is big over $Z$.
\end{itemize}
Then the discriminant b-divisor ${\bf B}_Z$ has coefficients in $(-\infty, 1-\delta]$.
\end{conj}

Note that in particular we are assuming that $M'$ is nef over $Z$ as this is part of the definition of a generalised pair. 
The next result says that \ref{conj-sh-sing-gen-fib} holds in the setting of generalised Fano type fibrations.

\begin{thm}\label{t-sh-conj-bnd-base-gen-fib}
Let $d,r$ be natural numbers and $\epsilon$ be a positive real number. 
 Then there is a positive real number $\delta$ depending only on $d,r,\epsilon$ satisfying the following. 
Let $(X,B+M)\to Z$ be any generalised $(d,r,\epsilon)$-Fano type fibration as in \ref{d-gen-FT-fib}. 
Then  the discriminant b-divisor ${\bf B}_Z$ has coefficients in $(-\infty, 1-\delta]$.
\end{thm}

Now we propose a conjecture which is stronger than \ref{conj-sh-sing-gen-fib} in the sense that we replace 
the bigness of $-K_X$ over $Z$ with a weaker condition. 

\begin{conj}\label{conj-cb-sing-gen-fib}
Let $d,v$ be natural numbers and $\epsilon$ be a positive real number. 
 Then there is a positive real number $\delta$ depending only on $d,v,\epsilon$ satisfying the following. 
 Assume that $(X,B+M)$ is a generalised pair with data $X'\to X\overset{f}\to Z$ and $M'$ 
where $f$ is a contraction such that
\begin{itemize}
\item $(X,B+M)$ is generalised $\epsilon$-lc of dimension $d$,

\item $K_X+B+M\sim_\R 0/Z$, and  

\item there is an integral divisor $G\ge 0$ on $X$ with 
$$
0<\vol((B+M+G)|_F)<v
$$ 
for the general fibres $F$ of $f$.
\end{itemize}
Then the discriminant b-divisor ${\bf B}_Z$ has coefficients in $(-\infty, 1-\delta]$.
\end{conj}

When $-K_X$ is big over $Z$, the general fibres $F$ 
belong to a bounded family by [\ref{B-BAB}], so $\vol((B+M)|_F)=\vol(-K_X|_F)$ 
is positive and bounded from above, hence in this case 
we can take $G=0$. That is, \ref{conj-sh-sing-gen-fib} is a special case of \ref{conj-cb-sing-gen-fib}.

The next statement says that Conjecture \ref{conj-cb-sing-gen-fib} holds if we put some boundedness assumptions on the 
general fibres. 

\begin{thm}\label{t-cb-conj-sing-bnd-gen-fib}
Let $d,v,p$ be natural numbers and $\epsilon$ be a positive real number. 
 Then there is a positive real number $\delta$ depending only on $d,v,p,\epsilon$ satisfying the following. 
 Assume that $(X,B+M)$ is a generalised pair with data $X'\to X\overset{f}\to Z$ and $M'$ 
where $f$ is a contraction such that
\begin{itemize}
\item $(X,B+M)$ is generalised $\epsilon$-lc of dimension $d$,

\item $K_X+B+M\sim_\R 0/Z$,

\item there is an integral divisor $G\ge 0$ on $X$ with 
$$
0<\vol(((\Supp B)+M+G)|_F)<v
$$ 
for the  general fibres $F$ of $f$, and 

\item $pM'$ is b-Cartier.
\end{itemize}
Then the discriminant b-divisor ${\bf B}_Z$ has coefficients in $(-\infty, 1-\delta]$.
\end{thm}

The b-Cartier condition of $pM'$ means that the pullback of $pM'$ to some resolution of $X$ 
is a Cartier divisor.

\begin{cor}\label{cor-cb-conj-sing-bnd-gen-fib}
Let $p$ be a natural number and $\tau$ be a positive real number. Then Conjectures 
\ref{conj-sh-sing-gen-fib} and \ref{conj-cb-sing-gen-fib} hold for those $(X,B+M)$ which in addition satisfy:
\begin{itemize}
\item any horizontal$/Z$ component of $B$ has coefficient $\ge \tau$, and 

\item $pM'$ is b-Cartier.
\end{itemize}
\end{cor}

Note that  we allow the case when $B$ has no horizontal$/Z$ components.\\

{\textbf{\sffamily{Plan of the paper.}}}
We will prove Theorem \ref{t-bnd-comp-lc-global} in Section 4, Theorems \ref{t-bnd-cy-fib}, \ref{t-log-bnd-cy-fib}, 
\ref{t-log-bnd-cy-gen-fib}, \ref{t-sing-FT-fib-totalspace},  \ref{t-bnd-comp-lc-global-cy-fib} in Section 5, and 
Theorems \ref{t-sh-conj-bnd-base}, \ref{t-cb-conj-sing-bnd-fib}, \ref{t-sing-gen-FT-fib-totalspace}, 
\ref{t-sh-conj-bnd-base-gen-fib}, \ref{t-cb-conj-sing-bnd-gen-fib} and 
Corollary \ref{cor-cb-conj-sing-bnd-gen-fib} in Section 6, and Theorems \ref{t-towers-of-Fanos} 
and \ref{cor-bnd-cy-fib-non-product} in Section 7.

\section{\bf Preliminaries}

All the varieties in this paper are quasi-projective over a fixed algebraically closed field $k$ of characteristic zero
unless stated otherwise. 

\subsection{Numbers}

Let $\mathfrak{R}$ be a subset of $[0,1]$. Following [\ref{PSh-II}, 3.2] we define 
$$
\Phi(\mathfrak{R})=\left\{1-\frac{r}{m} \mid r\in \mathfrak{R},~ m\in \N\right\}
$$
to be the set of \emph{hyperstandard multiplicities} associated to $\mathfrak{R}$.

\subsection{Contractions}

By a \emph{contraction} we mean a projective morphism $f\colon X\to Y$ of varieties 
such that $f_*\mathcal{O}_X=\mathcal{O}_Y$ ($f$ is not necessarily birational). In particular, 
$f$ is surjective and has connected fibres.

\subsection{Divisors}\label{ss-divisors}

Let $X$ be a variety. If $D$ is a prime divisor on birational models of $X$ whose centre on $X$ is non-empty, 
then we say $D$ is a prime divisor \emph{over} $X$. If $X$ is normal and $M$ is an $\R$-divisor on $X$, 
we let 
$$
|M|_\R=\{N \ge 0\mid N\sim_\R M\}.
$$
Recall that $N\sim_\R M$ means that 
$N-M=\sum r_i\Div(\alpha_i)$ for certain real numbers $r_i$ and rational functions $\alpha_i$. 
When all the $r_i$ can be chosen to be rational numbers, then we write $N\sim_\Q M$.
We define $|M|_\Q$ similarly by replacing $\sim_\R$ with $\sim_\Q$. 

Assume $\rho\colon X\bir Y/Z$ is a rational map of normal varieties projective over a base variety $Z$. 
For an $\R$-Cartier divisor $L$ on $Y$ we define the pullback $\rho^*L$ as follows. Take a 
common resolution $\phi\colon W\to X$ and $\psi\colon W\to Y$. Then let $\rho^*L:=\phi_*\psi^*L$. 
It is easy to see that this does not depend on the choice of the common resolution as 
any two such resolutions are dominated by a third one.

\begin{lem}\label{l-semi-ample-div}
Assume $Y\to X$ is a contraction of normal projective varieties, $C$ is a nef $\R$-divisor on 
$Y$ and $A$ is the pullback of an ample $\R$-divisor on $X$. If $C$ is semi-ample over $X$, then 
$C+aA$ is semi-ample (globally) for any real number $a>0$.
\end{lem}
\begin{proof}
Since $C$ is semi-ample over $X$, it defines a contraction $\phi \colon Y\to Z/X$ to a normal projective variety. Replacing $Y$ 
with $Z$ and replacing $C,A$ with $\phi_*C,\phi_*A$, respectively, we can assume $C$ is ample over $X$. Pick $a>0$. 
Now $C+bA$ is ample for some $b\gg a$ because 
$A$ is the pullback of an ample divisor on $X$. Since $C$ is globally nef, 
$$
C+tbA=(1-t)C+t(C+bA)
$$ 
is ample for any $t\in (0,1]$. In particular, taking $t=\frac{a}{b}$ we see that $C+aA$ is ample.

\end{proof}

\subsection{Linear systems}

Let $X$ be a normal projective variety and $M$ be an integral Weil divisor on $X$. 
A \emph{sub-linear system} $L\subseteq |M|$ is given by some linear subspace $V\subseteq \PP(H^0(M))$, that is, 
$$
L=\{\Div(\alpha)+M \mid \alpha \in V\}.
$$
The general members of $L$, by definition, are those $\Div(\alpha)+M$ 
where $\alpha$ is in some given non-empty open subset $W\subseteq V$ (open in the Zariski topology). 
Being a general member then depends on the choice of $W$ but we usually shrink it if necessary without notice.

\begin{lem}\label{l-gen-element-lin-system}
Let $X$ be a normal projective variety, $M$ be an integral Weil divisor on $X$, and $L\subseteq |M|$ be a sub-linear system. 
Assume that $x\in X$ is a smooth closed point and that some member of $L$ is smooth at $x$. 
Then a general member of $L$ is smooth at $x$.
\end{lem}
\begin{proof}
We can assume that every member of $L$ passes through $x$ otherwise the general members do not pass 
through $x$, hence the statement holds 
trivially. By assumption, there is $D\in  L$ such that $D$ is smooth at $x$. 
Let $H$ be a general hypersurface section of $X$ passing through $x$. Then 
$H$ and $H|_D$ are both smooth at $x$. In particular, $D|_H$ is also smooth at $x$ 
because, in a neighbourhood of $x$, both $H|_D$ and $D|_H$ considered as schemes coincide 
with the scheme-theoretic intersection $H\cap D$. 

We can choose $H$ so that it is not a component of any member of $L$ (e.g. enough to choose $H$ so that 
$H^d>H^{d-1}\cdot M$). Let $N:=L|_H$ be the restriction of $L$ to $H$, that is, $N$ consists of divisors $E|_H$ where $E\in L$. 
Although $E$ may not be $\Q$-Cartier but $E|_H$ is well-defined as a Weil divisor by our choice of $H$.
Moreover, again by our choice of $H$, the map $H^0(M)\to H^0(M|_H)$ is injective, hence 
induces a map $\PP(H^0(M))\to \PP(H^0(M|_H))$, and $N$ is given by the image of $V$ under this map.

By construction, $D|_H\in N$ is smooth at $x$. Then by induction a general member of $N$  
is smooth at $x$. This is possible only if a general member of $L$ is smooth at $x$ which can be seen 
as follows. Let $E$ be a general member of $L$ and let 
$h,e$ be the defining equations of $H,E$ near $x$. Since $H$ and $E|_H$ are smooth at $x$, we can choose a system of local 
parameters $e,t_2,\dots,t_r$ for $H$ at $x$ where $r=\dim H$. But then $h,e,t_2,\dots,t_r$ is a system of local parameters 
for $X$ at $x$, hence $E$ is smooth at $x$.

\end{proof}

\begin{lem}\label{l-gen-element-v.ample-div-passing-x,y}
Let $X$ be a normal projective variety and $A$ be a very ample divisor on $X$. 
For each pair of closed points $x,y\in X$, let $L_{x,y}$ be the sub-linear system of $|2A|$ 
consisting of the members that pass through both $x,y$. Then there is a 
non-empty open subset $U\subseteq X$ such that for any pair of closed points $x,y\in U$, a general 
member of $L_{x,y}$ is smooth at both $x,y$.
\end{lem}
\begin{proof}
By definition, a general member of $|A|$ is the element given by a section in some non-empty open subset 
$W$ of $\PP(H^0(A))$. Perhaps after shrinking $W$, we can assume that the restriction of these 
general members to the smooth locus of $X$ are smooth. 

We claim that there is a finite set $\Pi$  of closed points of $X$ (depending on $W$) such 
that for each closed point $x\in X\setminus \Pi$ 
 we can find a general member of $|A|$ passing through $x$:  indeed since $A$ is very ample, 
the set $H_z$ of elements of $\PP(H^0(A))$ vanishing at a given closed point $z$ is a hyperplane, 
and for distinct points $z,z'$ we have $H_z\neq H_{z'}$; but there are at most finitely many $z$ 
with $W\cap H_z=\emptyset$ because the complement of $W$ in $\PP(H^0(A))$ is a proper closed set; 
hence for any closed point $x$ other than those finite set we can 
find an element of $W$ vanishing at $x$ which proves the claim.  
Thus there is a non-empty open subset $U$ of the smooth locus of $X$ 
such that  for each closed point $x\in U$ we can find a 
member of $|A|$ passing through $x$ which is smooth at $x$.

Now pick a closed point $x\in U$ and let $L_x$ be the sub-linear system of $|A|$ consisting of 
members passing through $x$. By the above arguments some member of $L_x$ is smooth at $x$, 
hence a general member of 
$L_x$ is also smooth at $x$, by Lemma \ref{l-gen-element-lin-system}. 
In particular, for any other closed point $y\in U$, we can pick a member of  
$L_x$ smooth at $x$ but not containing $y$.

Now let $x,y\in U$ be a pair of distinct closed points and let $L_{x,y}$ be the sub-linear system 
of $|2A|$ consisting of members passing through $x,y$. 
By the previous paragraph, there exists a member $D$ (resp. $E$) of $|A|$ which passes through and 
smooth at $x$ (resp. $y$) but not containing $y$ (resp. $x$). Then $D+E$ is a member of $L_{x,y}$ 
passing through $x,y$ and smooth at both $x,y$. Therefore, a general member  of $L_{x,y}$ 
is smooth at both $x,y$, by Lemma \ref{l-gen-element-lin-system}.

\end{proof}

\subsection{Pairs and singularities}
A \emph{pair} $(X,B)$ consists of a normal variety $X$ and 
an $\R$-divisor $B\ge 0$ such that $K_X+B$ is $\R$-Cartier. 
Let $\phi\colon W\to X$ be a log resolution of $(X,B)$ and let 
$$
K_W+B_W=\phi^*(K_X+B).
$$
 The \emph{log discrepancy} of a prime divisor $D$ on $W$ with respect to $(X,B)$ 
is $1-\mu_DB_W$ and it is denoted by $a(D,X,B)$.
We say $(X,B)$ is \emph{lc} (resp. \emph{klt})(resp. \emph{$\epsilon$-lc}) 
if $a(D,X,B)$ is $\ge 0$ (resp. $>0$)(resp. $\ge \epsilon$) for every $D$. Note that if $(X,B)$ is $\epsilon$-lc, then 
automatically $\epsilon\le 1$ because $a(D,X,B)=1$ for almost all $D$. 

A \emph{non-klt place} of $(X,B)$ is a prime divisor $D$ over $X$, that is, 
on birational models of $X$, such that $a(D,X,B)\le 0$. A \emph{non-klt centre} is the image on 
$X$ of a non-klt place. 

\emph{Sub-pairs} and their singularities are defined similarly by letting the coefficients of $B$ to be any real number. 
In this case instead of lc, klt, etc, we say sub-lc, sub-klt, etc.

\begin{lem}\label{l-gen-element-v.ample-div-passing-x,y-contraction}
Let $(X,B)$ be a projective $\epsilon$-lc pair for some $\epsilon>0$ and $f\colon X\to Z$ be a contraction. 
Let $A$ be a very ample divisor on $Z$. 
For each pair of closed points $z,z'\in Z$, let $L_{z,z'}$ be the sub-linear system of $|2A|$ 
consisting of the members that pass through $z,z'$. Then there is a 
non-empty open subset $U\subseteq Z$ such that if $z,z'\in U$ are closed points and  if 
$H$ is a general member of $L_{z,z'}$, then $(X,B+G)$ is a plt pair where 
$G:=f^*H$. In particular, $G$ is normal and $(G,B_G)$ is an $\epsilon$-lc pair where 
$$
K_G+B_G:=(K_X+B+G)|_G.
$$
\end{lem}
\begin{proof}
Let $U$ be as in Lemma \ref{l-gen-element-v.ample-div-passing-x,y} chosen for $|2A|$ on $Z$.
We can assume that $U$ is contained in the smooth locus of $Z$. 
Let $\phi\colon W\to X$ be a log resolution of $(X,B)$ and let $\Sigma$ be the union of the exceptional divisors of 
$\phi$ and the birational transform of $\Supp B$. Shrinking $U$ we can assume that for any stratum $S$ 
of $(W,\Sigma)$ the morphism $S\to Z$ is smooth over $U$.  A stratum of $(W,\Sigma)$ is either  
$W$ itself or an irreducible component of 
$\bigcap_{i\in I} D_i$ for some $I\subseteq \{1,\dots,r\}$ where $D_1,\dots,D_r$ are the irreducible components of 
$\Sigma$. For each stratum $S$ we can assume that either 
$S\to Z$ is surjective or that its image is contained in $Z\setminus U$. 

Pick closed points $z,z'\in U$ and a general member $H$ of $L_{z,z'}$, and let $G=f^*H$ and $E=\phi^*G$. 
We claim that $(W,\Sigma+E)$ is log smooth. 
Let $S$ be a stratum of $(W,\Sigma)$. 
Since $L_{z,z'}$ is base point 
free outside $z,z'$ by definition of $L_{z,z'}$, 
the pullback of $L_{z,z'}$ to $S$ is base point free outside the fibres of $S\to Z$ over $z,z'$. Thus 
any singular point of $E|_S$ (if there is any) is mapped to  $z$ or $z'$. In particular, if 
$S\to Z$ is not surjective, then $E|_S$ is smooth because in this case $E|_S$ has no point 
mapping to either $z$ or $z'$. Assume that $S\to Z$ is surjective. 
Then $E|_S\to H$ is surjective. Moreover, 
by our choice of $U$, the fibres of $E|_S\to H$ over $z,z'$ are both smooth as they coincide with the 
fibres of $S\to Z$ over $z,z'$. Therefore, $E|_S$ is 
smooth because $H$ is smooth at $z,z'$ by the choice of $U$ and by 
Lemma \ref{l-gen-element-v.ample-div-passing-x,y}. To summarise we have shown that 
$E|_S$ is smooth for each stratum $S$ of $(W,\Sigma)$.  This implies that $(W,\Sigma+E)$ is log smooth. 

Let $K_W+B_W$ be the pullback of $K_X+B$. Then each coefficient of $B_W$ is $\le 1-\epsilon$. 
Since $K_W+B_W+E$ is the pullback of $K_X+B+G$, we deduce that $(X,B+G)$ is plt, hence 
$G$ is normal [\ref{kollar-mori}, Proposition 5.51]. On the other hand, 
$$
K_E+B_E:=K_E+B_W|_E=(K_W+B_W+E)|_E
$$ 
is the pullback of 
$$
K_G+B_G:=(K_X+B+G)|_G.
$$
Moreover, since $(W,B_W+E)$ is log smooth and since $E$ is not a component of $B_W$,  
the coefficients of $B_E$ are each at most $1-\epsilon$. Therefore, $(G,B_G)$ is an 
$\epsilon$-lc pair.

\end{proof}

\subsection{Rational approximation of boundary divisors}

\begin{lem}\label{l-rational-approximation}
Let $(X,B)$ be an $\epsilon$-lc pair and $X\to Z$ be a contraction such that $K_X+B\sim_\R 0/Z$. 
Then for each positive real number $\delta$ we can find $\Q$-boundaries $B_i$ and 
 real numbers $r_i>0$ such that 
\begin{itemize}
\item $\sum r_i=1$,  

\item $K_X+B=\sum r_i(K_X+B_i)$, 

\item $(X,B_i)$ is $\frac{\epsilon}{2}$-lc, 

\item $K_{X}+B_i\sim_\Q 0/Z$, 

\item  $\Supp B_i=\Supp B$, and  

\item the coefficients of $B-B_i$ are in $[-\delta,\delta]$. 
 \end{itemize} 
\end{lem}
\begin{proof}
If $B$ is a $\Q$-boundary, then the statement holds trivially by taking $r_1=1$ and $B_1=B$. 
We can then assume that $B$ is not a $\Q$-boundary, in particular, $B\neq 0$.  
Write $B=\sum_1^s b_lD_l$ where $D_l$ are the irreducible components of $B$. 
Since $K_X+B\sim_\R 0/Z$, we can write 
$$
K_X+B=K_X+\sum_1^s b_lD_l=\sum_1^p \alpha_j \Div(f_j)+\sum_1^q \beta_kP_k
$$
where $\alpha_j,\beta_k$ are real numbers, $f_j$ are rational functions, and $P_k$ are pullbacks of Cartier divisors on $Z$. 
Consider the affine space $\A^{s+p+q}$ with coordinates 
$$
u_1,\cdots,u_s, v_1,\cdots,v_p,w_1,\cdots,w_q.
$$
Let $H\subset \A^{s+p+q}$ be the set of points satisfying 
$$
K_X+\sum_1^s u_lD_l=\sum_1^p v_j \Div(f_j)+\sum_1^q w_kP_k.
$$ 
Then $H$ is non-empty as it contains the point
$$
{\bf{x}}:=(b_1,\cdots,b_s, \alpha_1,\cdots,\alpha_p,\beta_1,\cdots,\beta_q).
$$ 
Moreover, $H$ is affine, that is, if ${\bf{x}}_i\in H$ and if $\sum e_i=1$ where $e_i$ are real numbers, 
then $\sum e_i{\bf{x}}_i\in H$. Furthermore, since $K_X,D_l,\Div(f_j),P_k$ are all integral divisors, 
$H$ is a rational affine subspace, that is, it is generated by finitely many points with rational coordinates. 
In particular, we can choose ${\bf{x}}_i\in H$ with rational 
coordinates and real numbers $r_i>0$ with $\sum r_i=1$ such that ${\bf{x}}=\sum r_i {\bf{x}}_i$, 
and we can assume that the coordinates of $ {\bf{x}}_i$ are arbitrarily close to those of $\bf{x}$.  
The first $s$ coordinates of ${\bf{x}}_i$ define $B_i$ which satisfy all the properties of the lemma. 

\end{proof}

\subsection{Fano type varieties}
Assume $X$ is a variety and $X\to Z$ is a contraction. We say $X$ is \emph{of Fano type over} $Z$ 
if there is a boundary $C$ such that $(X,C)$ is klt and $-(K_X+C)$ is ample over $Z$ 
(or equivalently, nef and big over $Z$). This is equivalent to 
having a boundary $B$ such that $(X,B)$ is klt, $K_X+B\sim_\R 0/Z$ and $B$ is big over $Z$. 
By [\ref{BCHM}], we can run MMP over $Z$ on any $\R$-Cartier divisor $M$ on $X$ and the MMP ends 
with a minimal model or a Mori fibre space for $M$.

\begin{lem}\label{l-div-pef-on-FT}
Assume $f\colon X\to Z$ is a contraction of normal projective varieties, $L$ is an $\R$-divisor on 
$X$ and $A$ is an ample $\R$-divisor on $Z$. If $L$ is pseudo-effective and $X$ is of Fano type over $Z$, then 
$|L+af^*A|_\R\neq \emptyset$  for any real number $a>0$.
\end{lem}
\begin{proof}
Since  $L$ is pseudo-effective and $X$ is of Fano type over $Z$, $L$ has a minimal model over $Z$. 
Replacing $X$ with the minimal model we can assume $L$ is semi-ample over $Z$. Thus $L$ defines 
a contraction $X\to Y/Z$ and $L$ is the pullback of an ample$/Z$ $\R$-divisor $N$ on $Y$. Then 
$N+bg^*A$ is ample for any $b\gg 0$ where $g$ denotes $Y\to Z$. Since $L$ is pseudo-effective, 
$N$ is pseudo-effective, hence 
$$
|N+tbg^*A|_\R=|(1-t)N+t(N+bg^*A)|_\R \neq \emptyset
$$ 
for any $t\in (0,1]$.
In particular, if $b>a>0$, then letting $t=\frac{a}{b}$ we see that $|N+ag^*A|_\R\neq \emptyset$ which in turn implies 
$|L+af^*A|_\R\neq \emptyset$.

\end{proof}

\subsection{Complements}\label{ss-compl}
Let $(X,B)$ be a pair  and let $X\to Z$ be a contraction. 
A \emph{strong $n$-complement} of $K_{X}+B$ over a point $z\in Z$ is of the form 
$K_{X}+{B}^+$ such that over some neighbourhood of $z$ we have the following properties:
\begin{itemize}
\item $(X,{B}^+)$ is lc, 

\item $n(K_{X}+{B}^+)\sim 0$, and 

\item ${B}^+\ge B$.
\end{itemize}

When $Z$ is a point, we just say that $K_X+B^+$ is a strong $n$-complement of $K_X+B$. We recall one of the main 
results of [\ref{B-compl}].

\begin{thm}[{[\ref{B-compl},  Theroem 1.7]}]\label{t-bnd-compl-usual}
Let $d$ be a natural number and $\mathfrak{R}\subset [0,1]$ be a finite set of rational numbers.
Then there exists a natural number $n$ 
depending only on $d$ and $\mathfrak{R}$ satisfying the following.  
Assume $(X,B)$ is a projective pair such that 
\begin{itemize}

\item $(X,B)$ is lc of dimension $d$,

\item  the coefficients of $B$ are in $\Phi(\mathfrak{R})$, 

\item $X$ is of Fano type, and 

\item $-(K_{X}+B)$ is nef.\
\end{itemize}
Then there is a strong $n$-complement $K_{X}+{B}^+$ of $K_{X}+{B}$. 
Moreover, the complement is also a strong $mn$-complement for any $m\in \N$.\\ 

\end{thm}

\subsection{Bounded families of pairs}\label{ss-bnd-couples}
A \emph{couple} $(X,D)$ consists of a normal projective variety $X$ and a  divisor 
$D$ on $X$ whose non-zero coefficients are all equal to $1$, i.e. $D$ is a reduced divisor. 
The reason we call $(X,D)$ a couple rather than a pair is that we are concerned with 
$D$ rather than $K_X+D$ and we do not want to assume $K_X+D$ to be $\Q$-Cartier 
or with nice singularities. Two couples $(X,D)$ and $(X',D')$ are isomorphic 
(resp. isomorphic in codimension one) if 
there is an isomorphism  $X\to X'$ (resp. birational map $X\bir X'$ which is an isomorphism in codimension one) 
mapping $D$ onto $D'$ (resp. such that $D$ is the birational transform of $D'$).

We say that a set $\mathcal{P}$ of couples  is \emph{birationally bounded} if there exist 
finitely many projective morphisms $V^i\to T^i$ of varieties and reduced divisors $C^i$ on $V^i$ 
such that for each $(X,D)\in \mathcal{P}$ there exist an $i$, a closed point $t\in T^i$, and a 
birational isomorphism $\phi\colon V^i_t\bir X$ such that $(V^i_t,C^i_t)$ is a couple and 
$E\le C_t^i$ where 
$V_t^i$ and $C_t^i$ are the fibres over $t$ of the morphisms $V^i\to T^i$ and $C^i\to T^i$ 
respectively, and $E$ is the sum of the 
birational transform of $D$ and the reduced exceptional divisor of $\phi$.
We say $\mathcal{P}$ is \emph{bounded} if we can choose $\phi$ to be an isomorphism. 

We say that a set $\mathcal{P}$ of couples  is \emph{bounded up to isomorphism in codimension one} 
if there is a bounded set $\mathcal{P}'$ of couples such that each $(X,D)\in \mathcal{P}$ is 
isomorphic in codimension one with some $(X',D')\in \mathcal{P}'$.
  
A set $\mathcal{R}$ of projective pairs $(X,B)$ is said to be \emph{log birationally bounded} 
(resp. \emph{log bounded}, etc) 
if the set of the corresponding couples $(X,\Supp B)$ is birationally bounded (resp. bounded, etc).
Note that this does not put any condition on the coefficients of $B$, e.g. we are not requiring the 
coefficients of $B$ to be in a finite set. If $B=0$ for all the $(X,B)\in\mathcal{R}$ we usually remove the 
log and just say the set is birationally bounded (resp. bounded, etc).

\begin{lem}\label{l-bnd-Cartier-index-family}
Let $\mathcal{P}$ be a bounded set of couples and $\mathfrak{R}\subset [0,1]$ be a finite set of 
rational numbers. Then there is a natural number $I$ depending only on $\mathcal{P},\mathfrak{R}$ 
such that if 
\begin{itemize}
\item $(X,B)$ is a projective klt pair,

\item the coefficients of $B$ are in $\mathfrak{R}$, and 

\item $(X,\Supp B)\in \mathcal{P}$,
\end{itemize}
then $I(K_X+B)$ is Cartier.
\end{lem}

\begin{proof}
When $K_X$ is $\Q$-Cartier, the lemma follows from [\ref{B-compl}, Lemma 2.24]. 
The proof of the general case is actually quite similar to the proof of  [\ref{B-compl}, Lemma 2.24]. 
We write the details for convenience. 

Assume there is a sequence $(X_i,B_i)$ of pairs  as in the lemma such that if $I_i$ 
 is the smallest natural number so that $I_i(K_{X_i}+B_i)$ is Cartier, 
 then the $I_i$ form a strictly increasing sequence of numbers.  Perhaps after replacing 
 the sequence with a subsequence, by [\ref{B-compl}, Lemma 2.21], 
we can assume there is a projective morphism $V\to T$ of varieties, a reduced divisor $C$ on 
$V$, and a dense set of closed points $t_i\in T$ such that $X_i$ is the fibre of $V\to T$ over 
$t_i$ and  $\Supp B_i$ is the fibre of $C\to T$ over $t_i$. Since $X_i$ are normal, 
replacing $V$ with its normalisation 
and replacing $C$ with its inverse image with reduced structure, we can assume $V$ is normal.

Let $\phi\colon W\to V$ be a resolution of $V$ and let $\Delta$ be  the 
reduced exceptional divisor of $\phi$. Running an MMP$/V$ on $K_W+\Delta$ with scaling of an 
ample divisor, we reach a model $V'$ on which $K_{V'}+\Delta'$ is a limit of movable$/V$ divisors.  
Let $V'\to V$ be the induced morphism and $X_i',\Delta_i'$ be the fibres of $V'\to T$ and $\Delta'\to T$ 
over $t_i$, respectively (note that $\Delta_i'=\Delta'|_{X_i'}$ and since we work in 
characteristic zero, we can assume $\Delta_i'$ is reduced). 
Now we can assume $X_i'$ are general fibres of $V'\to T$, 
hence $\Delta_i'$ is the reduced exceptional divisor of $X_i'\to X_i$.  Since $(X_i,B_i)$ is klt, 
we can write the pullback of $K_{X_i}+B_i$ to $X_i'$ as $K_{X_i'}+B_i'$  where $B_i'$  
has coefficients strictly less than $1$. But then since $X_i'$ are general fibres, 
$$
\Delta_i'-B_i'=K_{X_i'}+\Delta_i'-(K_{X_i'}+B_i')\sim_\Q K_{X_i'}+\Delta_i'/X_i
$$ 
is a limit of movable$/X_i$
 divisors, hence $\Delta_i'-B_i'\le 0$ by the general negativity lemma [\ref{B-lc-flips}, Lemma 3.3] 
which in turn implies $\Delta_i'=0$ as $\Delta_i'$ is reduced. Thus $X_i'\to X_i$ is a small contraction.

There is a $\Q$-divisor $\Gamma_i'\ge 0$ on $X_i'$ which is anti-ample over $X_i$. Rescaling it we can 
assume $(X_i',B_i'+\Gamma_i')$ is klt. In particular, $X_i'\to X_i$ is a $K_{X_i'}+B_i'+\Gamma_i'$-negative 
contraction of an extremal face of the Mori-Kleiman cone of $X_i'$. 
Thus by the cone theorem [\ref{kollar-mori}, Theorem 3.7], the 
Cartier index of $K_{X_i'}+B_i'$ and $K_{X_i}+B_i$ are the same. 

If $C'\subset V'$ denotes the birational transform of $C$, then $\Supp B_i'$ is the fibre of $C'\to T$ over $t_i$.
Thus replacing $V,C$ with $V',C'$ we can replace  
$(X_i,B_i)$ with $(X_i',B_i')$, hence assume $V$ is $\Q$-factorial. 
Moreover, since $X_i$ is a general fibre, $K_{X_i}=K_{V}|_{X_i}$ which shows that the 
Cartier index of $K_{X_i}$ is bounded, so we need to bound the Cartier index of $B_i$.

Pick $I$ so that $IC$ is Cartier. Let $D_i=\Supp B_i$. Then  $D_i=C|_{X_i}$, hence 
$ID_i$ is Cartier. This gives a contradiction if $B_i$ are all irreducible. 
In general, let $h_i\in\Q$ be the largest number such that $B_i-h_iD_i\ge 0$. Then $B_i-h_iD_i$ has 
at least one component less than $B_i$, and the coefficients of $B_i-h_iD_i$ belong to some 
finite set which is independent of $i$. Thus we can apply induction on the number of components of $B_i$ 
(which is a bounded number) to derive a contradiction.  
 
\end{proof}

\subsection{BAB and lower bound on lc thresholds}

We recall some of the main results of [\ref{B-BAB}] regarding boundedness of Fano's and lc thresholds in families.

\begin{thm}[{[\ref{B-BAB}, Theorem 1.1]}]\label{t-BAB}
Let $d$ be a natural number and $\epsilon$ a positive real number. Then the projective 
varieties $X$ such that  

$\bullet$ $(X,B)$ is $\epsilon$-lc of dimension $d$ for some boundary $B$, and 

$\bullet$  $-(K_X+B)$ is nef and big,\\\\
form a bounded family. 
\end{thm}

On the other hand, lc thresholds are bounded from below under suitable assumptions:

\begin{thm}[{[\ref{B-BAB}, Theorem 1.6]}]\label{t-bnd-lct}
Let $d,r$ be natural numbers and $\epsilon$ be a positive real number. 
Then  there is a positive real number $t$ depending only on $d,r,\epsilon$ satisfying the following. 
Assume 

$\bullet$ $(X,B)$ is a projective $\epsilon$-lc pair of dimension $d$, 

$\bullet$ $A$ is a very ample divisor on $X$ with $A^d\le r$,

$\bullet$ $A-B$ is ample, and 

$\bullet$ $M\ge 0$ is an $\R$-Cartier $\R$-divisor with $|A-M|_\R\neq \emptyset$.\\
Then  
$$
\lct(X,B,|M|_\R)\ge \lct(X,B,|A|_\R)\ge t.
$$
\end{thm}

Note that the conditions on $A,B,M$ essentially say that $X$ belongs to a bounded family and that the 
``degrees" of $B,M$ with respect $A$ are bounded.

\subsection{b-divisors}\label{ss-b-divisor}

We recall some definitions regarding b-divisors but not in full generality. 
Let $X$ be a variety. A \emph{b-$\R$-Cartier b-divisor over $X$} is the choice of  
a projective birational morphism 
$Y\to X$ from a normal variety and an $\R$-Cartier divisor $M$ on $Y$ up to the following equivalence: 
 another projective birational morphism $Y'\to X$ from a normal variety and an $\R$-Cartier divisor
$M'$ defines the same b-$\R$-Cartier  b-divisor if there is a common resolution $W\to Y$ and $W\to Y'$ 
on which the pullbacks of $M$ and $M'$ coincide.  

A b-$\R$-Cartier  b-divisor  represented by some $Y\to X$ and $M$ is \emph{b-Cartier} if  $M$ is 
b-Cartier, i.e. its pullback to some resolution is Cartier.

\subsection{Generalised pairs}

A \emph{generalised pair} consists of 
\begin{itemize}
\item a normal variety $X$ equipped with a projective
morphism $X\to Z$, 

\item an $\R$-divisor $B\ge 0$ on $X$, and 

\item a b-$\R$-Cartier  b-divisor over $X$ represented 
by some projective birational morphism $X' \overset{\phi}\to X$ and $\R$-Cartier divisor
$M'$ on $X'$
\end{itemize}
such that $M'$ is nef$/Z$ and $K_{X}+B+M$ is $\R$-Cartier,
where $M := \phi_*M'$. 

We usually refer to the pair by saying $(X,B+M)$ is a  generalised pair with 
data $X'\to X\to Z$ and $M'$, and call $M'$ the nef part. Note that our notation 
here, which seems to be preferred by others, is slightly different from those in [\ref{BZh}][\ref{B-compl}]. 

We now define generalised singularities.
Replacing $X'$ we can assume $\phi$ is a log resolution of $(X,B)$. We can write 
$$
K_{X'}+B'+M'=\phi^*(K_{X}+B+M)
$$
for some uniquely determined $B'$. For a prime divisor $D$ on $X'$ the \emph{generalised log discrepancy} 
$a(D,X,B+M)$ is defined to be $1-\mu_DB'$. 
We say $(X,B+M)$ is 
\emph{generalised lc} (resp. \emph{generalised klt})(resp. \emph{generalised $\epsilon$-lc}) 
if for each $D$ the generalised log discrepancy $a(D,X,B+M)$ is $\ge 0$ (resp. $>0$)(resp. $\ge \epsilon$).

A \emph{generalised non-klt centre} of a generalised pair $(X,B+M)$ is the image on $X$ of a prime divisor 
$D$ over $X$  with $a(D,X,B+M)\le 0$, and 
the \emph{generalised non-klt locus} of the generalised pair is the union of all the generalised non-klt centres.  

\emph{Generalised sub-pairs} and their singularities are similarly defined by allowing the coefficients of 
$B$ to be any real number. 

To state the next lemma we need the notion of \emph{very exceptional divisors}. 
Given a contraction $f\colon X\to Z$ of normal varieties and an $\R$-divisor $N$ on $X$, we say that 
$N$ is very exceptional over $Z$ if $\Supp N$ is vertical over $Z$ and  that 
for any prime divisor $D$ on $Z$ there is a prime divisor $S$ on $X$ mapping onto $D$ but such that 
$S$ is not a component of $N$. 

 \begin{lem}\label{l-mmp-v-exc}
Let $(X,B+M)$ be a $\Q$-factorial generalised klt generalised pair with data $X'\to X\to Z$ and $M'$. Assume 
$K_X+B+M\sim_\R N/Z$ where $N\ge 0$ is very exceptional over $Z$. Then any MMP 
on $K_X+B+M$ over $Z$ with scaling of an ample divisor terminates with a model $Y$ on which 
$K_Y+B_Y+M_Y\sim_\R 0/Z$.
\end{lem}
\begin{proof}
This proof is similar to that of [\ref{BZh}, Theorem 1.8].
Let $C\ge 0$ be an ample $\R$-divisor such that $(X,B+C+M)$ is generalised klt (with the same nef part $M'$) 
and $K_X+B+C+M$ is ample over $Z$. Run the MMP on $K_X+B+M$ over $Z$ with scaling of $C$ (as defined 
before [\ref{BZh}, Lemma 4.4]). This consists of a sequence $X_i\bir X_{i+1}$ 
of divisorial contractions and flips where $X=X_1$. Let $\lambda_i$ be the numbers obtained in the process so that 
$K_{X_i}+B_i+\lambda_iC_i+M_i$ is nef over $Z$. 
If $\lambda=\lim \lambda_i>0$, 
then the MMP terminates by [\ref{BZh}, Lemma 4.4] because the MMP is also an MMP 
on $K_X+B+\frac{\lambda}{2} C+M$, so in this case replacing $X$ with the 
minimal model we can assume $K_X+B+M$ is nef over $Z$. If  $\lambda=\lim \lambda_i=0$, then replacing $X$ 
with $X_i$ for some $i\gg 0$, we can assume 
that $K_X+B+M$ is a limit of movable$/Z$ $\R$-divisors. In either case (that is, $\lambda>0$ or $\lambda=0$), 
for any prime divisor $S$ on $X$ and for 
the general curves $\Gamma$ of $S$ contracted over $Z$, we have $N\cdot \Gamma\ge 0$. Therefore, 
$N=0$ by the general negativity lemma [\ref{B-lc-flips}, Lemma 3.3] as $N$ is very exceptional over $Z$.
In other words the MMP contracts $N$.

\end{proof}

The next lemma is useful for reducing problems about generalised log Calabi-Yau fibrations to 
usual log Calabi-Yau fibrations.

\begin{lem}\label{l-from-gen-fib-to-usual-fib}
Let $d,r$ be natural numbers and $\epsilon$ be a positive real number. Let 
$(X,B+M)\to Z$ be a generalised $(d,r,\epsilon)$-Fano type fibration (as in \ref{d-gen-FT-fib}) such that  
 $-(K_X+\Delta)$ is big over $Z$ for some $0\le \Delta\le B$. 
Then we can find a boundary $\Theta\ge \Delta$ such that $(X,\Theta)\to Z$ is a $(d,r,\frac{\epsilon}{2})$-Fano type fibration.
\end{lem}
\begin{proof}
Taking a $\Q$-factorialisation we can assume $X$ is $\Q$-factorial.  
Since $B-\Delta$ is effective and $M$ is pseudo-effective (as it is the pushdown of a nef divisor),
$B-\Delta+M$ is pseudo-effective. Moreover, since  $-(K_X+\Delta)$ is big over $Z$, 
$$
B-\Delta+M\sim_\R -(K_X+\Delta)/Z
$$ 
is big over $Z$. Thus  
$$
t(B-\Delta+M)+f^*A
$$ 
is globally big for any sufficiently small $t>0$, hence 
 $B-\Delta+M+f^*A$ is globally big as  $B-\Delta+M$ is pseudo-effective.

Let $\phi\colon X'\to X$ be a log resolution of $(X,B)$ 
on which the nef part $M'$ of $(X,B+M)$ resides. Write 
$$
K_{X'}+B'+M'=\phi^*(K_X+B+M)
$$
and 
$$
K_{X'}+\Delta'=\phi^*(K_X+\Delta).
$$
Then
$$
B'-\Delta'+M'=\phi^*(B-\Delta+M).
$$ 
Since $(X,B+M)$ is generalised $\epsilon$-lc, the coefficients of $B'$ do not exceed $1-\epsilon$. 
In addition, since $M'$ is nef and $B\ge \Delta$, we have $B'\ge \Delta'$ by the negativity lemma 
applied to $B'-\Delta'$ and the morphism $\phi$.

Since $B-\Delta+M+f^*A$ is big, we can write 
$$
B'-\Delta'+M'+\phi^*f^*A=\phi^*(B-\Delta+M+f^*A)\sim_\R G'+H'
$$ 
where $G'\ge 0$ and $H'$ is ample. Replacing $\phi$ we can assume $\phi$ is a log resolution of 
$(X,B+\phi_*G')$.
Pick a small real number $\alpha>0$ and pick a general 
$$
0\le R'\sim_\R \alpha H'+(1-\alpha)M'.
$$
Let 
$$
\Theta':=\Delta'+(1-\alpha)(B'-\Delta')+\alpha G'+R'.
$$
We can make the above choices so that $(X',\Theta')$ is log smooth. Since 
$$
\Delta'\le \Delta'+(1-\alpha)(B'-\Delta')\le \Delta'+(B'-\Delta')=B',
$$ 
we have 
$$
\Delta'\le\Theta'\le B'+\alpha G'+R'.
$$
In particular, we can choose $\alpha$ and $R'$ so that the coefficients of $\Theta'$ do not exceed $1-\frac{\epsilon}{2}$. 

By construction, we have 
$$
K_{X'}+\Theta'=K_{X'}+ \Delta'+(1-\alpha)(B'-\Delta')+\alpha G'+R'
$$
$$
\sim_\R K_{X'}+  \Delta'+(1-\alpha)(B'-\Delta')+\alpha G'+\alpha H'+(1-\alpha)M'
$$
$$
\sim_\R K_{X'}+ \Delta'+(1-\alpha)(B'-\Delta')+\alpha(B'-\Delta'+M'+ \phi^* f^*A)+(1-\alpha)M'
$$
$$
= K_{X'}+ \Delta'+B'-\Delta'+\alpha(M'+ \phi^* f^*A)+(1-\alpha)M'
$$
$$
=K_{X'}+ B'+M'+\alpha\phi^*f^*A\sim_\R \phi^*f^*(L+\alpha A).
$$
Therefore,   letting $\Theta=\phi_*\Theta'$, we have 
$$
K_X+\Theta\sim_\R f^*(L+\alpha A).
$$ 
Choosing $\alpha$ small enough we can ensure $A-(L+\alpha A)$ is 
ample. Moreover, since $K_{X'}+\Theta' \sim_\R 0/X$ and since the coefficients of $\Theta'$ do not 
exceed $1-\frac{\epsilon}{2}$, the pair $(X,\Theta)$ is $\frac{\epsilon}{2}$-lc.
Thus $(X,\Theta)\to Z$ is a $(d,r,\frac{\epsilon}{2})$-Fano type fibration. Finally, it is obvious that $\Theta\ge \Delta$.
 
\end{proof}

\subsection{Bound on singularities}

\begin{lem}\label{l-bnd-sing-gen-cy-pairs}
Let $d,p\in\N$ and $\Phi\subset [0,1]$ be a DCC set. Then there is a real number $\epsilon>0$ depending only 
on $d,p$ and $\Phi$ such that if $(X,B+M)$ is a projective generalised pair with data $X'\to X$ and $M'$ 
satisfying:
\begin{itemize}
\item $(X,B+M)$ is generalised klt of dimension $d$,

\item the coefficients of $B$ are in $\Phi$,  

\item $pM'$ is b-Cartier, and 

\item $K_X+B+M\sim_\R 0$,
\end{itemize}
then $(X,B+M)$ is generalised $\epsilon$-lc.
\end{lem}
\begin{proof}
If the lemma does not hold, then there exist a decreasing sequence $\epsilon_i$ of numbers approaching $0$ 
and a sequence $(X_i,B_i+M_i)$ of pairs as in the statement such that $(X_i,B_i+M_i)$ is not 
generalised $\epsilon_i$-lc. There is a prime divisor 
$D_i$ over $X_i$ with generalised log discrepancy 
$$
a(D_i,X_i,B_i+M_i)<\epsilon_i.
$$ 
If $D_i$ is a divisor on $X_i$, we let $X_i''\to X_i$ be the identity morphism. If not, then since 
$(X_i,B_i+M_i)$ is generalised klt, there is a birational morphism $X_i''\to X_i$ extracting 
$D_i$ but no other divisors. We can assume that the induced map $X_i'\bir X_i''$ is a morphism. 

Let $K_{X_i''}+B_i''+M_i''$ be the pullback of $K_{X_i}+B_i+M_i$ where $M_i''$ 
is the pushdown of $M_i$. We consider $(X_i'',B_i''+M_i'')$ as a generalised pair with 
data $X_i'\to X_i''$ and $M_i'$.
Let 
$$
b_i=1-a(D_i,X_i,B_i+M_i)
$$ 
which is  the coefficient of $D_i$ in $B_i''$. Then $b_i\ge 1-\epsilon_i$ and 
$B_i''$ has coefficients in $\Phi'':=\Phi\cup \{b_i \mid i\in\N\}$. 
Replacing the sequence, we can assume $\Phi''$ is a DCC set. 
Now we get a contradiction, by [\ref{BZh}, Theorem 1.6],  because 
$$
K_{X_i''}+B_i''+M_i''\sim_\R 0
$$
and because $\{b_i \mid i\in\N\}$ is not finite 
as the $b_i$ form an infinite sequence approaching $1$. 

\end{proof}

\subsection{Towers of Mori fibre spaces}

We will use the following result of [\ref{DiCerbo-Svaldi}] in the proof of Theorem \ref{cor-bnd-cy-fib-non-product}.

\begin{thm}[{[\ref{DiCerbo-Svaldi}, Theorem 3.2]}]\label{t-tower-of-Mfs}
Let $(X,B)$ be a projective klt pair such that $K_X+B\sim_\R 0$, $B\neq 0$, and $(X,B)$ is not of 
product type. Then there exist a birational map $\phi\colon X\bir X_1$ and a sequence of contractions 
$$
X_1\to  X_2\to \cdots \to X_l
$$
such that $\phi^{-1}$ does not contract divisors, each $X_i\to X_{i+1}$ is a $K_{X_i}$-Mori fibre space, 
and $X_l$ is a point.
\end{thm}

The main point in the proof of the theorem is similar to the following: 

\begin{lem}\label{l-contraction-after-flops}
Suppose that 
\begin{itemize}
\item $(X,B)$ is a $\Q$-factorial projective klt pair, 

\item $h\colon X\to Y$ is a Mori fibre space structure, i.e. a non-birational extremal contraction,

\item  $Y\to Z$ is a contraction,  

\item $K_X+B\sim_\R 0/Z$, and

\item $Y\bir Y'/Z$ is a birational map to a $\Q$-factorial normal projective variety 
which is an isomorphism in codimension one.
\end{itemize}
Then there exist a birational map $X\bir X'$ to a $\Q$-factorial normal projective variety which is an isomorphism 
in codimension one and such that the induced map $X'\bir Y'$ is an extremal contraction (hence a morphism).
\end{lem}
\begin{proof}
Since $K_X+B \sim_\R 0/Y$, by adjunction, we can write 
$$
K_X+B\sim_\R h^*(K_Y+B_Y+M_Y)
$$ 
where we consider $(Y,B_Y+M_Y)$ as a generalised pair (see \ref{rem-base-fib-gen-pair} below) 
which is generalised klt. Since $X$ is $\Q$-factorial and since $X\to Y$ is extremal and 
non-birational, $Y$ is also $\Q$-factorial [\ref{kollar-mori}, Corollary 3.18].  

Since $K_Y+B_Y+M_Y\sim_\R 0/Z$, 
by the same arguments as in [\ref{Kaw-flops}] we can decompose $Y\bir Y'$ into a sequence of flops: 
indeed if $H_{Y'}$ is an ample $\Q$-divisor on $Y'$, then after rescaling $H_{Y'}$, $(Y,B_Y+H_Y+M_Y)$ is 
generalised klt where $H_Y$ is the birational transform of $H_{Y'}$; now running an MMP on 
$K_Y+B_Y+H_Y+M_Y$ over $Z$ ends with $Y'$ as $K_{Y'}+B_{Y'}+H_{Y'}+M_{Y'}$ is ample; 
in particular only flips can occur in the MMP which are flops with respect to $K_Y+B_Y+M_Y$. 

By the previous paragraph, to obtain $X'$ it is enough to consider the case when $Y\bir Y'/Z$ is one single flop 
(in particular, we can assume $Y\to Z$ is an extremal flopping contraction). 
Let $H_{Y'},H_{Y}$ be as before, and let $G$ be the pullback of $H_Y$ to $X$. 
Since $B$ is big over $Y$ and since $Y\to Z$ is birational, $B$ is also big over $Z$.
Thus $X$ is of Fano type over $Z$ as $(X,B)$ is klt and K$_X+B\sim_\R 0/Z$. 
Therefore, there is a minimal model $X'$ for $G$ over $Z$. 
The birational transform of $G$ on $X'$, say $G'$, is semi-ample over $Z$, hence it defines a contraction 
$X'\to V'/Z$ and $G'$ is the pullback of an ample$/Z$ divisor $H_{V'}$. By construction, 
$V'$ is the ample model of $G$ over $Z$. Since $G$ is the pullback of $H_Y$, $V'$ is also the ample 
model of $H_Y$ over $Z$. But the ample model of $H_Y$ is $Y'$, hence $V'=Y'$. In particular, 
$X'\bir Y'$ is a morphism. 

Finally, since  $X\to Z$ has relative Picard number two,  $X'\to Z$ also has relative Picard number two 
as $X\bir X'$ is an isomorphism in codimension one. On the other hand, $Y'\to Z$ has relative 
Picard number one, so $X'\to Y'$ is an extremal contraction. 

\end{proof}


\section{\bf Boundedness of complements}

In this section we construct certain kinds of complements as in Theorem \ref{t-bnd-comp-lc-global} 
which are not usual complements but rather similar to those of [\ref{B-BAB}, Theorem 1.7]. 
The main differences with [\ref{B-BAB}, Theorem 1.7] are that $M-(K_X+B)$ is no longer assumed to be ample, 
$X$ is not necessarily $\Q$-factorial, $M|_S\sim 0$ is replaced with $M|_S\equiv 0$, 
and $n+1$ is replaced with $n+2$. We start with treating a special case of the theorem.

\begin{prop}\label{l-bnd-comp-Gamma}
Theorem \ref{t-bnd-comp-lc-global} holds under the additional assumption that there is a boundary 
$\Gamma$ such that $(X,\Gamma)$ is plt with $S=\rddown{\Gamma}$ 
and such that $\alpha M-(K_X+\Gamma)$ is ample for some real number $\alpha>0$. 
\end{prop}
\begin{proof}
 
\emph{Step 1.} 
\emph{In this step we consider bounded complements on $S$.}
Since $M-(K_X+B)$ is nef and big and  $\alpha M-(K_X+\Gamma)$ is ample,
$$
(1-t+t\alpha)M-(K_X+(1-t)B+t\Gamma)=(1-t)(M-(K_X+B))+t(\alpha M-(K_X+\Gamma))
$$
is ample for any $t\in (0,1)$. Thus replacing $\Gamma$ 
with $(1-t)B+t\Gamma$ for some sufficiently small number $t$, 
we can replace $\alpha$ by some rational number in $(0,2)$. Note that since $(X,B)$ is lc and $S\le \rddown{B}$, 
this change preserves the plt property of $(X,\Gamma)$ and the condition $S=\rddown{\Gamma}$.

Define $K_S+B_S=(K_X+B)|_S$ by adjunction. Then $(S,B_S)$ is lc and 
the coefficients of $B_S$ belong to $\Phi(\mathfrak{S})$ 
for some finite set $\mathfrak{S}\subset [0,1]$ of rational numbers 
depending only on $\mathfrak{R}$ [\ref{PSh-II}, Proposition 3.8][\ref{B-compl}, Lemma 3.3]. 
By assumption, $M|_S\equiv 0$, hence $M|_S\sim_\Q 0$ as $M$ is semi-ample. 
In particular, 
$$
-(K_X+\Gamma)|_S\sim_\R (\alpha M-(K_X+\Gamma))|_S
$$ 
is ample, and since $(X,\Gamma)$ is plt, we deduce that $S$ is of Fano type. Therefore, as 
$$
-(K_S+B_S)\sim_\Q (M-(K_X+B))|_S
$$ 
is nef, there is a natural number $n$ depending only on $d,\mathfrak{S}$ 
such that $K_S+B_S$ has a strong $n$-complement $K_S+B_S^+$ which by definition satisfies 
$B_S^+\ge B_S$ [\ref{B-compl}, Theorem 1.7] (=Theorem \ref{t-bnd-compl-usual}). 
In addition, we can assume that $nB$ is an integral divisor after replacing $n$ with a bounded multiple 
depending on $\mathfrak{R}$. 

Note that since $S$ is of Fano type and $M|_S\sim_\Q 0$, 
in fact we have $M|_S\sim 0$ as $\Pic(S)$ is torsion-free (cf. [\ref{Isk-Prokh}, Proposition 2.1.2] the main 
point being the vanishing $h^i(\mathcal{O}_S)=0$ for $i>0$ which is a consequence of 
Kawamata-Viehweg vanishing theorem).\\

\emph{Step 2}.
\emph{In this step we take a resolution and define appropriate divisors on it.}
Let $\phi\colon X'\to X$ be a log resolution of $(X,B+\Gamma)$, $S'$ be the birational transform of $S$, and 
$\psi\colon S'\to S$ be the induced morphism. Put
$$
N:=M-(K_{X}+B)
$$
and let $K_{X'}+B',M',N'$ be the pullbacks of 
$K_X+B,M,N$, respectively. Let $E'$ be the sum of the components of $B'$ which have coefficient $1$, 
and let $\Delta'=B'-E'$. Define  
$$
L':=(n+2)M'-nK_{X'}-nE'-\rddown{(n+1){\Delta'}}
$$
which is an integral divisor. Note that 
$$
L'=(n+2)M'-nK_{X'}-nB'+n\Delta'-\rddown{(n+1){\Delta'}}
$$
$$
=2M'+n(M'-K_{X'}-B')+n\Delta'-\rddown{(n+1){\Delta'}}
$$
$$
=2M'+nN'+n\Delta'-\rddown{(n+1){\Delta'}}.
$$
Now write $K_{X'}+\Gamma'=\phi^*(K_{X}+\Gamma).$ 
We can assume $B'-\Gamma'$ has sufficiently small (positive or negative) coefficients 
by taking $t$ in the beginning of Step 1 to be sufficiently small.\\ 

\emph{Step 3}.
\emph{In this step we introduce a boundary $\Lambda'$ and study related divisors.}
Let $P'$ be the unique integral divisor so that 
$$
\Lambda':=\Gamma'+{n{\Delta'}}-\rddown{(n+1){\Delta'}}+P'
$$ 
is a boundary, $(X',\Lambda')$ is plt, and $\rddown{\Lambda'}=S'$ (in particular, we are assuming $\Lambda'\ge 0$). 
More precisely, we let $\mu_{S'}P'=0$ and for each prime divisor $D'\neq S'$, we let 
$$
\mu_{D'}P':=-\mu_{D'}\rddown{\Gamma'+{n{\Delta'}}-\rddown{(n+1){\Delta'}}}
$$
which satisfies 
$$
\mu_{D'}P'=-\mu_{D'}\rddown{\Gamma'-\Delta'+\langle (n+1)\Delta'\rangle}
$$ 
where $\langle (n+1)\Delta'\rangle$ is the fractional part of $(n+1)\Delta'$.
This implies $0\le \mu_{D'}P'\le 1$ for any prime divisor $D'$: this is obvious if $D'=S'$, so assume $D'\neq S'$; 
if $D'$ is a component of $E'$, then 
$D'$ is not a component of $\Delta'$ and $\mu_{D'}\Gamma'\in(0,1)$ as $B'-\Gamma'$ has 
small coefficients, hence $\mu_{D'}P'=0$; 
on the other hand, if $D'$ is not a 
component of $E'$, then the absolute value of $\mu_{D'}(\Gamma'-\Delta')=\mu_{D'}(\Gamma'-B')$ is sufficiently small, 
hence $0\le \mu_{D'}P'\le 1$. 

We show $P'$ is exceptional$/X$. 
Assume $D'$ is a component of $P'$ which is not exceptional$/X$ and let $D$ be its pushdown. 
Then $D'\neq S'$ and $D'$ is a component of $\Delta'$ as $\mu_{D'}\Gamma'=\mu_D\Gamma\in [0,1)$, hence 
$$
1>\mu_{D'}\Delta'=\mu_{D'}B'=\mu_{D}B\ge 0.
$$ 
Moreover, since $nB$ is integral,  $\mu_{D'}n\Delta'$ is integral, hence  
$\mu_{D'}\rddown{(n+1)\Delta'}=\mu_{D'}n\Delta'$ which implies 
$$
\mu_{D'}P'=-\mu_{D'}\rddown{\Gamma'}=-\mu_{D}\rddown{\Gamma}=0,
$$
a contradiction.\\

\emph{Step 4.}
\emph{In this step we show that sections of $(L'+P')|_{S'}$ can be lifted to $X'$.}
Let $A:=\alpha M-(K_X+\Gamma)$. 
Letting  $A'=\phi^*A$ we have 
$$
K_{X'}+\Gamma'+A'-\alpha M'=0.
$$ 
Then 
$$
L'+P'= 2M'+nN'+ {n{\Delta'}}-\rddown{(n+1){\Delta'}}+ P'
$$
$$
=K_{X'}+\Gamma'+A'-\alpha M'+2M'+nN'+{n{\Delta'}}-\rddown{(n+1){\Delta'}}+ P'
$$
$$
=K_{X'}+\Lambda'+A'+nN'+(2-\alpha) M' . 
$$
Since $A'+nN'+(2-\alpha) M' $ is nef and big and $(X',\Lambda')$ is plt with $\rddown{\Lambda'}=S'$, 
we have $h^1(L'+P'-S')=0$ by the Kawamata-Viehweg vanishing theorem. Thus 
$$
H^0(L'+P')\to H^0((L'+P')|_{S'})
$$
is surjective.\\

\emph{Step 5.}
\emph{In this step we introduce an effective divisor $G_{S'}\sim (L'+P')|_{S'}$.}
Recall the $n$-complement $K_S+B_S^+$ from step 1.
Let $R_{S}:=B_{S}^+-B_{S}$ which satisfies 
$$
-n(K_{S}+B_{S})=-n(K_{S}+B_{S}^++B_S-B_S^+)\sim -n(B_S-B_S^+)=nR_{S}\ge 0.
$$
Letting $R_{S'}$ be the pullback of $R_{S}$, we get 
$$
nN'|_{S'}=n(M'-(K_{X'}+B'))|_{S'}\sim-n(K_{X'}+B')|_{S'}
$$
$$
=-n\psi^*(K_{S}+B_{S})\sim n\psi^*R_S=nR_{S'}\ge 0. 
$$
Then 
$$
(L'+P')|_{S'}=( 2M'+ nN'+{n{\Delta'}}-\rddown{(n+1){\Delta'}}+P')|_{S'}
$$
$$
\sim G_{S'}:=nR_{S'}+{n{\Delta_{S'}}}-\rddown{(n+1){\Delta_{S'}}}+P_{S'}
$$
where $\Delta_{S'}=\Delta'|_{S'}$ and $P_{S'}=P'|_{S'}$. Note that 
$\rddown{(n+1){\Delta'}}|_{S'}=\rddown{(n+1){\Delta'}|_{S'}}$ since $\Delta'$ and $S'$ intersect transversally.

We show $G_{S'}\ge 0$. Assume $C'$ is a component of $G_{S'}$ 
with negative coefficient. Then 
 there is a component $D'$ of ${{\Delta'}}$ such that $C'$ is a 
component of $D'|_{S'}$. But 
$$
\mu_{C'} ({n{\Delta_{S'}}}-\rddown{(n+1){\Delta_{S'}}})=\mu_{C'} (-\Delta_{S'}+\langle(n+1)\Delta_{S'}\rangle)\ge 
-\mu_{C'} \Delta_{S'}=-\mu_{D'} \Delta'>-1
$$ 
which gives $\mu_{C'}G_{S'}>-1$ and this in turn implies $\mu_{C'}G_{S'}\ge 0$ because $G_{S'}$ is integral, a contradiction. 
Therefore $G_{S'}\ge 0$, and by Step 4, $L'+P'\sim G'$ for some effective divisor $G'$ whose support does not contain $S'$ 
and $G'|_{S'}=G_{S'}$.\\ 

\emph{Step 6.}
\emph{In this step we introduce $\Lambda$ and show that it satisfies the properties listed in the theorem.}
Let $L,P,G,E,\Delta$ be the pushdowns to $X$ of $L',P',G',E',\Delta'$.
By the definition of $L'$, by the previous step, and by the exceptionality of $P'$,   we have  
$$
(n+2)M-nK_{X}-nE-\rddown{(n+1)\Delta}=L=L+P\sim G\ge 0.
$$
Since $nB$ is integral, $\rddown{(n+1)\Delta}= n\Delta$, so  
$$
(n+2)M-n(K_{X}+B)
$$
$$
=(n+2)M-nK_{X}-nE-{n\Delta}
=L\sim nR:=G\ge 0.
$$

Let $\Lambda:={B}^+:=B+R$. By construction, $n(K_X+B^+)\sim (n+2)M$. 
It remains to show that $(X,{B}^+)$ is lc over $z=f(S)$. First we show that $(X,{B}^+)$ is lc near $S$:
this  follows from inversion of adjunction [\ref{kawakita}], if we show 
$$
K_{S}+B_{S}^+=(K_{X}+{B}^+)|_{S}
$$
which is equivalent to showing $R|_{S}=R_{S}$.
Since  
$$
nR':=G'-P'+\rddown{(n+1)\Delta'}- n\Delta'\sim L'+\rddown{(n+1)\Delta'}- n\Delta'=2M'+nN'\sim_\Q 0/X
$$
and since $\rddown{(n+1)\Delta}- n\Delta=0$, we get  $\phi_*nR'=G=nR$ and that $R'$ is the pullback of $R$. 
Now
$$
nR_{S'}=G_{S'}-P_{S'}+\rddown{(n+1)\Delta_{S'}}- n\Delta_{S'}
$$
$$
=(G'-P'+\rddown{(n+1)\Delta'}- n\Delta')|_{S'}=nR'|_{S'}
$$
which means $R_{S'}=R'|_{S'}$, hence $R_S$ and $R|_S$ both pull back to $R_{S'}$ which implies
$R_{S}=R|_{S}$.

Finally, $(X,{B}^+)$ is lc over $z=f(S)$ otherwise by the plt property of $(X,\Gamma)$ we 
can take $u>0$ to be sufficiently small so that the non-klt locus of  
$$
(X,(1-u)B^++u\Gamma)
$$
has at least two connected components (one of which is $S$) near the fibre $f^{-1}\{z\}$. This  
contradicts the connectedness principle [\ref{Kollar-flip-abundance}, Theorem 17.4]
as 
$$
-(K_X+(1-u)B^++u\Gamma)=-(1-u)(K_X+B^+)-u(K_X+\Gamma)
$$
$$
\sim_\R -u(K_X+\Gamma)\sim_\R u\alpha M-u(K_X+\Gamma)/Z
$$
is ample over $Z$.

\end{proof}

\begin{proof}(of Theorem \ref{t-bnd-comp-lc-global})
\emph{Step 1.}
\emph{In this step we make some modifications of the setting of the theorem and introduce a boundary $\Delta$.}
Adding $1$ to $\mathfrak{R}$ and replacing $(X,B)$ with a $\Q$-factorial dlt model, 
we can assume $X$ is $\Q$-factorial and 
that $S$ is a component of $\rddown{B}$. All the assumptions of the theorem are preserved. 
By assumption, $M$ is the pullback of an ample divisor on $Z$. Moreover, 
$M-(K_X+B)$ is nef and big, hence in particular it is semi-ample over $Z$ 
since $X$ is of Fano type over $Z$. Thus by Lemma \ref{l-semi-ample-div}, 
$aM-(K_X+B)$ is semi-ample for any rational number $a>1$, so
it defines a birational contraction $X\to Y$ which is simply the contraction over $Z$ defined 
by $M-(K_X+B)$: indeed, for any curve $C$ on $Y$, 
$$
\mbox{$(aM-(K_X+B))\cdot C=0$ iff $(M-(K_X+B))\cdot C=0$ and $M\cdot C=0$.}
$$
In particular, the induced map $Y\bir Z$ is a morphism and $K_X+B\sim_\Q 0/Y$. 

After running an MMP over $Y$ on $-(K_X+S)$ and replacing $X$ with the resulting model 
we can assume $-(K_X+S)$ is semi-ample over $Y$. Note that $S$ is not contracted by the MMP 
since the MMP is also an MMP on $(B-S)$ whose support does not contain $S$. Letting $\Delta=(1-b)B+bS$ 
for a sufficiently small rational number $b>0$ (depending on $a$). Then $(X,\Delta)$ is lc, and 
since $aM-(K_X+B)$ is the pullback of an ample divisor on $Y$ and $-(K_X+S)$ is semi-ample over $Y$, 
we see that 
$$
aM-(K_X+\Delta)=(1-b)(aM-(K_X+B))+b(aM-(K_X+S))
$$ 
 is semi-ample and nef and big. 
Moreover, every non-klt centre of $(X,\Delta)$ is also a non-klt centre of $(X,S)$, hence such centres are contained in $S$
because $X$ is $\Q$-factorial and of Fano type over $Z$ which ensures $(X,0)$ is klt. 

Replacing $(X,B)$ once again with a $\Q$-factorial dlt model and replacing $K_X+\Delta$ with its pullback, 
we can assume that 
\begin{itemize}
\item we have  a boundary $\Delta\le B$,

\item $(X,\Delta)$ is $\Q$-factorial dlt, 

\item  $S$ is a component of $\rddown{\Delta}$, 

\item $aM-(K_X+B)$ and $aM-(K_X+\Delta)$ are semi-ample and nef and big for some rational number 
$a>1$, 

\item if $X\to Y$ and $X\to V$ are the contractions defined by $aM-(K_X+B)$ and $aM-(K_X+\Delta)$, respectively, 
then $V\bir Y$ and $Y\bir Z$ are morphisms, and 

\item all the non-klt centres of $(X,\Delta)$ map to $z=f(S)$.
\end{itemize} 
In particular, $M|_{\rddown{\Delta}}\sim_\Q 0$.\\

\emph{Step 2.}
\emph{In this step we introduce another boundary $\tilde{\Delta}$.}
By Step 1, $aM-(K_X+\Delta)$ is  semi-ample defining a contraction 
$X\to V$ so that $V\bir Y$ is a morphism. After running an MMP on $-K_X$ over $V$ 
we can assume that $-K_X$ is semi-ample over $V$. This preserves all the properties listed in 
Step 1 except that the dlt property of $(X,\Delta)$ maybe lost and $S$ maybe contracted. 
We will recover these properties in a bit. Let 
$\tilde{\Delta}=(1-c)\Delta$ for some sufficiently small $c>0$. Then $(X,\tilde{\Delta})$ is klt as $(X,0)$ is klt. 
Moreover, we can assume that 
$$
aM-(K_X+\tilde{\Delta})=(1-c)(aM-(K_X+\Delta))+c(aM-K_X)
$$ 
is semi-ample and nef and big because $aM-(K_X+\Delta)$ is the pullback of an ample divisor on $V$ 
and $aM-K_X$ is semi-ample over $V$. Now replacing $(X,\Delta)$ with a $\Q$-factorial dlt model 
on which $S$ of Step 1 is a divisor, and replacing $K_X+B$ and $K_X+\tilde{\Delta}$ with their pullbacks 
to the dlt model we can assume that in addition to the properties listed in Step 1 we have 
\begin{itemize}
\item a boundary $\tilde{\Delta}\le \Delta$, 

\item $(X,\tilde{\Delta})$ is klt, 

\item the coefficients of $\Delta-\tilde{\Delta}$ are sufficiently small, and 

\item that $aM-(K_X+\tilde{\Delta})$ is nef and big.
\end{itemize}
Note that the coefficients of $B$ are still in $\mathfrak{R}$ because $\Delta\le B$ implies that taking the above 
$\Q$-factorial dlt model only extracts divisors whose log discrepancy with respect to $(X,B)$ is zero.\\

\emph{Step 3.}
\emph{In this step we introduce divisors $H,G$ and deal with the case when 
$\Supp G$ does not contain non-klt centres of $(X,\Delta)$.} Write 
$$
aM-(K_X+\Delta)\sim_\Q H+G
$$
where $H,G\ge 0$ are $\Q$-divisors and $H$ is ample. 
First assume that $\Supp G$ does not contain any non-klt centre of $(X,\Delta)$. 
Then, for some small $\delta>0$,
$$
aM-(K_X+\Delta+\delta G)\sim_\Q H+G-\delta G= \delta H+ (1-\delta)(H+G)
$$ 
is ample and $(X,\Delta+\delta G)$ is lc. Perturbing the coefficients of $\Delta+\delta G$ we can then produce a 
boundary $\Gamma$ such that $(X,\Gamma)$ is plt, 
$S=\rddown{\Gamma}\subseteq \rddown{B}$ 
and such that $a M-(K_X+\Gamma)$ is ample. Then apply Proposition \ref{l-bnd-comp-Gamma}. 
From now on we can assume that $\Supp G$ contains some non-klt centre of $(X,\Delta)$. \\

\emph{Step 4.}
\emph{In this step we introduce yet another boundary $\Omega$ and study some of its properties.}
Let $t$ be the lc threshold of $G+\Delta-\tilde{\Delta}$ with respect to $(X,\tilde{\Delta})$.  
Let
$$
\Omega=\tilde{\Delta}+t(G+\Delta-\tilde{\Delta}).
$$
We claim that we can ensure that  
any non-klt place of $(X,\Omega)$ is a non-klt place of $(X,\Delta)$. 
Indeed let $W\to X$ be a log resolution of $(X,\Delta+G)$, and let $K_W+\Delta_W$, 
$K_W+\tilde{\Delta}_W$, $K_W+\Omega_W$, 
$G_W$ be the pullbacks of $K_X+\Delta$, $K_X+\tilde{\Delta}$, $K_X+\Omega$, $G$, respectively. 
Since $\Supp G$ contains some non-klt centre of $(X,\Delta)$, we can assume that 
some component $T$ of $\rddown{\Delta_W}$ is a component of $G_W$. 
Since $\Delta-\tilde{\Delta}$ has sufficiently small coefficients, we can assume 
that $\Delta_W-\tilde{\Delta}_W$ also has arbitrarily small coefficients, perhaps after 
replacing $\tilde{\Delta}$. In particular, the lc threshold of 
$G$ with respect to $(X,\tilde \Delta)$ is sufficiently small as $\mu_T\tilde{\Delta}_W$ can be made arbitrarily 
close to $1$. Thus $t$ is also sufficiently small.  
Moreover, we can assume that $\Delta_W-\Omega_W$ also has arbitrarily small positive 
or negative coefficients, hence $\rddown{\Omega_W}\subseteq \rddown{\Delta_W}$ 
which implies that any non-klt place of $(X,\Omega)$ is a non-klt place of $(X,\Delta)$. 
In particular, each non-klt centre of 
 $(X,\Omega)$ is mapped to $z$, by construction of $\Delta$.
 
 By construction, 
$$
 \begin{array}{l l}
aM-(K_X+\Omega)& =aM-(K_X+\tilde{\Delta}+t(G+\Delta-\tilde{\Delta}))\\
&=aM-K_X-\tilde{\Delta}-t(G+\Delta-\tilde{\Delta})\\
&=aM-(K_X+\Delta)+\Delta-\tilde{\Delta}-t(G+\Delta-\tilde{\Delta})\\
& \sim_\Q H+G+\Delta-\tilde{\Delta}-t(G+\Delta-\tilde{\Delta})\\
& =H+(1-t)(G+\Delta-\tilde{\Delta})\\
&=tH+(1-t) (H+G+\Delta-\tilde{\Delta} )
\end{array}
$$
which implies that $aM-(K_X+\Omega)$ is ample because 
$$
aM-(K_X+\tilde{\Delta})=aM-(K_X+\Delta)+\Delta-\tilde{\Delta}\sim_\Q H+G+\Delta-\tilde{\Delta}
$$
is nef and big by Step 2.\\ 

\emph{Step 5.}
\emph{In this step we produce a plt pair and apply Proposition \ref{l-bnd-comp-Gamma} to finish the proof.}
Assume that $\rddown{\Omega}\neq 0$.  Replacing $S$ we can assume it is a component of 
$\rddown{\Omega}\le \rddown{\Delta}$. Since $(X,\Delta)$ is $\Q$-factorial dlt by Step 1,  $(X,S)$ is plt. 
Thus we can produce a boundary $\Gamma$ out of $\Omega$ so that $(X,\Gamma)$ is plt, $S=\rddown{\Gamma}$ 
maps to $z$, and $a M -(K_X+\Gamma)$ is ample. We then  
apply Proposition \ref{l-bnd-comp-Gamma}.

Now assume $\rddown{\Omega}=0$. Let $(X',\Omega')$ be a $\Q$-factorial dlt model of $(X,\Omega)$. 
By Step 4, each non-klt place of $(X,\Omega)$ is a non-klt place of $(X,\Delta)$ which is in turn a non-klt 
place of $(X,B)$ as $\Delta\le B$.
Thus if we denote the pullback of $K_X+B$ to $X'$ by $K_{X'}+B'$, then 
each exceptional$/X$ divisor on $X'$ appears in $B'$ with coefficient $1$. 
Running an MMP on $K_{X'}+\rddown{\Omega'}$ over $X$ ends with $X$ because $\rddown{\Omega'}$ 
is the reduced exceptional divisor of $X'\to X$ and because $X$ is $\Q$-factorial klt. The last step of the 
MMP is a divisorial contraction $X''\to X$ contracting one prime divisor $S''$. Then 
$(X'',S'')$ is plt and $-(K_{X''}+S'')$ is ample over $X$. 

Define $\Gamma''=(1-v) \Omega''+vS''$  for some sufficiently  small $v>0$. 
Let $M''$ be the pullback of $M$. Assume $\alpha=(1-v)a$. Since $aM-(K_X+\Omega)$
is ample and since $-(K_{X''}+S'')$ is ample over $X$, 
$$
\alpha M''-(K_{X''}+\Gamma'')=(1-v)aM''-(K_{X''}+(1-v) \Omega''+vS'')
$$
$$
=(1-v)a M''-(1-v)(K_{X''}+\Omega'')-v(K_{X''}+S'')
$$
$$
=(1-v)(a M''-(K_{X''}+\Omega''))-v(K_{X''}+S'')
$$
is ample. Moreover, $(X'',\Gamma'')$ is plt and $S''=\rddown{\Gamma''}$ maps to $z$. 
If $K_{X''}+B''$ is the pullback of $K_X+B$, then we can 
replace $(X,B)$ with $(X'',B'')$ and apply Proposition \ref{l-bnd-comp-Gamma}.

\end{proof}


\section{\bf Boundedness of Fano type fibrations}

In this section we treat boundedness properties of Fano type log Calabi-Yau fibrations. 
We will frequently refer to Definition \ref{d-FT-fib} and use the notation therein.

\subsection{Numerical boundedness}
We start with bounding numerical properties.
The next statement and its proof are similar to [\ref{Jiang}, Lemma 3.2].

\begin{prop}\label{l-bnd-cy-numerical-bndness}
Let $d,r$ be natural numbers and $\epsilon$ be a positive real number. 
Assume that Theorem \ref{t-sing-FT-fib-totalspace} holds in dimension $d-1$. Then there is a 
natural number $l$ depending only on $d,r,\epsilon$ satisfying the following. 
Let $(X,B)\to Z$ be a $(d,r,\epsilon)$-Fano type fibration (as in \ref{d-FT-fib}) such that 
\begin{itemize}
\item $-(K_X+\Delta)$ is nef over $Z$ for some $\R$-divisor $\Delta\ge 0$, and 

\item $f^*A+B-\Delta$ is pseudo-effective. 
\end{itemize}
Then $lf^*A-(K_X+\Delta)$ is nef (globally). 
\end{prop}
\begin{proof}
\emph{Step 1.}
\emph{In this step we do some basic preparations and introduce some notation.}
Replacing $X$ with a $\Q$-factorialisation we can assume $X$ is $\Q$-factorial. All 
the assumptions of the lemma are preserved. Put $C:=f^*A-(K_X+B)$.
Let $m\ge 2$ be a natural number. We can write 
$$
 \begin{array}{l l}
 (m+2)f^*A-(K_X+\Delta) &  = 2f^*A+m(K_X+B)+mC-(K_X+\Delta) \\
 & = 2f^*A+m(K_X+B)-(K_X+\Delta)+mC \\
 & =2f^*A+(m-1)(K_X+B)+B-\Delta+m C \\
 & =(m-1)(K_X+B+\frac{1}{m-1}(2f^*A+B-\Delta)) +m C.
  \end{array} 
$$

On the other hand, since $f^*A+B-\Delta$ is pseudo-effective and since $X$ is of Fano type over $Z$, there is 
$$
0\le P\sim_\R 2f^*A+B-\Delta
$$
by Lemma \ref{l-div-pef-on-FT}.\\

\emph{Step 2.}
\emph{In this step we find a real number $t>0$ depending only on $d,r,\epsilon$ such that 
the non-klt locus of $(X,B+tP)$ is mapped to a finite set of closed points of $Z$.}
We can assume $\dim Z>0$ otherwise the lemma is trivial. 
Let $H\in |A|$ be a general element and let $G=f^*H$, and $g$ be the induced morphism $G\to H$. 
Let
$$
K_G+B_G:=(K_X+B+G)|_G.
$$
By definition of $G$, we have $B_G=B|_G$.
Then 
\begin{itemize}
\item $(G,B_G)$ is $\epsilon$-lc as $(X,B)$ is $\epsilon$-lc, 

\item  $-K_G$ is big over $H$ as $-K_G=-(K_X+G)|_G\sim -K_X|_G/H$,

\item we have 
$$
K_G+B_G \sim_\R (f^*L+f^*H)|_G \sim_\R g^*(L+H)|_H\sim_\R g^*(L+A)|_H,
$$ 
and

\item $2A|_H-(L+A)|_H$ is ample as $A-L$ is ample. 
\end{itemize}
Therefore, $(G,B_G)\to H$ is a $(d-1,2^{d-1}r,\epsilon)$-Fano type fibration.

Letting $P_G=P|_G$ and $Q_G=\Delta|_G$ 
we have 
$$
P_G+Q_G\sim_\R (2f^*A+B-\Delta)|_G+\Delta|_G=(2f^*A+B)|_G\sim g^*2A|_H+B_G.
$$
Since we are assuming Theorem \ref{t-sing-FT-fib-totalspace} in dimension $d-1$, 
we deduce that there is a real number $t>0$ depending only on $d,r,\epsilon$ such that 
 $(G,B_G+tP_G)$ is klt. Note that here we used the assumption $\Delta\ge 0$ to ensure that 
$g^*2A|_H+B_G-P_G$ is pseudo-effective.

By inversion of adjunction [\ref{kollar-mori}, Theorem 5.50] (which is stated for $\Q$-divisors but 
also holds for $\R$-divisors) and by the previous paragraph, 
$$
(X,B+G+tP)
$$ 
is plt near $G$. Since $G$ is a general member of $|f^*A|$, we deduce that the non-klt locus of 
$(X,B+tP)$ (possibly empty) is mapped to a finite set of closed points of $Z$.\\ 

\emph{Step 3.}
\emph{In this step we consider $(m+2)f^*A-(K_X+\Delta)$-negative extremal rays.}
Fix a natural number $m\ge 2$ so that $\frac{1}{m-1}<t$. Let 
$\Theta=B+\frac{1}{m-1}P$. 
By Step 1 and definition of $P$, we have 
$$ 
(m+2)f^*A-(K_X+\Delta)=(m-1)(K_X+B+\frac{1}{m-1}(2f^*A+B-\Delta)) +m C
$$
$$
\sim_\R (m-1)(K_X+B+\frac{1}{m-1}P) +m C=(m-1)(K_X+\Theta)+mC.
$$

Assume that $R$ is an extremal ray of $X$ with 
$$
((m+2)f^*A-(K_X+\Delta))\cdot R<0. 
$$
Since $-(K_X+\Delta)$ is nef over $Z$, $R$ is not vertical over $Z$, that is, $f^*A\cdot R>0$. 
On the other hand, by the previous paragraph,
$$
((m-1)(K_X+\Theta) +m C)\cdot R<0
$$
which in turn implies 
$$
(K_X+\Theta)\cdot R<0
$$
because $C\sim_\R f^*(A-L)$ is nef.\\
 
\emph{Step 4.}
\emph{In this step we apply boundedness of length of extremal rays and finish the proof.}
Let $T$ be the non-klt locus of $(X,\Theta)$. By Step 2, $T$ is mapped to a finite set of closed points of $Z$. 
Let $V$ be the image of $\overline{NE}(T)\to \overline{NE}(X)$ where by convention we put $\overline{NE}(T)=0$ 
if $T$ is zero-dimensional  or empty. Then $V\cap R=0$ as $f^*A$ intersects 
$R$ positively but intersects every class in $V$ trivially. 
Therefore, by [\ref{Fujino-rays}, Theorem 1.1(5)], 
$R$ is generated by a curve $\Gamma$ with 
$$
-2d\le (K_X+\Theta)\cdot \Gamma,
$$
hence 
$$
-2d(m-1)\le (m-1)(K_X+\Theta)\cdot \Gamma\le ((m-1)(K_X+\Theta) +m C)\cdot \Gamma
$$
$$
=((m+2)f^*A-(K_X+\Delta))\cdot \Gamma.
$$
Moreover, $f^*A\cdot \Gamma \ge 1$.
Therefore, taking $l=m+2+2d(m-1)$ we have 
$$
0\le -2d(m-1)+ 2d(m-1)f^*A\cdot \Gamma \le (lf^*A-(K_X+\Delta))\cdot R
$$
ensuring that $lf^*A-(K_X+\Delta)$ is nef.

\end{proof}

\subsection{Bounded very ampleness}

\begin{lem}\label{l-bnd-cy-fib-v-ampleness}
Let $d,r$ be natural numbers, $\epsilon$ be a positive real number, and $\mathfrak{R}\subset [0,1]$ be a 
finite set of rational numbers. Assume that Theorem \ref{t-sing-FT-fib-totalspace} holds in dimension $d-1$ 
and that Theorem \ref{t-log-bnd-cy-fib} holds in dimension $d$. 
Then there exist natural numbers $l,m$ depending only on $d,r,\epsilon,\mathfrak{R}$ satisfying the following. 
If  $(X,B)\to Z$ is a $(d,r,\epsilon)$-Fano type fibration (as in \ref{d-FT-fib}) such that  
\begin{itemize}
\item we have  $\Delta\le B$ with coefficients in $\mathfrak{R}$, and

\item $-(K_X+\Delta)$ is ample over $Z$,
\end{itemize} 
then 
$$
m(lf^*A-(K_X+\Delta))
$$ 
is very ample.
\end{lem}

\begin{proof}
Since we are assuming Theorem \ref{t-log-bnd-cy-fib}, $(X,\Delta)$ is log bounded. Thus by 
Lemma \ref{l-bnd-Cartier-index-family}, there is a bounded natural number $m$ such that 
$m(K_X+\Delta)$ is Cartier. On the other hand, by Proposition \ref{l-bnd-cy-numerical-bndness}, 
there is a bounded natural number $l$ such that $lf^*A-(K_X+\Delta)$ is nef. As $-(K_X+\Delta)$ 
is ample over $Z$, replacing $l$ with $2l+1$ we can assume that $lf^*A-2(K_X+\Delta)$ is ample 
(which also implies that $lf^*A-(K_X+\Delta)$ is ample). 

Now  
$$
m(lf^*A-(K_X+\Delta))=K_X+\Delta+(m(lf^*A-(K_X+\Delta))-(K_X+\Delta))
$$ 
where the term 
$$
(m(lf^*A-(K_X+\Delta))-(K_X+\Delta))=m(lf^*A-(1+\frac{1}{m})(K_X+\Delta))
$$
is ample because $1+\frac{1}{m}\le 2$ and $lf^*A-2(K_X+\Delta)$ is ample. 
Thus replacing $m$ we can assume 
$$
|m(lf^*A-(K_X+\Delta))|
$$
is base point free, by the effective base point free theorem [\ref{kollar-ebpf}, Theorem 1.1]. Therefore, 
$$
(d+4)m(lf^*A-(K_X+\Delta))
$$
is very ample by [\ref{kollar-ebpf}, Lemma 1.2]. Now replace $m$ with $(d+4)m$.

\end{proof}

\subsection{Effective birationality}

The next statement and its proof are somewhat similar to [\ref{Jiang}, Lemma 3.3].

\begin{prop}\label{l-bnd-cy-effective-bir}
Let $d,r$ be natural numbers and $\epsilon$ be a positive real number.
Assume that Theorems \ref{t-log-bnd-cy-fib} and \ref{t-sing-FT-fib-totalspace} hold in dimension $d-1$. Then there exist 
natural numbers $l,m$ depending only on $d,r,\epsilon$ satisfying the following. If $(X,B)\to Z$ is a 
$(d,r,\epsilon)$-Fano type fibration (as in \ref{d-FT-fib}), then the linear system $|m(lf^*A-K_X)|$ defines 
a birational map. 
\end{prop}
\begin{proof}
We first replace $X$ with a $\Q$-factorialisation so that $K_X$ is $\Q$-Cartier. 
Then after running an MMP on $-K_X$ over $Z$ 
we can assume $-K_X$ is nef and big over $Z$. Replacing $X$ with the ample model of $-K_X$ over $Z$ 
we can assume $-K_X$ is ample over $Z$  ($X$ may no longer be $\Q$-factorial but 
we do not need it any more). If $\dim Z=0$, then the proposition holds by [\ref{B-compl}, Theorem 1.2]. 
We then assume $\dim Z>0$.
By Proposition \ref{l-bnd-cy-numerical-bndness}, there is $l\in \N$  depending only on $d,r,\epsilon$  such 
 that $lf^*A-K_X$ is nef.
Since $-K_X$ is ample over $Z$, replacing $l$ with $l+1$ we can assume $lf^*A-K_X$ is ample.

By Lemma \ref{l-gen-element-v.ample-div-passing-x,y-contraction}, there is a 
non-empty open subset $U\subseteq Z$ such that for any pair of closed points $z,z'\in U$ and any general 
member $H$ of the sub-linear system $L_{z,z'}$ of $|2A|$ consisting of elements passing through $z,z'$, 
the pullback $G=f^*H$ is normal and $(G,B_G)$ is an $\epsilon$-lc pair where 
$$
K_G+B_G=(K_X+B+G)|_G.
$$

Pick distinct closed points $x,x'\in X$ such that $z:=f(x)$ and $z':=f(x')$ are in $U$. 
Let $H,G,B_G$ be as in the previous paragraph constructed for $z,z'$.
Denoting $G\to H$ by $g$ we then have 
\begin{itemize}
\item $(G,B_G)$ is $\epsilon$-lc, 

\item $K_G+B_G\sim_\R (f^*L+G)|_G\sim g^*(L+2A)|_H$, 

\item $-K_G=-(K_X+G)|_G$ is ample over $H$, 

\item the divisor 
$$
3A|_H-(L+2A)|_H=(A-L)|_H
$$ 
is ample, and 

\item the divisor 
$$
(l+4)g^*A|_H-K_G\sim (l+4)f^*A|_G-(K_X+G)|_G\sim ((l+2)f^*A-K_X)|_G
$$ 
is ample.
\end{itemize}
In particular, $(G,B_G)$ is a $(d-1, 2(3^{d-1})r, \epsilon)$-Fano type fibration.
Applying Lemma \ref{l-bnd-cy-fib-v-ampleness} and perhaps replacing $l$, we can assume 
that 
$$
m((l+4)g^*A|_H-K_G)
$$ 
is very ample, for some bounded natural number $m\ge 2$.

On the other hand, by assumption  
$$
C:=f^*A-(K_X+B)\sim_\R f^*(A-L)
$$ 
is nef, and we can write 
$$
m((l+2)f^*A-K_X)-G=2mf^*A+m(lf^*A-K_X)-G
$$
$$
\sim_\R (2m-2)f^*A+m(lf^*A-K_X)
$$
$$
\sim_\R K_X+B+C+(2m-3)f^*A+m(lf^*A-K_X).
$$
Thus by the Kawamata-Viehweg vanishing theorem,
$$
h^1(m((l+2)f^*A-K_X)-G)=0,
$$ 
hence the map 
$$
H^0(m((l+2)f^*A-K_X)) \to H^0(m((l+2)f^*A-K_X)|_G)
$$
is surjective. Recall that by the previous paragraph, 
$$
H^0(m((l+2)f^*A-K_X)|_G) = H^0(m((l+4)g^*A|_H-K_G)).
$$

Now since  $m((l+4)g^*A|_H-K_G)$ 
is very ample, we can find a section in 
$$
H^0(m((l+4)g^*A|_H-K_G))
$$ 
vanishing at $x$ but not at $x'$ (and vice versa).
This in turn gives a section in 
$$
H^0(m((l+2)f^*A-K_X))
$$ 
vanishing at $x$ but not at $x'$ (and vice versa). 
Therefore, $|m((l+2)f^*A-K_X)|$ defines a birational map. Now replace $l$ with $l+2$.

\end{proof}

\subsection{Boundedness of volume}

The next statement and its proof are similar to [\ref{Jiang}, Theorem 4.1] (also see [\ref{DiCerbo-Svaldi}, 4.1]).

\begin{prop}\label{l-bnd-cy-bndness-volume}
Let $d,r,l$ be natural numbers and $\epsilon$ be a positive real number.
Then there exists  
a natural number $v$ depending only on $d,r,l,\epsilon$ satisfying the following. 
If $(X,B)\to Z$ is a $(d,r,\epsilon)$-Fano type fibration (as in \ref{d-FT-fib}), then $\vol(lf^*A-K_X)\le v$. 
\end{prop}
\begin{proof}
We can assume that $lf^*A-K_X$ is big otherwise $\vol(lf^*A-K_X)=0$. 
Moreover, taking a $\Q$-factorialisation, we can assume $X$ is $\Q$-factorial. 
If $\dim Z=0$, then $X$ belongs to a bounded family  [\ref{B-BAB}, Corollary 1.2] 
so the proposition follows. We then assume $\dim Z>0$.

Let $p$ be the largest integer such that $pf^*A-K_X$ is not big. 
Then $p<l$.  Let $H\in |A|$ be a general element and $G=f^*H$. Then 
$$
lf^*A-K_X=pf^*A-K_X+(l-p)f^*A\sim pf^*A-K_X+(l-p)G,
$$
so by [\ref{Jiang}, Lemma 2.5], 
$$
\vol(lf^*A-K_X)\le \vol(pf^*A-K_X)+d(l-p)\vol((lf^*A-K_X)|_G).
$$
Since $pf^*A-K_X$ is not big, $\vol(pf^*A-K_X)=0$, so it is then enough to bound 
$$
d(l-p)\vol(lf^*A-K_X)|_G)
$$
from above. 

Define $K_G+B_G=(K_X+B+G)|_G$. Then $(G,B_G)$ is $\epsilon$-lc,  $-K_G$ is big over $H$, 
 $K_G+B_G\sim_\R g^*(L+A)|_H$ where $g$ denotes $G\to H$, and $2A|_H-(L+A)|_H$ is ample. 
  Thus $(G,B_G)\to H$ is a 
$(d-1, 2^{d-1}r, \epsilon)$-Fano type fibration, hence  applying induction on dimension shows that 
$$
\vol((lf^*A-K_X)|_G)=\vol(((l+1)f^*A-(K_X+G))|_G)
$$
$$
=\vol((l+1)g^*A|_H-K_G)\le \vol((l+1)g^*2A|_H-K_G)
$$
is bounded from above. Therefore, it is enough to show that $l-p$ is bounded from above 
which is equivalent to showing that $p$ is bounded from below. 

By definition of $p$, $(p+1)f^*A-K_X$ is big. Thus we can find 
$$
0\le R\sim_\Q  (p+1)f^*A-K_X.
$$
Let $F$ be a general fibre of $f\colon X\to Z$. Then
$$
K_F+B_F:=(K_X+B)|_F\sim_\R 0
$$ 
and
$$
R_F:=R|_F\sim_\Q -K_X|_F=-K_F\sim_\R B_F.
$$ 

By [\ref{B-BAB}, Corollary 1.2], 
$F$ belongs to a bounded family of varieties as $(F,B_F)$ is $\epsilon$-lc and $-K_F$ is big. 
We can then find a very ample divisor 
$J_F$ on $F$ with bounded degree $J_F^{\dim F}$ such that 
$$
J_F-R_F \sim_\R J_F-B_F \sim_\Q J_F+K_F  
$$ 
is ample. Therefore, by [\ref{B-BAB}, Theorem 1.6] (=Theorem \ref{t-bnd-lct}), 
the pair $(F,B_F+tR_F)$ is klt for some real number $t\in(0,1)$ depending only 
on $d,J_F^{\dim F},\epsilon$. In particular, letting 
$$
\Delta:=(1-t)B+tR ~~\mbox{ and}~~ \Delta_F:=\Delta|_F,
$$ 
we see that 
$$
(F,\Delta_F=(1-t)B_F+tR_F)
$$ 
is klt. Therefore, $(X,\Delta)$ is klt near the generic fibre of $f$.
By construction, 
$$
K_X+\Delta= K_X+(1-t)B+tR\sim_\R K_X+(1-t)B+t(p+1)f^*A-tK_X
$$
$$
= (1-t)(K_X+B)+t(p+1)f^*A \sim_\R f^*((1-t)L+t(p+1)A).
$$

Now by adjunction we can write 
$$
K_X+\Delta\sim_\R f^*(K_Z+\Delta_Z+M_Z)
$$
where $\Delta_Z$ is the discriminant divisor and $M_Z$ is the moduli divisor (see \ref{fib-adj-setup} below for more 
details on adjunction). In particular,
$$
(1-t)L+t(p+1)A\sim_\R K_Z+\Delta_Z+M_Z.
$$
By [\ref{B-compl}, Theorem 3.6], $M_Z$ is pseudo-effective. 
On the other hand, if $\psi\colon Z'\to Z$ is a resolution, then $K_{Z'}+3d\psi^* A$ is big [\ref{B-compl}, Lemma 2.46], 
hence $K_Z+3dA$ is big. 
This implies that 
$$
K_Z+\Delta_Z+M_Z+3dA
$$
is big. But then  
$$
(1-t)L+t(p+1)A+3dA
$$
is big which in turn implies that 
$$
(1-t)A+t(p+1)A+3dA=(1-t)(A-L)+(1-t)L+t(p+1)A+3dA
$$
is big as $A-L$ is ample. Therefore, $(1-t)+t(p+1)+3d>0$ which implies that $p$ is bounded from below as required. 

\end{proof}

\subsection{Log birational boundedness}

\begin{prop}\label{l-bnd-cy-bir-bnd}
Let $d,r,l$ be natural numbers and $\epsilon,\delta$ be positive real numbers.
Assume that Theorems \ref{t-log-bnd-cy-fib} and \ref{t-sing-FT-fib-totalspace} hold in 
dimension $d-1$. Then there exists a bounded set of couples $\mathcal{P}$  depending only 
on $d,r,l,\epsilon,\delta$ satisfying the following. 
Assume that $(X,B)\to Z$ is a $(d,r,\epsilon)$-Fano type fibration (as in \ref{d-FT-fib}) 
and that $\Lambda\ge 0$ is an $\R$-divisor such that 
\begin{itemize}
\item each non-zero coefficient of $\Lambda$ is $\ge \delta$, and 

\item $lf^*A-(K_X+\Lambda)$ is pseudo-effective.
\end{itemize}
Then there exist a couple
$(\overline{X}, \overline{\Sigma})$, a very ample divisor  $\overline{D}\ge 0$ on ${\overline{X}}$ and 
a birational map $\rho\colon \overline{X} \bir X /Z$ such that 
\begin{enumerate}
\item $(\overline{X}, \overline{\Sigma}+\overline{D})$ is log smooth and belongs to $\mathcal{P}$, 

\item $\overline{\Sigma}$ contains the exceptional divisors of $\rho$ union 
the birational transform of $\Supp \Lambda$, 

\item and if $N:=lf^*A-K_X$ is nef and $\overline{N}:=\rho^*N$ (defined as in \ref{ss-divisors}), 
then $\overline{D}-\overline{N}$ is ample. 
\end{enumerate}
\end{prop}
\begin{proof}
\emph{Step 1.}
\emph{In this step we introduce some notation.}
Taking a $\Q$-factorialisation we can assume $X$ is $\Q$-factorial. 
By Proposition \ref{l-bnd-cy-effective-bir}, 
perhaps after replacing $l$ with a bounded multiple,  we can assume that 
there exists a natural number 
$m$ depending only on $d,r,\epsilon$ such that the linear system $|m(lf^*A-K_X)|$ defines a birational map. Pick 
$$
0\le M\sim m(lf^*A-K_X).
$$
Take a log resolution 
$\phi\colon W\to X$ of $(X,\Lambda+M)$ such that $|\phi^*M|$ decomposes as the 
sum of a free part $|F_W|$ plus fixed 
part $R_W$ (cf. [\ref{B-compl}, Lemma 2.6]). 
Let $\Sigma_W$ be the sum of the reduced exceptional divisor of $\phi$ and the birational transform of $\Supp (\Lambda+M)$ 
plus a general element $G_W$ of $|\phi^*f^*A+F_W|$. 
Let $\Sigma,F,R,G$ be the pushdowns of $\Sigma_W,F_W,R_W, G_W$ to $X$.\\ 

\emph{Step 2.}
\emph{In this step we show that $\vol(K_W+\Sigma_W+(4d+2)G_W)$ is bounded from above.}
From the definition of $\Sigma_W$ and the assumption that the non-zero coefficients of 
$\Lambda$ are $\ge \delta$, we can see that 
$$
\Sigma =\Supp (\Lambda+M)+G \le \frac{1}{\delta}\Lambda+M+G.
$$
Moreover, 
$$
G+R\sim f^*A+F+R\sim f^*A+M,
$$
 and by assumption  
$lf^*A-(K_X+\Lambda)$  is pseudo-effective. Taking all these into account and  
letting $p=4d+2$ and $q=\frac{1}{\delta}$ we then have 
$$
 \begin{array}{l l}
\vol(K_W+\Sigma_W+pG_W) & \le \vol(K_X+\Sigma+pG) \\
& \le \vol(K_X+q\Lambda+M+(p+1)G) \\
& \le \vol(K_X+q\Lambda+M+(p+1)f^*A +(p+1)M)\\
& = \vol(K_X+q\Lambda+(p+1)f^*A +(p+2)M)\\
& = \vol((1-q)K_X+q(K_X+\Lambda)+(p+1)f^*A+(p+2)M)\\
& \le \vol((1-q)K_X+qlf^*A+(p+1)f^*A+(p+2)M)\\
& = \vol((1-q)K_X+(ql+p+1)f^*A+(p+2)M)\\
& =\vol((ql+p+1)f^*A+(p+2)mlf^*A-((p+2)m+q-1)K_X).
\end{array}
$$
The latter volume is bounded from above by Proposition  \ref{l-bnd-cy-bndness-volume} 
as all the numbers $l,m,p,q$ are fixed.  Thus 
$\vol(K_W+\Sigma_W+pG_W)$ is bounded from above.\\ 

\emph{Step 3.}
\emph{In this step we show that $(X,\Supp (\Lambda+M))$ is log birationally bounded.}
Since $G_W\sim \phi^*f^*A+ F_W$, $|G_W|$ is base point free and it defines a birational contraction $W\to \overline{X}'$.
In particular,
$$
K_W+\Sigma_W+(p-1)G_W
$$ 
is big (cf. [\ref{B-compl}, Lemma 2.46]), hence 
$$
\vol(G_W)\le \vol(K_W+\Sigma_W+pG_W)
$$ 
which implies that the left hand side is also bounded from above. 
 Moreover,  by [\ref{HMX1}, Lemma 3.2], 
$\Sigma_W\cdot G_W^{d-1}$ is bounded from above.
Therefore, if 
$ \overline{\Sigma}'$ is the pushdown of $\Sigma_W$, then $(\overline{X}', \overline{\Sigma}')$ 
is log bounded (this follows from [\ref{HMX1}, Lemma 2.4.2(4)]). 
Also the induced map ${\overline{X}'} \bir Z$ is a morphism by the choice of $G_W$: indeed 
any curve contracted by $W\to \overline{X}'$ intersects $G_W$ trivially hence it intersects the pullback of 
$A$ trivially which means the curve is also contracted over $Z$.

Now we can take a log resolution $\overline{X}\to \overline{X}'$ of $(\overline{X}', \overline{\Sigma}')$ such that if 
$\overline{\Sigma}$ is the union of the exceptional divisors and the birational transform of $\overline{\Sigma}'$, 
then $(\overline{X}, \overline{\Sigma})$ is log smooth and log bounded. By construction, 
$\overline{\Sigma}$ contains the reduced exceptional divisor of the induce map $\rho\colon \overline{X}\bir X$ union the 
birational transform of $\Supp (\Lambda+M)$. This settles (1) and (2) of the proposition except that 
we need to add $\overline{D}$.\\

\emph{Step 4.}
\emph{In this step we prove the existence of the very ample divisor  $\overline{D}$.}
Denote the induced map $\overline{X}\bir W$ by $\alpha$.
By construction, $G_{W}\sim 0/ \overline{X}'$, hence $\overline{G}:=\alpha^*G_W\sim 0/ \overline{X}'$ 
(pullback under $\alpha$ is defined as in \ref{ss-divisors}).
Moreover, $\overline{G}\le \overline{\Sigma}$. 
In particular, $(\overline{X}, \overline{G})$ is log bounded and $\overline{G}$ is big, 
hence we can find a bounded natural 
number $b$ and a very ample divisor $\overline{D}$
such that $b\overline{G}-\overline{D}$ is big. Then
$(\overline{X}, \overline{\Sigma}+\overline{D})$ is log bounded, hence it belongs to some 
fixed bounded set of couples $\mathcal{P}$.

From now on we assume that $N=lf^*A-K_X$ is nef. Then $M\sim mN$ is also nef.
We will show that we can choose $\overline{D}$ so that $\overline{D}-\overline{M}$ is ample where 
$\overline{M}=\rho^*M$. 
Let $\pi\colon V\to W$ and $\mu\colon V\to \overline{X}$ be a common resolution. 
Then 
$$
\begin{array}{l l}
\overline{D}^{d-1}\cdot \overline{M} = \mu^*\overline{D}^{d-1}\cdot \mu^* \overline{M}
& = \mu^*\overline{D}^{d-1}\cdot \pi^*\phi^*M\\
& \le \vol(\mu^*\overline{D}+\pi^*\phi^*M)\\
& \le  \vol(b\mu^*\overline{G}+\pi^*\phi^*M)\\
&=\vol(b\pi^* G_W+\pi^*\phi^*M)\\
& =\vol(bG_W+\phi^*M)\\
& \le \vol(bG+M)\\
& \le \vol(bf^*A+bM+M)\\
&=\vol(bf^*A+(b+1)mlf^*A-(b+1)mK_X)
\end{array}
$$
where to get the second equality we use the fact $\overline{M}=\mu_*(\pi^*\phi^*M)$ and to get the 
first inequality we have used the fact that $\mu^*\overline{D},\pi^*\phi^*M$ are both nef. 
Therefore, $\overline{D}^{d-1}\cdot \overline{M}$ is bounded from above as the latter volume is bounded 
from above by Proposition \ref{l-bnd-cy-bndness-volume}.
This implies that the coefficients of $\overline{M}$ are bounded from above. 

Now since $\Supp \overline{M}\le \overline{\Sigma}$, $(\overline{X}, \Supp \overline{M})$ 
 is log bounded. Thus replacing $\overline{D}$ we can assume that $\overline{D}-\overline{M}$ is ample.
Finally, note that  $\overline{N}\sim_\Q \frac{1}{m} \overline{M}$, hence  
$$
\overline{D}-\overline{N}\sim_\Q \overline{D}-\frac{1}{m} \overline{M}
=\frac{1}{m}(m\overline{D}- \overline{M})
$$ 
is ample.  

\end{proof}

\subsection{Lower bound on lc thresholds: special case}

We prove a special case of Theorem \ref{t-sing-FT-fib-totalspace} which is crucial for the 
rest of this section.

\begin{prop}\label{l-bnd-cy-bnd-lct-special}
Let $d,r,l$ be natural numbers and $\epsilon$ be a positive real number.
Assume that Theorems \ref{t-log-bnd-cy-fib} and \ref{t-sing-FT-fib-totalspace}
hold in dimension $d-1$. Then there exists  
a positive real number $t$ depending only on $d,r,l,\epsilon$ satisfying the following. 
Assume that $(X,B)\to Z$ is a $(d,r,\epsilon)$-Fano type fibration (as in \ref{d-FT-fib}) and that 
$0\le P\sim_\R lf^*A$. Then $(X,B+tP)$ is klt. 
\end{prop}

We prove a lemma before proving the proposition.

\begin{lem}\label{l-bnd-cy-bnd-lct-special-I}
Let $d,r,l,n$ be natural numbers and $\epsilon$ be a positive real number.
Assume that Theorems \ref{t-log-bnd-cy-fib} and \ref{t-sing-FT-fib-totalspace}
hold in dimension $d-1$. Then there exists  
a natural number $v$ depending only on $d,r,l,n,\epsilon$ satisfying the following. 
Assume that $(X,B)\to Z$ is a $(d,r,\epsilon)$-Fano type fibration (as in \ref{d-FT-fib}) and that 
\begin{itemize}
\item $0\le P\sim_\R lf^*A$,

\item $T$ is a prime divisor over $X$,

\item $(X,\Lambda)$ is lc over a neighbourhood of $z$ where $\Lambda\ge 0$ and 
$z$ is the generic point of the image of $T$ on $Z$,

\item $a(T,X,B)\le 1$ and $a(T,X,\Lambda)=0$, and

\item  $n(K_X+\Lambda)\sim (n+2)lf^*A$.
\end{itemize}
Then $\mu_TP\le v$.
\end{lem}
\begin{proof}
\emph{Step 1.}
\emph{In this step we discuss log birational boundedness of $(X,\Supp \Lambda)$.} 
Taking a $\Q$-factorialisation we can assume that $X$ is $\Q$-factorial. 
After running an MMP on $-K_X$ over $Z$, we can assume that $-K_X$ is nef over $Z$. 
By Lemma \ref{l-bnd-cy-numerical-bndness}, $qf^*A-K_X$ is globally nef for some bounded natural number $q$. 
By Lemma \ref{l-bnd-cy-bir-bnd}, $(X,\Supp \Lambda)$ is log birationally bounded,
that is, there exist a couple $(\overline{X}, \overline{\Sigma})$, a very ample divisor $\overline{D}$, and  
a birational map $\rho\colon \overline{X}\bir X/Z$ such that 
\begin{itemize}
\item $(\overline{X}, \overline{\Sigma}+\overline{D})$ is log smooth and belongs to a bounded set of couples $\mathcal{P}$, 

\item $\overline{\Sigma}$ contains the exceptional divisors of $\rho$ union 
the birational transform of $\Supp \Lambda$, 

\item and if  $N:=qf^*A-K_X$ and $\overline{N}=\rho^*N$, then $\overline{D}-\overline{N}$ is ample. 
\end{itemize}
Let $K_{\overline{X}}+\overline{B}$ and $K_{\overline{X}}+\overline{\Lambda}$ be the 
pullbacks of $K_X+B$ and $K_X+\Lambda$, respectively. Since $(X,\Lambda)$ is lc over $z$ 
and $K_X+\Lambda\sim_\Q 0/Z$, 
$(\overline{X}, \overline{\Lambda})$ is  sub-lc over $z$. Moreover, from 
$
a(T,{X}, \Lambda)=0
$ 
we get 
$
a(T,\overline{X}, \overline{\Lambda})=0.
$ 
 But then since 
$\Supp  \overline{\Lambda} \subseteq \overline{\Sigma}$, we have $\overline{\Lambda} \le \overline{\Sigma}$ 
over $z$, hence $T$ is also an lc place of $(\overline{X}, \overline{\Sigma})$, 
that is, 
$$
a(T,\overline{X}, \overline{\Sigma})=0.
$$\ 

\emph{Step 2.}
\emph{In this step we study numerical properties of $\overline{D}$.}
Since $(\overline{X}, \overline{\Sigma})$ is log bounded, we can assume that 
$\overline{D}-\overline{\Sigma}$
is ample. Moreover, by adding a general element of $|(n+2)f^*A|$ to $\Lambda$ we can assume that some element of 
$|(n+2)\overline{f}^*A|$ is a component of $\Sigma_{\overline{X}}$ where $\overline{f}$ denotes $\overline{X}\to Z$; 
this requires replacing $l$ with 
$l+n$ to preserve the condition $n(K_X+\Lambda)\sim (n+2)lf^*A$, and replacing $P$ accordingly. 
Thus we can also assume that $\overline{D}-(n+2)\overline{f}^*A$ is ample and that $\overline{D}-\overline{P}$ 
is ample where $\overline{P}=\rho^*P$.

Now
$$
\overline{D}-(K_{\overline{X}}+\overline{B})
\sim_\R \overline{D}-\overline{f}^*A+\overline{f}^*A-(K_{\overline{X}}+\overline{B})
$$
$$
\sim_\R \overline{D}-\overline{f}^*A+ \overline{f}^*(A-L)
$$ 
is ample. In addition we can assume 
 $\overline{D}+K_{\overline{X}} $ is ample as well, hence replacing $\overline{D}$ with  $2\overline{D}$
we can assume that  $\overline{D}-\overline{B}$ is ample.\\ 

\emph{Step 3.}
\emph{In this step we prove the lemma assuming that the coefficients of $\overline{B}$ are bounded from below.} 
That is, assume that the coefficients of $\overline{B}$ are $\ge p$ for some fixed integer $p$.
Under this assumption, there is $c\in (0,1)$ depending only on $p$ such that 
$$
\overline{\Delta}:=c\overline{B}+(1-c)\overline{\Sigma}\ge 0
$$
because the components of $\overline{B}$ with negative coefficients are exceptional over $X$, hence 
are components of $\overline{\Sigma}$.
In particular, since $(\overline{X},\overline{B})$ is sub-$\epsilon$-lc, 
$(\overline{X},\overline{\Delta})$ is an $c\epsilon$-lc pair. Moreover, 
$$
a(T,\overline{X},\overline{\Delta})=ca(T,\overline{X},\overline{B})+(1-c)a(T,\overline{X},\overline{\Sigma})
=ca(T,\overline{X},\overline{B})<1.
$$
In addition,
$$
\overline{D}-\overline{\Delta}=\overline{D}-c\overline{B}-(1-c)\overline{\Sigma}
=c(\overline{D}-\overline{B})+(1-c)(\overline{D}-\overline{\Sigma})
$$
is ample. 

Now applying [\ref{B-BAB}, Theorem 1.6] (=Theorem \ref{t-bnd-lct}), we deduce that 
$(\overline{X},\overline{\Delta}+t\overline{P})$ is klt for some real number $t>0$ bounded away from zero. 
Therefore, $\mu_TP=\mu_T\overline{P}$ is bounded from above because 
$$
0\le a(T,\overline{X},\overline{\Delta}+t\overline{P})
=a(T,\overline{X},\overline{\Delta})-t\mu_T\overline{P}<1-t\mu_T\overline{P}.
$$\

\emph{Step 4.}
\emph{Finally it is enough to show that the coefficients of $\overline{B}$ are bounded from below.} 
Define $K_{\overline{X}}+\overline{E}=\rho^*K_X$. It is enough to show that 
the coefficients of $\overline{E}$ are bounded from below because $\overline{E}\le \overline{B}$. Write 
$\overline{E}$ as the difference $\overline{E}^+ - \overline{E}^-$ where $\overline{E}^+, \overline{E}^-\ge 0$ 
have no common components. 
Observe that 
$$
\overline{N}=\rho^*(qf^*A-K_X)= q\overline{f}^*A-(K_{\overline{X}}+\overline{E})
=q\overline{f}^*A-(K_{\overline{X}}+\overline{E}^+)+\overline{E}^-,
$$
hence 
$$
2\overline{D}-\overline{E}^-\sim_\R \overline{D}-\overline{N}+\overline{D}-(K_{\overline{X}}+\overline{E}^+)+q\overline{f}^*A.
$$
By Step 1, $\overline{D}-\overline{N}$ is ample and $\overline{E}^+\le \overline{\Sigma}$.
Replacing $\overline{D}$ with a multiple we can assume that $\overline{D}-(K_{\overline{X}}+\overline{E}^+)$ is ample.
Thus $2\overline{D}-\overline{E}^-$ is ample which implies that $\overline{D}^{d-1}\cdot\overline{E}^-$ is bounded 
from above, hence the coefficients of 
$\overline{E}^-$ are bounded from above which in turn implies that the coefficients of $\overline{E}$ 
are bounded from below as required.

\end{proof}

\begin{proof}(of Proposition \ref{l-bnd-cy-bnd-lct-special})
\emph{Step 1.}
\emph{In this step we will translate the problem into showing that the multiplicity of $P$ along certain 
divisors is bounded from above.} First we can assume 
$P\neq 0$ otherwise the statement is trivial. In particular,  $\dim Z>0$.
Taking a $\Q$-factorialisation we can assume $X$ is $\Q$-factorial.
Pick a small $\epsilon'\in (0,\epsilon)$. Let $s$ be the $\epsilon'$-lc threshold of $P$ with respect to 
$(X,B)$, that is, $s$ is the largest number such that $(X,B+sP)$ is $\epsilon'$-lc. It is enough to show that 
$s$ is bounded from below away from zero. In particular, we can assume $s<1$.

There is a prime divisor $T$ over $X$ with log discrepancy
$$
a(T,X,B+sP)=\epsilon'<1.
$$ 
Since $P$ is vertical over $Z$, $T$ is vertical over $Z$.
It is enough to show that $\mu_TP$, the coefficient of $T$ in the pullback of $P$ on any resolution, 
is bounded from above because 
$$
s\mu_TP=a(T,X,B)-a(T,X,B+sP)\ge \epsilon-\epsilon'.
$$
We devote the rest of the proof to showing that $\mu_TP$ is bounded from above.\\ 

\emph{Step 2.}
\emph{In this step we apply induction and reduce to the case when $T$ maps to a 
closed point on $Z$.} By the choice of $P$, 
$$
K_X+B+sP\sim_\R f^*(L+slA),
$$
and since $s<1$, $(l+1)A-(L+slA)$ is ample. Thus 
replacing $B$ with $B+sP$, replacing $A$ with $(l+1)A$ (and replacing $r$ accordingly), 
and replacing $\epsilon$ with $\epsilon'$, we can assume 
that $\epsilon$ is sufficiently small and that  $a(T,X,B)=\epsilon$ (we will not use $s$ any more). 
Extracting $T$ we can also assume $T$ is a divisor on $X$. 
Our goal still is to show that $\mu_TP$ is bounded from above.

Take a hyperplane section $H\sim A$ of $Z$ and let $G=f^*H$. Consider
$$
K_G+B_G:=(K_X+B+G)|_G
$$ 
and $P_G:=P|_G$. Then $(G,B_G)$ is $\epsilon$-lc, $-K_G$ is big over $H$, 
$K_G+B_G\sim_\R g^*(L+A)|_H$ where $g$ denotes $G\to H$, and $2A|_H-(L+A)|_H$ is ample. 
Thus $(G,B_G)\to H$ is a $(d-1, 2^{d-1}r, \epsilon)$-Fano type fibration.
Moreover, $2P_G\sim_\R lg^*2A|_H$.
Applying induction on dimension we find a real number $u>0$ depending only on $d,r,l,\epsilon$ 
such that $(G,B_G+uP_G)$ is klt. Then by inversion of adjuction  [\ref{kollar-mori}, Theorem 5.50] (which is 
stated for $\Q$-divisors but also holds for $\R$-divisors),
the pair $(X,B+G+uP)$ is plt near $G$. In particular, if the image of $T$ on $Z$ is positive-dimensional, 
then $T$ intersects $G$, so $\mu_TP$ is bounded from above. 
Therefore, we can assume that the image of $T$ on $Z$ is a closed point.\\ 

\emph{Step 3.}
\emph{In this step we finish the proof by applying Lemma \ref{l-bnd-cy-bnd-lct-special-I}.}
Let $\Theta=T$. Since $\Theta$ is vertical over $Z$, $-(K_X+\Theta)$ is big over $Z$.
Run an MMP on $-(K_X+\Theta)$ over $Z$ 
and let $X'$ be the resulting model. We denote the pushdown of each divisor $D$ to $X'$ 
by $D'$. Then $-(K_{X'}+\Theta')$ is nef and big over $Z$.
By construction, $-\epsilon T'\le B'-\Theta'$, hence since $T'$ is mapped to a closed point on $Z$, 
$(f^*A)'+B'-\Theta'$ is pseudo-effective. Thus   
by Proposition \ref{l-bnd-cy-numerical-bndness}, we can assume that $(lf^*A)'-(K_{X'}+\Theta')$ is nef for some 
bounded natural number $l$. Increasing $l$ by $1$ we can assume $(lf^*A)'-(K_{X'}+\Theta')$ is nef and big. 

 On the other hand, $(X',\Theta'-\epsilon T')$ is klt as $\Theta'-\epsilon T'\le B'$. 
Since $\epsilon$ is assumed to be sufficiently small, by the ACC for 
lc thresholds [\ref{HMX2}, Theorem 1.1], $(X',\Theta')$ is lc.
Then applying Theorem \ref{t-bnd-comp-lc-global} (by taking $B=\Theta'$, $M=(lf^*A)'$ and $S$ to be the 
centre of $T$ on $X'$), there exist  
a bounded natural numbers $n$ and $\Lambda'\ge \Theta'$ such that $(X',\Lambda')$ is lc over $z$ and 
$n(K_{X'}+\Lambda')\sim (n+2)(lf^*A)'$. Since $X\bir X'$ is an MMP on $-(K_X+\Theta)$ over $Z$, 
taking $K_X+\Lambda$ to be the crepant pullback of 
$K_{X'}+\Lambda'$ to $X$ we get $\Lambda\ge \Theta\ge \Delta$ such that $(X,\Lambda)$ is lc over $z$ and 
$n(K_{X}+\Lambda)\sim (n+2)lf^*A$. Finally, apply Lemma \ref{l-bnd-cy-bnd-lct-special-I} 
to deduce that $\mu_TP$ is bounded.\\

\end{proof}

\subsection{Bounded klt complements}
In this subsection we treat Theorem \ref{t-bnd-comp-lc-global-cy-fib} inductively. 
We first consider a weak version.

\begin{prop}\label{l-bnd-cy-bnd-klt-compl-1}
Let $d,r$ be natural numbers, $\epsilon$ be a positive real number, and $\mathfrak{R}\subset [0,1]$ 
be a finite set of rational numbers.
Assume that Theorems \ref{t-log-bnd-cy-fib} and \ref{t-sing-FT-fib-totalspace}
hold in dimension $d-1$. Then there exist  
natural numbers $n,m$ depending only on $d,r,\epsilon,\mathfrak{R}$ satisfying the following. 
Assume that $(X,B)\to Z$ is a $(d,r,\epsilon)$-Fano type fibration (as in \ref{d-FT-fib}) and that 
\begin{itemize}
\item we have $0\le \Delta\le B$ with coefficients in $\mathfrak{R}$, and 

\item $-(K_X+\Delta)$ is big over $Z$. 
\end{itemize}
Then for each point $z\in Z$ there is a $\Q$-divisor $\Lambda\ge \Delta$ such that 
\begin{itemize}
\item $(X,\Lambda)$ is lc over $z$, and

\item $n(K_X+\Lambda)\sim mf^*A$. 
\end{itemize} 
\end{prop}

\begin{proof}
\emph{Step 1.}
\emph{In this step we create singularities over $z$.}
We can assume that $\dim Z>0$ otherwise we apply [\ref{B-compl}, Theorem 1.7].
It is enough to prove the proposition with $z$ replaced by any closed point $z'$ in the closure $\bar{z}$ because 
any open neighbourhood of $z'$ contains $z$. Thus from now on we assume that $z$ is a closed point.  
Taking a $\Q$-factorialisation we can assume $X$ is $\Q$-factorial. 
Consider the sub-linear system $V_z$ of $|f^*A|$ consisting of elements 
containing the fibre $f^{-1}\{z\}$, and pick $P$ in $V_z$.
Since $A$ is very ample, $V_z$ is base point free outside $f^{-1}\{z\}$. 
Replacing $A$ with $2A$ we can assume that $\dim V_z>0$.

Let $p$ be a natural number such that $\frac{1}{p}<1-\epsilon$. 
Pick distinct general elements $M_1,\dots,M_{p(d+1)}$ in $V_z$ and let 
$$
M=\frac{1}{p}(M_1+\dots+M_{p(d+1)}).
$$ 
Then $(X,B+M)$ is $\epsilon$-lc  outside $f^{-1}\{z\}$ by generality of the $M_i$ and the assumption 
$\frac{1}{p}<1-\epsilon$.  On the other hand, $(X,B+M)$ is not lc at any point of 
 $f^{-1}\{z\}$ by [\ref{Kollar-flip-abundance}, Theorem 18.22].\\
 
\emph{Step 2.}
\emph{In this step we reduce the problem to the situation when there is a prime divisor 
$T$ on $X$ mapping to $z$ with $a(T,X,B)=\epsilon$ sufficiently small.} 
Now pick a sufficiently small rational number $\epsilon'\in (0,\epsilon)$ and 
let $u$ be the largest number such that $(X,B+uM)$ is $\epsilon'$-lc. There is a prime divisor $T$ over $X$ 
such that 
$$
a(T,X,B+uM)=\epsilon'.
$$
As $(X,B+M)$ is not lc near $f^{-1}\{z\}$, $u<1$. 
Since $(X,B+uM)$ is $\epsilon$-lc outside $f^{-1}\{z\}$ and since $\epsilon'<\epsilon$, 
the centre of $T$ on $X$ is contained in $f^{-1}\{z\}$. On the other hand, it is clear that 
$$
K_X+B+uM\sim_\R f^*(L+u(d+1)A).
$$ 
  
Replacing $\epsilon$ with $\epsilon'$ and replacing $B$ with $B+uM$ (and replacing $A,r$ accordingly) 
we can assume that 
$\epsilon$ is sufficiently small and that there is a prime divisor $T$ over $X$ mapping to $z$ with $a(T,X,B)=\epsilon$. 
Extracting $T$ we can assume it is a divisor on $X$; if $T$ is not exceptional over the 
original $X$, we increase the coefficient of $T$ in $\Delta$ to $1-\epsilon$; but if $T$ is exceptional over the original $X$, 
then we let $\Delta$ be the birational transform of the original $\Delta$ plus $(1-\epsilon)T$. 
The bigness of $-(K_X+\Delta)$ over $Z$ is preserved as $T$ is vertical over $Z$.\\ 

\emph{Step 3.}
\emph{In this step we find a bounded complement of $K_X+\Delta$ using Theorem \ref{t-bnd-comp-lc-global}.}
Let $\Theta$ be the same as $\Delta$ except that we increase the coefficient of $T$ to $1$.
Adding $1$ to $\mathfrak{R}$ we can assume that the coefficients of $\Theta$ are in $\mathfrak{R}$. 
Since $T$ is vertical over $Z$, $-(K_X+\Theta)$ is big over $Z$.
Run an MMP on $-(K_X+\Theta)$ over $Z$ 
and let $X'$ be the resulting model. We denote the pushdown of each divisor $D$ to $X'$ 
by $D'$. Then $-(K_{X'}+\Theta')$ is nef and big over $Z$.
By construction, $-\epsilon T'\le B'-\Theta'$, hence since $T$ is mapped to a closed point on $Z$, 
$(f^*A)'+B'-\Theta'$ is pseudo-effective. Thus   
by Proposition \ref{l-bnd-cy-numerical-bndness}, we can assume that $(lf^*A)'-(K_{X'}+\Theta')$ is nef for some 
bounded natural number $l$. Increasing $l$ by $1$ we can assume $(lf^*A)'-(K_{X'}+\Theta')$ is nef and big. 
 
 On the other hand, $(X',\Theta'-\epsilon T')$ is klt as $\Theta'-\epsilon T'\le B'$. 
Since $\epsilon$ is assumed to be sufficiently small, by the ACC for 
lc thresholds [\ref{HMX2}, Theorem 1.1], $(X',\Theta')$ is lc.
Then applying Theorem \ref{t-bnd-comp-lc-global} (by taking $B=\Theta'$, $M=(lf^*A)'$ and 
$S$ to be the centre of $T$ on $X'$), there exist a 
bounded natural number $n$ and $\Lambda'\ge \Theta'$ such that $(X',\Lambda')$ is lc over $z$ and 
$n(K_{X'}+\Lambda')\sim (mf^*A)'$ where  $m:=l(n+2)$.
Since $X\bir X'$ is an MMP on $-(K_X+\Theta)$ over $Z$, 
taking $K_X+\Lambda$ to be the crepant pullback of 
$K_{X'}+\Lambda'$ to $X$ we get $\Lambda\ge \Theta\ge \Delta$ such that $(X,\Lambda)$ is lc over $z$ and 
$n(K_{X}+\Lambda)\sim mf^*A$. 

\end{proof}

Now we strengthen the previous statement by replacing lc over $z$ with klt over $z$. 

\begin{prop}\label{l-bnd-cy-bnd-klt-compl-2}
Let $d,r$ be natural numbers, $\epsilon$ be a positive real number, and $\mathfrak{R}\subset [0,1]$ 
be a finite set of rational numbers.
Assume that Theorems \ref{t-log-bnd-cy-fib}, \ref{t-sing-FT-fib-totalspace}, and 
\ref{t-bnd-comp-lc-global-cy-fib} hold in dimension $d-1$. Then there exist  
natural numbers $n,m$ depending only on $d,r,\epsilon,\mathfrak{R}$ satisfying the following. 
Assume that $(X,B)\to Z$ is a $(d,r,\epsilon)$-Fano type fibration (as in \ref{d-FT-fib}) and that 
\begin{itemize}
\item we have $0\le \Delta\le B$ with coefficients in $\mathfrak{R}$, and 

\item $-(K_X+\Delta)$ is big over $Z$. 
\end{itemize}
Then for each point $z\in Z$ there is a $\Q$-divisor $\Lambda\ge \Delta$ such that 
\begin{itemize}
\item $(X,\Lambda)$ is klt over $z$, and

\item $n(K_X+\Lambda)\sim mf^*A$. 
\end{itemize} 
\end{prop}
\begin{proof}
\emph{Step 1.}
\emph{In this step we modify $B$ and introduce a divisor $\tilde\Delta$.}
We can assume that $\dim Z>0$ otherwise we apply [\ref{B-compl}, Corollary 1.2] which shows that 
$(X,\Delta)$ is log bounded. Moreover, it is enough to prove the lemma by replacing $z$ 
with any closed point $z'$ in $\bar{z}$ because 
$(X,\Lambda)$ being klt over $z'$ implies that it is klt over $z$. Thus from now on we assume that $z$ is a closed point. 
Then as $A$ is very ample we can find $P\in |f^*A|$ containing $f^{-1}\{z\}$.

 By Proposition \ref{l-bnd-cy-bnd-lct-special}, there is a rational number 
$t>0$ depending only on $d,r,\epsilon$ such that  $(X,B+2tP)$ is lc. Since 
$$
B+tP=\frac{1}{2}B+\frac{1}{2}(B+2tP),
$$
the pair $(X,B+tP)$ is $\frac{\epsilon}{2}$-lc. Moreover, 
$$
K_X+B+tP\sim_\R f^*(L+tA).
$$
Thus replacing $B$ with $B+tP$ 
 and replacing $\epsilon$ with $\frac{\epsilon}{2}$, 
we can assume that $B\ge \tilde{\Delta}:=\Delta +tP$ for some fixed rational number $t\in (0,1)$ (here we can replace 
$A$ with $2A$  to ensure that $f^*A-(K_X+B)$ is still nef, and then replace $r$ accordingly). 
Since $P$ is integral and $t$ is fixed, the coefficients of $\tilde{\Delta}$ belong to 
a fixed finite set, so expanding $\mathfrak{R}$ we can assume they belong to $\mathfrak{R}$.\\

\emph{Step 2.}
\emph{In this step we reduce the proposition to existence of a special lc complement.}
Assume that there exist bounded natural numbers $n,m$ and a $\Q$-divisor $\Lambda\ge \tilde{\Delta}$ 
such that 
\begin{enumerate}
\item $(X,\Lambda)$ is lc over $z$, 

\item the non-klt locus of $(X,\Lambda)$ is mapped to a finite set of closed points on $Z$, and 

\item that $n(K_X+\Lambda)\sim mf^*A$. 
\end{enumerate}

Assume that $Q\in |f^*A|$ is general and let $\Lambda':=\Lambda-tP+tQ$. 
By (2), any non-klt centre of $(X,\Lambda)$ intersecting $f^{-1}\{z\}$ is actually contained in 
$f^{-1}\{z\}$. Thus since $P$ contains $f^{-1}\{z\}$, $(X,\Lambda')$ is klt over $z$. Moreover, 
$\Lambda'\ge \Delta$, and perhaps after replacing $n,m$ with a bounded multiple 
we have   
$$
n(K_X+\Lambda')=n(K_X+\Lambda-tP+tQ)\sim n(K_X+\Lambda)\sim m f^*A.
$$
Therefore, it is enough to find $n,m,\Lambda$ as in (1)-(3). 
At this point we replace $\Delta$ with $\tilde{\Delta}$. The bigness of $-(K_X+\Delta)$ over $Z$ 
is preserved as $P$ is vertical.\\ 

\emph{Step 3.}
\emph{In this step we find a bounded lc complement of $K_X+\Delta$ and study it.}
After taking a $\Q$-factoriallisation of $X$ and 
running an MMP on $-(K_X+\Delta)$ over $Z$ we can assume that $-(K_X+\Delta)$ is nef over $Z$. 
Applying Proposition \ref{l-bnd-cy-numerical-bndness}, there is a bounded natural number $l$ such that 
$lf^*A-(K_X+\Delta)$ is nef globally. Replacing $l$ with $l+1$ we can assume 
$lf^*A-(K_X+\Delta)$ is nef and big.

By Proposition \ref{l-bnd-cy-bnd-klt-compl-1}, there exist bounded natural numbers $n,m$ and a 
$\Q$-divisor $\Lambda\ge \Delta$ such that $(X,\Lambda)$ is lc over $z$ and 
$n(K_{X}+\Lambda)\sim mf^*A$. Then 
$$
n(\Lambda-\Delta)=n(K_X+\Lambda)-n(K_X+\Delta)\sim mf^*A-n(K_X+\Delta),
$$
hence
 $$
 n(\Lambda-\Delta)\in |mf^*A-n(K_X+\Delta)|.
 $$
Multiplying $n,m$ by a bounded number we can assume that $n\Delta$ is integral.

Now by adding a general member of $|2lf^*A|$ to $\Lambda$ and replacing $m$ with $m+2nl$ 
to preserve $n(K_{X}+\Lambda)\sim mf^*A$, we can assume that $m-1\ge l(n+1)$, hence  
$$
(m-1)f^*A-(n+1)(K_X+\Delta)
$$
is nef and big.\\ 

\emph{Step 4.}
\emph{In this step we consider the restriction of $|mf^*A-n(K_X+\Delta)|$ to a general 
member of $|f^*A|$.} Let $H$ be a general member of $|A|$ and let $G=f^*H$. Then 
$$
mf^*A-n(K_X+\Delta)-G\sim K_X+\Delta+(m-1)f^*A-(n+1)(K_X+\Delta).
$$
 Thus 
$$
H^1(mf^*A-n(K_X+\Delta)-G)=0
$$ 
by the Kawamata-Viehweg vanishing theorem, hence the restriction map 
$$
H^0(mf^*A-n(K_X+\Delta)) \to H^0((mf^*A-n(K_X+\Delta))|_G) 
$$
is surjective. Note that for any Weil divisor $D$ on $X$, we have 
$$
\mathcal{O}_X(D)\otimes\mathcal{O}_G\simeq \mathcal{O}_G(D|_G)
$$
by the choice of $G$ (see [\ref{B-compl}, 2.41]). This is used to get the above surjectivity.\\ 

\emph{Step 5.}
\emph{In this step we consider complements on $G$.}
Define 
$$
K_G+B_G=(K_X+B+G)|_G
$$ 
and
$$
K_G+\Delta_G=(K_X+\Delta+G)|_G.
$$ 
Then as we have seen several times in this section, $(G,B_G)\to H$ is a $(d-1,r',\epsilon)$-Fano type 
fibration for some fixed $r'$. Moreover, $\Delta_G\le B_G$, the coefficients of $\Delta_G$ are in $\mathfrak{R}$, 
and $-(K_G+\Delta_G)$ is big over $H$. 

Since we are assuming Theorem \ref{t-bnd-comp-lc-global-cy-fib} in dimension $d-1$, 
there exist bounded natural numbers $p,q$ and there is a $\Q$-divisor $\Lambda_{G}'\ge \Delta_G$ such that 
$(G,\Lambda_G')$ is klt and $p(K_G+\Lambda_G')\sim qg^*A|_H$ where $g$ denotes the morphism $G\to H$. 
Replacing both $n$ and $p$ with $np$  and then replacing $m$ and $q$ with $mp$ and $nq$, respectively, 
we can assume that $n=p$. Next if $q<m+n$, then we increase $q$ to $m+n$ by adding $\frac{1}{n}D_G$ 
to $\Lambda_G'$ where $D_G$ is a general element of $(m+n-q)g^*A|_H$. 
If $q\ge m+n$, we similarly increase $m$ to $q-n$ by modifying $\Lambda$ 
so that we can again assume $q=m+n$. Thus we now have 
$$
n(K_G+\Lambda_G')\sim (m+n)g^*A|_H.
$$
Note that in the process, the  inequality $m-1\ge l(n+1)$
of Step 3 is preserved, so the surjectivity of Step 4 still holds.\\

\emph{Step 6.}
\emph{In this step we finish the proof.}
By construction,
$$
nR_G:=n(\Lambda_G'-\Delta_G)\in |(m+n)g^*A|_H-n(K_G+\Delta_G)|,
$$
and $(G,\Lambda_G'=\Delta_G+R_G)$ is klt. Thus if we replace $nR_G$ with any general element of
 $$
 |(m+n)g^*A|_H-n(K_G+\Delta_G)|,
 $$ 
then the pair $(G,\Delta_G+R_G)$ is still klt. On the other hand,
$$  
\begin{array}{l l}
 (m+n)g^*A|_H-n(K_G+\Delta_G) &= (m+n)g^*A|_H-n(K_X+\Delta+G)|_G\\
 & =((m+n)f^*A-nG-n(K_X+\Delta))|_G\\
 & \sim (mf^*A-n(K_X+\Delta))|_G.
\end{array}
$$
  Thus, by the surjectivity in Step 4, a general element 
$$
nR\in |mf^*A-n(K_X+\Delta)|
$$
restricts to a general element 
 $$
nR_G\in |(m+n)g^*A|_H-n(K_G+\Delta_G)|.
 $$ 
 
Now in view of 
$$
K_G+\Delta_G+R_G=(K_X+\Delta+R+G)|_G
$$ 
and inversion of adjunction [\ref{kollar-mori}, Theorem 5.50],
the pair $(X,\Delta+R+G)$ is plt near $G$, hence $(X,\Delta+R)$ is klt near $G$. 
Therefore, replacing $\Lambda$ with $\Delta+R$ we can assume that $(X,\Lambda)$ is klt near 
$G$. In other words, the non-klt locus of 
$(X,\Lambda)$ is mapped to a finite set of closed points of $Z$.
Note that $(X,\Lambda)$ is still lc over $z$. Thus we have satisfied the conditions (1)-(3) of Step 2.

\end{proof}

\begin{lem}\label{l-bnd-cy-bnd-klt-compl-induction}
Assume that Theorems \ref{t-log-bnd-cy-fib}, \ref{t-sing-FT-fib-totalspace}, and \ref{t-bnd-comp-lc-global-cy-fib} 
hold in dimension $d-1$. Then Theorem \ref{t-bnd-comp-lc-global-cy-fib} holds in dimension $d$.
\end{lem}
\begin{proof}
 By Proposition \ref{l-bnd-cy-bnd-klt-compl-2}, there exist 
natural numbers $n,m$ depending only on $d,r,\epsilon, \mathfrak{R}$ such that for each point $z\in Z$ 
 there is a $\Q$-divisor $\Gamma\ge \Delta$ such that 
\begin{itemize}
\item $(X,\Gamma)$ is klt over some neighbourhood $U_z$ of $z$, and

\item $n(K_X+\Gamma)\sim mf^*A$. 
\end{itemize} 
We can find finitely many closed points $z_1,\dots,z_p$ in $Z$ such that the corresponding open sets $U_{z_i}$ cover 
$Z$. For each $z_i$ let $\Gamma_i$ be the corresponding boundary as above.

From 
$$
n(\Gamma_i-\Delta)= n(K_X+\Gamma_i)-n(K_X+\Delta)\sim mf^*A -n(K_X+\Delta)
$$ 
we get
$$
n(\Gamma_i-\Delta) \in |mf^*A-n(K_X+\Delta)|.
$$
Therefore, if $nR$ is a general member of $|mf^*A-n(K_X+\Delta)|$ and if we let $\Lambda:=\Delta+R$, 
then 
\begin{itemize}
\item $(X,\Lambda)$ is klt over $U_{z_i}$,  and 

\item $n(K_X+\Lambda)\sim mf^*A$. 
\end{itemize} 
Finally, since we have only finitely many open sets $U_{z_i}$ involved, $(X,\Lambda)$ is klt everywhere. 

\end{proof}

\subsection{A special case of boundedness of Fano type fibrations}

We treat a special case of Theorem \ref{t-log-bnd-cy-fib} inductively.

\begin{lem}\label{l-bnd-cy-fib-ample-case}
Let $d,r$ be natural numbers, $\epsilon$ be a positive real number, and $\mathfrak{R}\subset [0,1]$ be a 
finite set of rational numbers. Assume that Theorems \ref{t-log-bnd-cy-fib}, \ref{t-sing-FT-fib-totalspace}, 
and \ref{t-bnd-comp-lc-global-cy-fib} hold in dimension $d-1$. Consider the set of all
$(d,r,\epsilon)$-Fano type fibrations $(X,B)\to Z$ (as in \ref{d-FT-fib}) and $\R$-divisors $0\le \Delta\le B$ such that  
\begin{itemize}
\item the coefficients of $\Delta$ are in $\mathfrak{R}$, and

\item $-(K_X+\Delta)$ is ample over $Z$. 
\end{itemize} 
Then the set of such $(X,\Delta)$ is log bounded.
\end{lem}

\begin{proof}
By Lemma \ref{l-bnd-cy-bnd-klt-compl-induction}, our assumptions imply Theorem \ref{t-bnd-comp-lc-global-cy-fib} 
in dimension $d$, hence there exist 
natural numbers $n,m$ depending only on $d,r,\epsilon,\mathfrak{R}$ and a $\Q$-divisor $\Lambda\ge \Delta$ such that 
\begin{itemize}
\item $(X,\Lambda)$ is klt, and  

\item $n(K_X+\Lambda)\sim mf^*A$. 
\end{itemize} 
We have 
$$
n(\Lambda-\Delta) \in |mf^*A-n(K_X+\Delta)|.
$$
Increasing $m$ (by adding to $\Lambda$ appropriately) and applying Proposition \ref{l-bnd-cy-numerical-bndness}, 
we can assume that $mf^*A-n(K_X+\Delta)$ is nef and that $l:=\frac{m}{n}$ is a natural number. 
Since $-(K_X+\Delta)$ is ample over $Z$, replacing $m$ with a bounded multiple (which then replaces 
$l$ with a bounded multiple), we can assume that $mf^*A-n(K_X+\Delta)$ is ample.
In particular,  $\Lambda-\Delta$ is ample. 

Since $(X,\Lambda)$ is klt and $n(K_X+\Lambda)$ is Cartier, $(X,\Lambda)$ is  $\frac{1}{n}$-lc. 
Pick a small $t>0$ such that 
$$
(X,\Theta:=\Lambda+t(\Lambda-\Delta))
$$
is $\frac{1}{2n}$-lc. Here $t$ depends on $(X,\Lambda)$. 
Then
$$
K_X+\Theta=K_X+\Lambda+t(\Lambda-\Delta)\sim_\Q l f^*A+t(\Lambda-\Delta)
$$
is ample. In addition, since $\Supp(\Lambda-\Delta)\subseteq \Lambda$ and since $n\Lambda$ is integral, 
each non-zero coefficient of $\Theta$ is at least $\frac{1}{n}$.

Now since $n\Lambda$ is integral and since $lf^*A -(K_X+\Lambda)\sim_\Q 0$, 
 $(X,\Lambda)$ is log birationally bounded, by Proposition \ref{l-bnd-cy-bir-bnd}. Thus  
 $(X,\Theta)$ is also log birationally bounded as $\Supp \Theta=\Supp \Lambda$. Therefore, $(X,\Theta)$ is log bounded by 
[\ref{HMX2}, Theorem 1.6] which implies that $(X,\Delta)$ is log bounded as $\Delta\le \Theta$. 

\end{proof}

\subsection{Boundedness of generators of N\'eron-Severi groups}

To treat the Theorem \ref{t-log-bnd-cy-fib} in full generality we need to discuss 
generators of relative N\'eron-Severi groups. We start with bounding global Picard numbers.

\begin{lem}\label{l-cy-fib-bnd-picard-number}
Let $d,r$ be natural numbers and $\epsilon$ be a positive real number. 
Assume that Theorems \ref{t-log-bnd-cy-fib}, \ref{t-sing-FT-fib-totalspace}, and \ref{t-bnd-comp-lc-global-cy-fib} 
hold in dimension $d-1$. Then there is a natural number $p$ depending only on $d,r,\epsilon$ 
satisfying the following.
If $(X,B)\to Z$ is a $(d,r,\epsilon)$-Fano type fibration (as in \ref{d-FT-fib}), then the Picard number $\rho(X)\le p$.
\end{lem}
\begin{proof}
Replacing $X$ with a $\Q$-factorialisation we can assume $X$ is $\Q$-factorial. 
Running an MMP on $-K_X$ over $Z$ we find $Y$ so that $-K_Y$ is nef and big over $Z$. 
Replace $Y$ with the ample model of $-K_Y$ over $Z$ so that  $-K_Y$ 
becomes ample over $Z$. Let $K_Y+B_Y$ be the pushdown of $K_X+B$. Then 
$(Y,B_Y)\to Z$ is a $(d,r,\epsilon)$-Fano type fibration. 
Now applying Lemma \ref{l-bnd-cy-fib-ample-case} to $(Y,B_Y)\to Z$ we deduce that 
$Y$ is bounded. 

By construction, if $D$ is a prime divisor on $X$ contracted over $Y$, then 
$$
a(D,Y,0)\le a(D,X,0)=1.
$$ 
Thus, by  [\ref{HX}, Proposition 2.5], there is a birational morphism $X'\to Y$ from a bounded normal 
projective variety such that the induced map $X\bir X'$ is an isomorphism in codimension one.

We can take a resolution $W\to X'$ such that $W$ is bounded. Then there exist finitely many surjective 
smooth projective morphisms $V_i\to T_i$ between smooth varieties, depending only on $d,r,\epsilon$,  such that 
$W$ is a fibre of $V_i\to T_i$ over some closed point for some $i$. Since smooth morphisms are locally 
products in the complex topology (here we can assume that the ground field is $\C$),  
$\dim_\R H^2(W,\R)$ is bounded by some number $p$ depending only on $d,r,\epsilon$. 
In particular, since the N\'eron-Severi group $N^1(W)$ is embedded in $H^2(W,\R)$ as a vector space, we get  
$$
\rho(W)\le \dim_\R H^2(W,\R)\le p.
$$ 
Since $X\bir X'$ is an isomorphism in codimension one and since $W\to X'$ is a morphism,  
the induced map $X\bir W$ does not contract divisors, hence $\rho(X)\le \rho(W)\le p$. 

\end{proof}

\begin{prop}\label{p-cy-fib-bnd-Neron-Severi}
Let $d,r$ be natural numbers, $\epsilon$ be a positive real number, and $\mathfrak{R}\subset [0,1]$ be a 
finite set of rational numbers. 
Assume that Theorems \ref{t-log-bnd-cy-fib}, \ref{t-sing-FT-fib-totalspace}, and \ref{t-bnd-comp-lc-global-cy-fib} 
hold in dimension $d-1$. Then there is a bounded set $\mathcal{P}$ of couples depending only on 
$d,r,\epsilon, \mathfrak{R}$ satisfying the following. 
Suppose that 
\begin{itemize}
\item $(X,B)\to Z$ is a $(d,r,\epsilon)$-Fano type fibration (as in \ref{d-FT-fib}), and that

\item  the coefficients of $B$ are in $\mathfrak{R}$.
\end{itemize}
Then there exist a birational map $X\bir X'$  and a reduced divisor $\Sigma'$ on $X'$ 
such that 
\begin{itemize}
\item $X'$ is a $\Q$-factorial normal projective variety, 

\item $X\bir X'$ is an isomorphism in codimension one, 

\item $(X',\Sigma')$ belongs to $\mathcal{P}$, 

\item $\Supp B'\subseteq \Sigma'$ where $B'$ is the birational transform of $B$, and 

\item the irreducible components of $\Sigma'$ generate $N^1(X'/Z)$.
\end{itemize}
\end{prop}

By $N^1(X'/Z)$ we mean $\Pic(X')\otimes \R$ modulo numerical equivalence over $Z$. 
Note that there is a natural surjective map $N^1(X')\to N^1(X'/Z)$.
We prove some lemmas before giving the proof of the proposition.

\begin{lem}\label{l-cy-fib-bnd-Neron-Severi-Mfs}
Assume that Proposition \ref{p-cy-fib-bnd-Neron-Severi} holds in dimension $\le d-1$. Then the proposition holds 
in dimension $d$ when $X$ is $\Q$-factorial and there is a non-birational extremal contraction 
$h\colon X\to Y/Z$.
\end{lem}
\begin{proof}
First note that if $\dim Y=0$, then $X$ is an $\epsilon$-lc Fano variety with Picard number one 
and $K_X+B\sim_\Q 0$, hence $X$ belongs to a bounded family by [\ref{B-compl}, Theorem 1.4], hence $(X,B)$ is log 
bounded as the coefficients of $B$ are in $\mathfrak{R}$ which implies the result in this case as $N^1(X/Z)$ is generated by the components of $B$.
We can then assume that $\dim Y>0$.

Let $F$ be a general fibre of $h$ and let $K_F+B_F:=(K_X+B)|_F$. Then 
$(F,B_F)$ is $\epsilon$-lc, $K_F+B_F\sim_\Q 0$, and $B_F$ is big with coefficients in $\mathfrak{R}$, 
hence $F$ belongs to a bounded family by [\ref{B-compl}, Theorem 1.4] which implies that 
$(F,B_F)$ is log bounded. Moreover, by adjunction, we can write 
$$
K_X+B\sim_\Q h^*(K_Y+B_Y+M_Y)
$$
where we consider $(Y,B_Y+M_Y)$ as a generalised pair as in Remark \ref{rem-base-fib-gen-pair} below.
By [\ref{B-sing-fano-fib}, Theorem 1.4], $(Y,B_Y+M_Y)$ is generalised $\delta$-lc for 
some fixed $\delta>0$ which depends only on $d,\epsilon, \mathfrak{R}$. 

By construction, 
$$
K_Y+B_Y+M_Y\sim_\R g^*L
$$
where $g$ denotes the morphism $Y\to Z$. Moreover, since $X$ is of Fano type over $Z$, $Y$ is also of 
Fano type over $Z$ (cf, the proof of [\ref{B-compl}, Lemma 2.12] works in the relative setting).
Then $(Y,B_Y+M_Y)\to Z$ is a generalised $(d',r,\delta)$-Fano type fibration for some $d'\le d-1$, 
as in \ref{d-gen-FT-fib}. 
By Lemma \ref{l-from-gen-fib-to-usual-fib}, we can find a boundary
$\Delta_Y$ so that $(Y,\Delta_Y)\to Z$ is a $(d',r,\frac{\delta}{2})$-Fano type fibration.
Therefore, applying Lemma \ref{l-bnd-cy-bnd-klt-compl-induction}, there exist bounded natural numbers 
$n,m$ and a boundary $\Lambda_Y$ such that $(Y,\Lambda_Y)$ is klt and 
$n(K_X+\Lambda_Y)\sim mg^*A$. In particular, $(Y,\Lambda_Y)\to Z$ is a $(d',r',\epsilon':=\frac{1}{n})$-Fano type 
fibration for some fixed $r'$. Moreover, increasing $m$ by adding to $\Lambda_Y$, we can assume that $m>n$.
Since we are assuming Proposition \ref{p-cy-fib-bnd-Neron-Severi} in dimension $d-1$, 
 applying it to $(Y,\Lambda_Y)\to Z$ we deduce that  there exist a birational map 
$Y\bir Y'/Z$ to a normal projective variety and 
a reduced divisor $\Sigma_{Y'}$ on $Y'$ satisfying the properties listed in \ref{p-cy-fib-bnd-Neron-Severi}. 

By Lemma \ref{l-contraction-after-flops}, there exists a birational map $X\bir X'/Z$ which is an isomorphism in codimension one 
so that the induced map $X'\bir Y'$ is an extremal contraction (hence a morphism) and 
$X'$ is normal projective and $\Q$-factorial. Let $B'$ on $X'$ be the birational transform of $B$ and let 
$\Lambda_{Y'}$ on $Y'$ be the birational transform of $\Lambda_Y$.  
To ease notation we can replace $(X,B)$ and $(Y,\Lambda_Y)$ with $(X',B')$ and $(Y',\Lambda_Y')$ 
and denote $\Sigma_{Y'}$ by $\Sigma_Y$. By construction, $\Supp \Lambda_Y\le \Sigma_Y$.

Since $\Supp \Lambda_Y\subseteq \Sigma_Y$ and since $(Y,\Sigma_Y)$ is log bounded, 
there is a very ample divisor $H$ on $Y$ with bounded $s:=H^{\dim Y}$ 
such that 
$$
H-(K_Y+\Lambda_Y)\sim_\Q H-\frac{m}{n}g^*A
$$ 
is ample which implies that $H-g^*A$ is ample as $m>n$. Then  
$$
H-g^*L=H-g^*A+g^*(A-L)
$$
is ample as $A-L$ is ample. Thus $(X,B)\to Y$ is a $(d,s,\epsilon)$-Fano type fibration in view of $K_X+B\sim_\R h^*g^*L$. 

Replacing $H$ we can in addition assume that $H-\Sigma_Y$ is ample. Thus we can find $0\le P\sim_\R h^*H$ such that 
$P\ge h^*\Sigma_Y$. Now applying 
Proposition \ref{l-bnd-cy-bnd-lct-special} to $(X,B)\to Y$, there is a fixed rational number $t\in (0,1)$ 
such that $(X,B+2tP)$ is klt. Thus $(X,B+2th^*\Sigma_Y)$ is klt, so
$$
(X,\Theta:=B+th^*\Sigma_Y)
$$ 
is $\frac{\epsilon}{2}$-lc. Therefore, from 
$$
K_X+\Theta=K_X+B+th^*\Sigma_Y\sim_\R h^*(g^*L+t\Sigma_Y)
$$
we deduce that $(X,\Theta)\to Y$ is a $(d,s,\frac{\epsilon}{2})$-Fano type fibration, 
perhaps after replacing $H$ with $2H$ and replacing $s$ accordingly.

On the other hand, the coefficients of $th^*\Sigma_Y$ 
belong to a fixed finite set because $t$ is fixed, the Cartier index of $\Sigma_Y$ is bounded [\ref{B-compl}, Lemma 2.24], 
and the coefficients of  $th^*\Sigma_Y$ are less than $1$. Thus the coefficients of $\Theta$ belong to a 
fixed finite set. Moreover, since $X\to Y$ is extremal and $\Theta$ is big over $Y$, $\Theta$ is ample 
over $Y$, hence  $-(K_X+\frac{1}{2}\Theta)\sim_\R \frac{1}{2}\Theta/Y$ is ample over $Y$. 
Therefore, applying Lemma \ref{l-bnd-cy-fib-ample-case} to $(X,\Theta)\to Y$ (by taking $\Delta=\frac{1}{2}\Theta$) 
we deduce that $(X,\frac{1}{2}\Theta)$ is log bounded.
Since $B$ is big over $Y$, it is ample over $Y$, 
hence the components of $B$ and $h^*\Sigma_Y$ together  generate $N^1(X/Z)$ 
as the components of $\Sigma_Y$ generate $N^1(Y/Z)$. Now let $\Sigma:=\Supp \Theta$.

\end{proof}

\begin{lem}\label{l-cy-fib-bnd-Neron-Severi-sqf}
Proposition \ref{p-cy-fib-bnd-Neron-Severi} holds when $X\to Z$ is a small $\Q$-factorialisation.
\end{lem}
\begin{proof}
We will apply induction on the relative Picard number $\rho(X/Z):=\dim_\R N^1(X/Z)$. 
By Lemma \ref{l-cy-fib-bnd-picard-number}, $\rho(X)$ is bounded, so 
$\rho(X/Z)$ is bounded as well because $\rho(X/Z)\le \rho(X)$. 
The case $\rho(X/Z)=0$ is trivial in which case $X\to Z$ is an isomorphism and 
$(X,B)$ is log bounded, so we assume $\rho(X/Z)>0$.

By Lemma \ref{l-bnd-cy-bnd-klt-compl-induction}, our assumptions imply Theorem \ref{t-bnd-comp-lc-global-cy-fib} 
in dimension $d$. Since $X\to Z$ is birational, $-(K_X+B)$ is big over $Z$, hence applying 
the theorem there exist bounded natural numbers $n,m$ and a boundary $\Lambda\ge B$ such that 
$(X,\Lambda)$ is klt and $n(K_X+\Lambda)\sim mf^*A$. In particular, $n(K_X+\Lambda)$ is Cartier and 
$(X,\Lambda)$ is $\frac{1}{n}$-lc. Replacing $B$ with $\Lambda$, $\epsilon$ with $\frac{1}{n}$, 
$A$ with $2mA$, and replacing $r,\mathfrak{R}$ accordingly, we can assume that $n(K_X+B)$ is Cartier for 
some fixed natural number $n$. Replacing $n$ with $2n$ we can assume $n\ge 2$.

Let $B_Z$ be the pushdown of $B$. By boundedness of length of extremal rays [\ref{kawamata-bnd-ext-ray}], 
$K_Z+B_Z+(2d+1)A$ is ample. Thus taking a general member $G\in |n(2d+1)f^*A|$, adding $\frac{1}{n}G$ to $B$, 
and then replacing $A$ with $(2d+2)A$ (to keep the ampleness of $A-L$), we can assume that 
$K_X+B$ is the pullback of some ample divisor on $Z$ and that $B-\frac{1}{2}f^*A$ is pseudo-effective. 
We have used the assumption $n\ge 2$ to make sure that the $\epsilon$-lc property of $(X,B)$ is preserved.

By the cone theorem [\ref{kollar-mori}, Theorem 3.7], we can decompose 
$X\to Z$ into a sequence 
$$
X=X_1\to X_2 \to \cdots \to X_l=Z
$$
of extremal contractions. 
Let $B_i$ be the pushdown of $B$. Then 
$$
(K_X+B)^d=\vol(K_X+B)=\vol(K_{X_i}+B_i)\le \vol(A)=A^d=r,
$$ 
hence there are only finitely many possibilities for $\vol(K_{X_i}+B_i)$ as $n(K_X+B)$ is Cartier. 
Therefore, by [\ref{DST}, Theorem 6], the set of such $(X_i,B_i)$ is log bounded.
In particular, there is a very ample divisor $G_{l-1}$ on $X_{l-1}$ 
with bounded $G_{l-1}^d$ such that $G_{l-1}-A_{l-1}$ is ample where $A_{l-1}$ is the 
pullback of $A$ (here we are using the property that $B-\frac{1}{2}f^*A$ is pseudo-effective). 

Let $G$ be the pullback of $G_{l-1}$ to $X$. Let $\Theta=B+\frac{1}{n}P$ for some general element $P\in |nG|$.  
Then $K_X+\Theta\sim_\Q 0/X_{l-1}$ and 
$$  
\begin{array}{l l}
2G-(K_X+\Theta) &=2G-(K_X+B)-\frac{1}{n}P\\
& \sim_\Q G-(K_X+B)\\
& =G-f^*A+f^*A-f^*L
\end{array}
$$
is the pullback of an ample divisor on $X_{l-1}$ as $G_{l-1}-A_{l-1}$ and $A-L$ are ample.
Thus $(X,\Theta)\to X_{l-1}$ is a $(d,u,\epsilon)$-Fano type fibration for some fixed number $u$. 

Now $\rho(X/X_{l-1})<\rho(X/Z)$.
Therefore, by induction on the relative Picard number, there is a birational map $X\bir X'/X_{l-1}$ and a reduced divisor 
 $\Sigma'$ on $X'$ satisfying the properties listed in \ref{p-cy-fib-bnd-Neron-Severi} with $\Theta, X_{l-1}$ instead of $B,Z$.
Now since $P'$, the birational transform of $P$, is the pullback of some ample$/Z$ divisor on $X_{l-1}$, since $P'\le \Sigma'$, 
and since $X_{l-1}\to X_l=Z$ is extremal, the components of $\Sigma'$ generate $N^1(X'/Z)$. This proves the lemma.

\end{proof}

\begin{proof}(of Proposition \ref{p-cy-fib-bnd-Neron-Severi})
Replacing $X$ with a $\Q$-factorialisation we can assume $X$ is $\Q$-factorial. 
By Lemma \ref{l-cy-fib-bnd-picard-number},  the Picard number $\rho(X)$ is bounded. 
We will apply induction on dimension and induction on the relative Picard number $\rho(X/Z)$, in the 
$\Q$-factorial case. We will assume that  $X\to Z$ is not an isomorphism otherwise the proposition holds by taking 
$X'=X$ and $\Sigma'=\Supp B$ as in this case $(X,B)$ would be log bounded by definition of $(d,r,\epsilon)$-Fano type fibrations.

First we prove the proposition assuming that there is a birational map $h\colon X\bir Y/Z$ 
to a normal projective variety such that $h^{-1}$ does not contract any divisor but $h$ contracts some divisor.
Since $K_X+B\sim_\R 0/Z$, $(Y,B_Y)$ is klt where $B_Y$ is the pushdown of $B$. 
Replacing $Y$ with a $\Q$-factorialisation we can assume it is $\Q$-factorial. 
The log discrepancy of any prime divisor $D$ contracted by $h$ satisfies 
$$
a(D,Y,B_Y)=a(D,X,B)\le 1.
$$ 
Thus modifying $Y$ by extracting all such divisors except one, we can assume that $h$ contracts a 
single prime divisor $D$. Moreover, replacing $h$ with the extraction morphism determined by $D$, 
we can assume that $h$ is an extremal divisorial contraction. 

Now $(Y,B_Y)\to Z$ is a $(d,r,\epsilon)$-Fano type fibration and $\rho(Y/Z)<\rho(X/Z)$, 
so by the induction hypothesis, there exist a birational map 
$Y\bir Y'/Z$ to a normal projective variety and 
a reduced divisor $\Sigma_{Y'}$ on $Y'$ satisfying the properties of the proposition.  
Replacing $Y$ with $Y'$ and replacing $X$ accordingly (as in the previous paragraph) 
we can assume $Y=Y'$. We change the notation $\Sigma_{Y'}$ to $\Sigma_Y$.

Since $\Supp B_Y\subseteq \Sigma_Y$ and since $(Y,\Sigma_Y)$ is log bounded, by 
[\ref{B-BAB}, Theorem 1.6] (=Theorem \ref{t-bnd-lct}), 
there is a fixed rational number $t>0$ such that 
$$
(Y,\Theta_Y:=B_Y+t\Sigma_Y)
$$ 
is $\frac{\epsilon}{2}$-lc. 
Thus since $a(D,Y,\Theta_Y)\le 1$, applying [\ref{HX}, Proposition 2.5] we deduce that 
there is a birational contraction $X'\to Y$ extracting $D$ but no other divisors and such that if 
$K_{X'}+\Theta'$ is the pullback of $K_Y+\Theta_Y$, then $(X',\Theta')$ is log bounded. 
In addition, from the proof of [\ref{HX}, Proposition 2.5]  we can see that 
if $\Sigma'=\Supp (D+\Theta')$, then $(X',\Sigma')$ is log bounded.
Now since $Y$ is $\Q$-factorial, $X'=X$. For convenience we change the notation $\Theta',\Sigma'$ to $\Theta,\Sigma$. 
Since $K_X+B\sim_\Q 0/Y$ and $B_Y\le \Theta_Y$, we have $B\le \Theta$, hence 
$\Supp B\subseteq \Supp \Theta\subseteq \Sigma$. By construction, $\Sigma$ generates $N^1(X/Z)$, 
so we are done in this case.

Now we prove the proposition in general. 
If $X\to Z$ is not birational, then running an MMP on $K_X$ ends with a Mori fibre space 
$\tilde{X}\to Y/Z$; applying the above we can assume that $X\bir \tilde{X}$ does not contract any divisor, 
hence replacing $X$ we can assume $X=\tilde{X}$; we can then apply Lemma \ref{l-cy-fib-bnd-Neron-Severi-Mfs}. 
Now assume that $X\to Z$ is birational. Applying the above again reduces the proposition to the case 
when $X\to Z$ is a small contraction. But then we can apply Lemma \ref{l-cy-fib-bnd-Neron-Severi-sqf}.

\end{proof}

\subsection{Boundedness of Fano type fibrations}

In this subsection we treat Theorems \ref{t-bnd-cy-fib} and \ref{t-log-bnd-cy-fib} inductively. 

\begin{lem}\label{l-log-bnd-cy-fib-induction}
Assume that Theorems \ref{t-log-bnd-cy-fib}, \ref{t-sing-FT-fib-totalspace}, and \ref{t-bnd-comp-lc-global-cy-fib} 
hold in dimension $d-1$. Then Theorems \ref{t-bnd-cy-fib} and \ref{t-log-bnd-cy-fib} hold in dimension $d$.
\end{lem}
\begin{proof}
It is enough to treat \ref{t-log-bnd-cy-fib} as it implies \ref{t-bnd-cy-fib} by taking $\Delta=0$. 
If $\dim Z=0$, then $X$ belong to a bounded family by [\ref{B-BAB}, Corollary 1.2] from which 
we can deduce that $(X,\Delta)$ is log bounded. We can then assume that $\dim Z>0$. 
Let $(X,B)\to Z$ be a $(d,r,\epsilon)$-Fano type fibration and $0\le \Delta \le B$, as in \ref{t-log-bnd-cy-fib}. 
Changing the coefficients of $\Delta$ we can assume that all its coefficients are equal to a fixed rational number, 
and that $\Supp (B-\Delta)=\Supp B$. This in particular implies that $-(K_X+\Delta)\sim_\R B-\Delta/Z$ 
is big over $Z$. 

Now by Lemma \ref{l-bnd-cy-bnd-klt-compl-induction}, 
our assumptions imply Theorem \ref{t-bnd-comp-lc-global-cy-fib} in dimension $d$, hence applying the theorem 
we can find bounded natural numbers $n,m\ge 2$ and a boundary $\Lambda\ge \Delta$ such that $(X,\Lambda)$ is klt 
and $n(K_X+\Lambda)\sim mf^*A$. Replacing $B$ with $\Lambda$, $\epsilon$ with $\frac{1}{n}$, and $A$ with $2mA$ 
we can assume that the coefficients of $B$ belong to some fixed finite set of rational numbers. Furthermore, 
taking a general element $P\in |n(2d+1)f^*A|$ and adding $\frac{1}{n}P$ to $B$ 
and replacing $A,r$ accordingly we can assume that $B$ is big and that $K_X+B$ is nef. 

By Proposition \ref{p-cy-fib-bnd-Neron-Severi}, 
there exist a birational map $X\bir X'/Z$  and a reduced divisor $\Sigma'$ on $X'$ satisfying the properties 
listed in the proposition. Since $(X',\Sigma')$ is log bounded,  
there is a very ample divisor $G'$ on $X'$ with bounded $G'^d$ and bounded 
$G'^{d-1}\cdot \Sigma'$. Moreover, as $\Supp B'\subseteq \Sigma'$, 
$(X',B')$ is log bounded. Now since $B'$ is big, there is a fixed natural number $l>0$ 
such that $lB'\sim \Sigma'+D'$ where $D'\ge 0$. 
From $B'\le \Sigma'$ we deduce that 
$$
(X',\Supp (\Sigma'+D'))
$$ 
is log bounded because 
$$
G'^{d-1}\cdot (\Sigma'+D')\le G'^{d-1}\cdot (\Sigma'+lB')\le G'^{d-1}\cdot (\Sigma'+l\Sigma')
$$ 
is bounded. Replacing $G'$ with a multiple we can then assume that $G'-B'$ and $G'-(\Sigma'+D')$ 
are ample. 
Therefore, by [\ref{B-BAB}, Theorem 1.6] (=Theorem \ref{t-bnd-lct}), there is a 
fixed rational number $t\in (0,1)$ such that 
$$
(X',B'+t\Sigma'+tD')
$$ 
is $\frac{\epsilon}{2}$-lc. Replacing  
$B'$ with 
$$
(1-t)B'+\frac{t}{l}(\Sigma'+D')\sim_\Q B',
$$
replacing $B$ accordingly, replacing $\Sigma'$ with $\Supp (\Sigma'+D')$, 
and replacing $\epsilon$ with $\frac{\epsilon}{2}$, we can assume that $\Supp B'=\Sigma'$. 
In addition, by the previous paragraph, we can assume that  $B\ge \frac{1}{n}P$ for some 
general member $P\in |n(2d+1)f^*A|$ hence that the birational transform $P'\le \Sigma'$. 

Let $H$ be an ample $\Q$-divisor on $X$ and let $H'$ be its birational 
transform on $X'$. Since the components of $\Sigma'$ generate $N^1(X'/Z)$, there exists an $\R$-divisor 
$R'\equiv H'/Z$ such that $\Supp R'\subseteq \Sigma'$. In particular, if $R$ is the birational transform of $R'$ 
on $X$, then $R$ is ample over $Z$. Replacing $R'$ with a small multiple and adding a multiple of 
$P'$ to it, we can assume that  $R$ is globally 
ample. Since $\Supp R\subseteq \Supp B$, rescaling $R$ we can in addition assume that $\Theta:=B+R\ge \frac{1}{2}\Delta$, that 
the coefficients of $\Theta$ are $\ge \frac{\delta}{2}$, and that $(X,\Theta)$ is $\frac{\epsilon}{2}$-lc. 

By construction, $\Supp \Theta'=\Supp B'=\Sigma'$ where $\Theta'$ is the birational transform of $\Theta$. 
Thus $(X,\Theta)$ is log birationally bounded. 
Moreover, $K_X+\Theta$ is ample as $K_X+B$ is nef and $R$ is ample. Therefore, applying [\ref{HMX2}, Theorem 1.6] 
we deduce that $(X,\Theta)$ is log bounded which in particular means that $(X,\Delta)$ is log bounded.

\end{proof}

\subsection{Lower bound on lc thresholds}

\begin{lem}\label{l-bnd-cy-bnd-lct-usual}
Assume that Theorems \ref{t-log-bnd-cy-fib}, \ref{t-sing-FT-fib-totalspace}, and \ref{t-bnd-comp-lc-global-cy-fib} 
hold in dimension $d-1$. Then Theorem \ref{t-sing-FT-fib-totalspace} holds in dimension $d$. 
\end{lem}
\begin{proof}
Assume that $(X,B)\to Z$ is a $(d,r,\epsilon)$-Fano type fibration and $P\ge 0$ is $\R$-Cartier 
such that either $f^*A+B-P$ or $f^*A-K_X-P$ is pseudo-effective. 
Taking a $\Q$-factorialisation we can assume $X$ is $\Q$-factorial.  
First assume that  $f^*A+B-P$ is pseudo-effective.
Since $A-L$ is nef, $f^*A-(K_X+B)$ is nef, hence $2f^*A-K_X-P$ is pseudo-effective. Thus replacing 
$A$ with $2A$, it is enough to treat the theorem 
in the case when  $f^*A-K_X-P$ is pseudo-effective. 

By Lemma \ref{l-log-bnd-cy-fib-induction}, Theorem \ref{t-log-bnd-cy-fib} holds in dimension $d$. 
Let $D\in |f^*A|$ be a general element. Let $\Theta:=B+\frac{1}{2}D$. Then 
$$
K_X+\Theta\sim_\R f^*(L+\frac{1}{2}A)
$$
and $(X,\Theta)$ is $\epsilon'$-lc where $\epsilon'=\min\{\epsilon,\frac{1}{2}\}$. 
Thus $(X,\Theta)\to Z$ is a $(d,2^{d-1}r,\epsilon')$-Fano type fibration. 
Applying \ref{t-log-bnd-cy-fib}, we  
deduce that  $(X,D)$ is log bounded. 
Thus there is a very ample divisor $H$ on $X$ such that 
$H^d$ is bounded and $H+K_X-D$ is ample. 

Since $f^*A-(K_X+B)$ is nef, 
$$
H-B\sim H+K_X-D+f^*A-(K_X+B)
$$ 
is ample. On the other hand, since $Q:=f^*A-K_X-P$ is pseudo-effective, 
$$
H-P=H+K_X-f^*A+Q
$$
is big, hence $|H-P|_\R\neq \emptyset$. Now by [\ref{B-BAB}, Theorem 1.6] (=Theorem \ref{t-bnd-lct}), 
there is a real number $t>0$ 
depending only on $d,H^d,\epsilon$ such that $(X,B+tP)$ is klt. By construction, $t$ depends only on 
$d,r,\epsilon$.

\end{proof}

\subsection{Proofs of \ref{t-bnd-cy-fib}, \ref{t-log-bnd-cy-fib}, \ref{t-log-bnd-cy-gen-fib}, 
\ref{t-sing-FT-fib-totalspace}, \ref{t-bnd-comp-lc-global-cy-fib}}

We are now ready to prove several of the main results of this paper. We apply induction so we assume that 
 \ref{t-log-bnd-cy-fib}, \ref{t-sing-FT-fib-totalspace}, and \ref{t-bnd-comp-lc-global-cy-fib} hold in dimension $d-1$.

\begin{proof}(of Theorems \ref{t-bnd-cy-fib}, \ref{t-log-bnd-cy-fib}, and \ref{t-log-bnd-cy-gen-fib})
Theorems \ref{t-bnd-cy-fib} and \ref{t-log-bnd-cy-fib} follow from Theorems \ref{t-log-bnd-cy-fib}, \ref{t-sing-FT-fib-totalspace}, and 
\ref{t-bnd-comp-lc-global-cy-fib} in dimension $d-1$, and Lemma \ref{l-log-bnd-cy-fib-induction}.
Theorem \ref{t-log-bnd-cy-gen-fib} follows from  Lemma \ref{l-from-gen-fib-to-usual-fib} and Theorem \ref{t-log-bnd-cy-fib}.

\end{proof}

\begin{proof}(of Theorem \ref{t-sing-FT-fib-totalspace})
This follows from Theorems \ref{t-log-bnd-cy-fib}, \ref{t-sing-FT-fib-totalspace}, 
and \ref{t-bnd-comp-lc-global-cy-fib} in dimension $d-1$, and Lemma \ref{l-bnd-cy-bnd-lct-usual}.

\end{proof}

\begin{proof}(of Theorem \ref{t-bnd-comp-lc-global-cy-fib})
This follows from Theorems \ref{t-log-bnd-cy-fib}, \ref{t-sing-FT-fib-totalspace}, and 
\ref{t-bnd-comp-lc-global-cy-fib} in dimension $d-1$, and Lemma \ref{l-bnd-cy-bnd-klt-compl-induction}.

\end{proof}

\section{\bf Generalised log Calabi-Yau fibrations}

In this section we discuss singularities and boundedness of log Calabi-Yau fibrations in the 
context of generalised pairs. 

\subsection{Adjunction for generalised fibrations.}\label{fib-adj-setup}
Consider the following set-up. Assume that 
\begin{itemize}
\item $(X,B+M)$ is a generalised sub-pair with data $X'\to X$ and $M'$,

\item $f\colon X\to Z$ is a contraction with $\dim Z>0$, 

\item $(X,B+M)$ is generalised sub-lc over the generic point of $Z$, and 

\item $K_{X}+B+M\sim_\R 0/Z$.
\end{itemize}
We define the discriminant divisor $B_Z$ for the above setting, similar to the definition in the introduction. 
Let $D$ be a prime divisor on $Z$. Let $t$ be the generalised lc threshold of $f^*D$ with respect to $(X,B+M)$ 
over the generic point of $D$. This makes sense even if $D$ is not $\Q$-Cartier because we only need 
the pullback $f^*D$ over the generic point of $D$ where $Z$ is smooth. 
We then put the coefficient of  $D$ in $B_Z$ to be $1-t$. Note that since $(X,B+M)$ is generalised 
sub-lc over the generic point of $Z$,  $t$ is a real number, that is, it is not $-\infty$ or $+\infty$.
Having defined $B_Z$, we can find $M_Z$ giving 
$$
K_{X}+B+M\sim_\R f^*(K_Z+B_Z+M_Z)
$$
where $M_Z$ is determined up to $\R$-linear equivalence. 
We call $B_Z$ the \emph{discriminant divisor of adjunction} for $(X,B+M)$ over $Z$. 

Let $Z'\to Z$ be a birational contraction from a normal variety. There is 
a birational contraction $X'\to X$ from a normal variety so that the induced map $X'\bir Z'$ is a morphism.  
Let $K_{X'}+B'+M'$ be the 
pullback of $K_{X}+B+M$. We can similarly define $B_{Z'},M_{Z'}$ for $(X',B'+M')$ over $Z'$. In this way we 
get the \emph{discriminant b-divisor ${\bf{B}}_Z$ of adjunction} for $(X,B+M)$ over $Z$. 
Fixing a choice of $M_Z$ we can pick the $M_{Z'}$ consistently so that it also defines a b-divisor ${\bf{M}}_Z$ 
which we refer to as the \emph{moduli b-divisor of adjunction} for $(X,B+M)$ over $Z$. 

\begin{rem}\label{rem-base-fib-gen-pair}
\emph{
Assume that $M=0$, $B$ is a $\Q$-divisor, $(X,B)$ is projective, and that $(X,B)$ is lc over 
the generic point of $Z$. Then ${\bf{M}}_Z$ is b-nef b-$\Q$-Cartier, that is, we can pick $Z'$ so that 
$M_{Z'}$ is a nef $\Q$-divisor and for any resolution $Z''\to Z'$, $M_{Z''}$ is the pullback of $M_{Z'}$ 
[\ref{B-compl}, Theorem 3.6] (this is derived from [\ref{FG-lc-trivial}] which is in turn derived from [\ref{ambro-lc-trivial}]
and this in turn is based on [\ref{kaw-subadjuntion}]). We can then consider 
$(Z,B_Z+M_Z)$ as a generalised pair with nef part $M_{Z'}$.
When $M\neq 0$, the situation is more complicated, see [\ref{Filipazzi}] for recent 
advances in this direction which we will not use in this paper. }
\end{rem}

\smallskip

\subsection{Lower bound for lc thresholds: proof of \ref{t-sing-gen-FT-fib-totalspace}}

\begin{proof}(of Theorem \ref{t-sing-gen-FT-fib-totalspace})
\emph{Step 1.}
\emph{In this step we make some preparations.}
Let $(X,B+M)\to Z$ and $P$ be as in Theorem \ref{t-sing-gen-FT-fib-totalspace} in dimension $d$.
Taking a $\Q$-factorialisation we can assume $X$ is $\Q$-factorial.  
Assume that  $f^*A+B+M-P$ is pseudo-effective.
Since 
$$
f^*A-(K_X+B+M)\sim_\R f^*(A-L)
$$ 
is nef, $2f^*A-K_X-P$ is pseudo-effective. Thus replacing 
$A$ with $2A$, it is enough to treat Theorem \ref{t-sing-gen-FT-fib-totalspace} 
in the case when  $f^*A-K_X-P$ is pseudo-effective.\\

\emph{Step 2.}
\emph{In this step we take a log resolution and introduce some notation.}
Since $B$ is effective and $M$ is pseudo-effective (as it is the pushdown of a nef divisor),
$B+M$ is pseudo-effective. Moreover, since  $-K_X$ is big over $Z$, 
$B+M$ is big over $Z$, hence  $B+M+f^*A$ is big globally.
Let $\phi\colon X'\to X$ be a log resolution of $(X,B)$ 
on which the nef part $M'$ of $(X,B+M)$ resides. Write 
$$
K_{X'}+B'+M'=\phi^*(K_X+B+M).
$$
Since $(X,B+M)$ is generalised $\epsilon$-lc, the coefficients of $B'$ do not exceed $1-\epsilon$.
We can write 
$$
\phi^*(B+M+f^*A)\sim_\R G'+H'
$$ 
where $G'\ge 0$ and $H'$ is general ample. Replacing $\phi$ we can assume $\phi$ is a log resolution of 
$(X,B+P+\phi_*G')$.

Pick a small $\alpha>0$ and pick a general 
$$
0\le R'\sim_\R \alpha H'+(1-\alpha)M'.
$$
Since $M'$ is nef, $\phi^*M=M'+E'$ where $E'$ is effective and exceptional. 
Let 
$$
\Delta':=B'-\alpha \phi^*B-\alpha E'+\alpha G'+R'.
$$
We can make the above choices so that the coefficients of $\Delta'$ do not exceed $1-\frac{\epsilon}{2}$ 
and so that $(X',\Delta')$ is log smooth.\\
 
\emph{Step 3.} 
\emph{In this step we show that $(X,\Delta)\to Z$ is a $(d,r,\frac{\epsilon}{2})$-Fano type fibration 
where $\Delta=\phi_*\Delta'$.}
By construction, we have 
$$
 \begin{array}{l l}
K_{X'}+\Delta' & =K_{X'}+ B'-\alpha \phi^*B-\alpha E'+\alpha G'+R'\\
& \sim_\R K_{X'}+ B'-\alpha \phi^*B-\alpha E'+\alpha G'+\alpha H'+(1-\alpha)M'\\
& \sim_\R K_{X'}+ B'-\alpha \phi^*B-\alpha E'+\alpha \phi^*(B+M+f^*A)+(1-\alpha)M'\\
& \sim_\R K_{X'}+ B'-\alpha \phi^*B-\alpha \phi^* M+\alpha \phi^*(B+M+f^*A)+M'\\
& \sim_\R K_{X'}+ B'+M'+\alpha\phi^*f^*A\sim_\R \phi^*f^*(L+\alpha A).
\end{array}
$$
Therefore,  
$$
K_X+\Delta\sim_\R f^*(L+\alpha A).
$$ 
Choosing $\alpha$ small enough we can ensure $A-(\alpha A+L)$ is 
ample. On the other hand,  since $K_{X'}+\Delta'\sim_\Q 0/X$,  we have $K_{X'}+\Delta'=\phi^*(K_X+\Delta)$, 
hence $(X,\Delta)$ is $\frac{\epsilon}{2}$-lc because 
the coefficients of $\Delta'$ do not exceed $1-\frac{\epsilon}{2}$.
Thus $(X,\Delta)\to Z$ is a $(d,r,\frac{\epsilon}{2})$-Fano type fibration.\\

\emph{Step 4.}
\emph{In this step we finish the proof.}
By Theorem \ref{t-sing-FT-fib-totalspace}, there is a real number $t>0$ 
depending only on $d,r,\epsilon$ such that $(X,\Delta+2tP)$ is klt.
Then letting $P'=\phi^* P$ we see that the coefficients of $\Delta'+2tP'$ do not exceed $1$ as 
$$
K_{X'}+\Delta'+2tP'=\phi^*(K_X+\Delta+2tP).
$$
 Thus 
the coefficients of 
$$
B'-\alpha \phi^*B-\alpha E'+2tP'
$$ 
do not exceed $1$.
Now $t$ is independent of the choice of $\alpha$, so taking the limit as $\alpha$ approaches zero, we see that the 
coefficients of $B'+2tP'$ do not exceed $1$. Therefore, the coefficients of $B'+tP'$ are 
strictly less than $1$ because the coefficients of $B'$ do not exceed $1-\epsilon$, hence $(X,B+tP+M)$ is generalised klt as 
$$
K_{X'}+B'+tP'+M'=\phi^*(K_X+B+tP+M).
$$

\end{proof}

\subsection{Upper bound for the discriminant b-divisor when the base is bounded}

\begin{proof}(of Theorems \ref{t-sh-conj-bnd-base} and \ref{t-sh-conj-bnd-base-gen-fib})
Since \ref{t-sh-conj-bnd-base} is a special case of \ref{t-sh-conj-bnd-base-gen-fib} we treat the latter only.
By induction we can assume that Theorem \ref{t-sh-conj-bnd-base-gen-fib} holds in dimension $d-1$.
Let $(X,B+M)\to Z$ be a generalised $(d,r,\epsilon)$-Fano type fibration.
Let $D$ be  a prime divisor over $Z$. First assume that the centre of $D$ on $Z$ is positive-dimensional. 
Take a resolution $Z'\to Z$ so that $D$ is a divisor on $Z'$. Take a log resolution $\phi\colon X'\to X$ of $(X,B)$ so that 
the nef part $M'$ of $(X,B+M)$  is on $X'$ and that the induced map $f'\colon X'\bir Z'$ is a morphism. 
Replacing $X'$ we can assume that $\phi$ is a log resolution of $(X,B+\phi_*f'^*D)$.

Let $K_{X'}+B'+M'$ be the pullback of $K_X+B+M$.
Let $t$ be the generalised lc threshold of $f'^*D$ with respect to $(X',B'+M')$ over the 
generic point of $D$: this coincides with the lc threshold of $f'^*D$ with respect to $(X',B')$ over the 
generic point of $D$ because $M'$ is nef. Since $(X',B'+tf'^*D)$ is log smooth and since it is sub-lc but not sub-klt 
over the generic point of $D$, 
there is a prime divisor $S$ on $X'$ mapping onto $D$ such that $\mu_SB'+t\mu_Sf'^*D=1$.

Let $H\in |A|$ be a general member  and let $H',G,G'$ be its pullback to $Z',X,X'$, 
respectively. Since the centre of $D$ on $Z$ is positive-dimensional, $H'$ intersects $D$ and 
$G'$ intersects $S$. By divisorial generalised adjunction we can write 
$$
K_G+B_G+M_G\sim_\R (K_X+B+G+M)|_G
$$
where $(G,B_G+M_G)$ is generalised $\epsilon$-lc with nef part $M_{G'}=M'|_{G'}$. 
Moreover, $-K_G$ is big over $H$, and 
$$
K_G+B_G+M_G\sim_\R g^*(L+A)|_H
$$
where $g$ denotes $G\to H$. Thus $(G,B_G+M_G)\to H$ is a generalised $(d,2^{d-1}r,\epsilon)$-Fano type 
fibration. We can write 
$$
K_{G'}+B_{G'}+M_{G'}\sim_\R (K_{X'}+B'+G'+M')|_{G'}
$$
where $B_{G'}=B'|_{G'}$ and $K_{G'}+B_{G'}+M_{G'}$ is the pullback of $K_G+B_G+M_G$. 

Let $C$ be a component of $D\cap H'$ 
and let $s$ be the generalised lc threshold of $g'^*C$ with respect to $(G',B_{G'}+M_{G'})$ over the 
generic point of $C$ where $g'$ denotes $G'\to H'$. Then $1-s$ is the coefficient of $C$ in the 
discriminant b-divisor of adjunction for $(G,B_G+M_G)\to H$. Thus applying 
Theorem \ref{t-sh-conj-bnd-base-gen-fib} in dimension $d-1$, 
we deduce that $1-s\le 1-\delta$ for some $\delta>0$ depending only on  $d,r,\epsilon$. Thus $s\ge \delta$.

By definition of $s$, 
for any prime divisor $T$ on $G'$ mapping onto $C$, 
we have  the inequality $\mu_TB_{G'}+s\mu_Tg'^*C\le 1$. In particular, 
if we take $T$ to be a component of $S\cap G'$ which maps onto $C$, then we have
$$
 \begin{array}{l l}
\mu_SB'+s\mu_Sf'^*D &=\mu_TB'|_{G'}+s \mu_Tf'^*D|_{G'}\\ 
&=\mu_TB_{G'}+s\mu_Tg'^*D|_{H'}\\ 
&=\mu_TB_{G'}+s\mu_Tg'^*C\\ 
&\le 1\\
&=\mu_SB'+t\mu_Sf'^*D
\end{array}
$$
 where we use the fact that over the generic point of $C$ 
the two divisors $g'^*D|_{H'}$ and $g'^*C$ coincide. Therefore, $\delta\le s\le t$, hence 
$\mu_DB_{Z'}=1-t\le 1-\delta$  where $B_{Z'}$ 
is the discriminant divisor on $Z'$ defined for $(X,B+M)$ over $Z$. Thus we have settled the case 
when the centre of $D$ on $Z$ is positive-dimensional.
 
From now on we can assume that the centre of $D$ on $Z$ is a closed point, say $z$. 
Let $Z',X', f', B',M'$ be as before. Pick $N\in |A|$ passing through $z$. 
Then
$$
f^*A+B+M-f^*N\sim B+M
$$ 
is obviously pseudo-effective. Thus by Theorem \ref{t-sing-gen-FT-fib-totalspace} in dimension $d$, 
the generalised lc threshold $u$ of 
$f^*N$ with respect to $(X,B+M)$ is bounded from below by some $\delta>0$ depending only on $d,r,\epsilon$.

Since $N$ passes through $z$,  we have $\psi^*N\ge D$ where $\psi$ denotes $Z'\to Z$. Thus 
the generalised lc threshold $v$ of $f'^*\psi^*N$ with respect to 
$(X',B'+M')$ over the generic point of $D$ is at most as large as the generalised lc threshold $t$ of $f'^*D$ 
with respect to $(X',B'+M')$ over the generic point of $D$. On the other hand, 
the generalised lc threshold $u$ of $f^*N$ with respect to $(X,B+M)$ globally coincides with 
the generalised lc threshold of $f'^*\psi^*N$ with respect to 
$(X',B'+M')$ globally which is at most as large as the generalised lc threshold $v$ of $f'^*\psi^*N$ 
with respect to $(X',B'+M')$ over the generic point of $D$.
Therefore, $\delta\le u\le v\le t$, hence  $\mu_DB_{Z'}=1-t\le 1-\delta$.

\end{proof}

\subsection{Upper bound for the discriminant b-divisor when log general fibres are bounded}

In this subsection, we prove  \ref{t-cb-conj-sing-bnd-gen-fib}, \ref{cor-cb-conj-sing-bnd-gen-fib}, and \ref{t-cb-conj-sing-bnd-fib}. 
We first prove \ref{t-cb-conj-sing-bnd-gen-fib} when the base is one-dimensional. We use ideas similar to the proof of 
[\ref{B-sing-fano-fib}, Theorem 1.4]. 

\begin{prop}\label{l-adj-disc-div-e-lc}
Theorem \ref{t-cb-conj-sing-bnd-gen-fib} holds when $\dim Z=1$.
\end{prop}
\begin{proof}
\emph{Step 1.}
\emph{In this step we do some preparations and introduce a boundary $\Delta'$ on $X'$.}
Let $D$ be a prime divisor on $Z$. We want to show that the coefficient $\mu_DB_Z$ 
is bounded from above away from $1$ where $B_Z$ is the discriminant divisor of 
adjunction of $(X,B+M)$ over $Z$. Since this is a local problem near $D$ we will 
shrink $Z$ around $D$ if necessary.
Taking a $\Q$-factorialisation, we can assume $X$ is $\Q$-factorial.

Denote the given morphism $X'\to X$  by $\phi$. 
Replacing $\phi$ we can assume it is a log resolution of $(X,B+G+f^*D)$.   Write 
$K_{X'}+B'+M'$, $G'$ for the pullbacks of $K_X+B+M$, $G$, respectively. Let $\Sigma'$ 
be  the birational transform of the horizontal over $Z$ part of $\Supp (B+G)$ union the horizontal 
over $Z$ exceptional divisors of $\phi$. Denote $X'\to Z$ by $f'$, and let  
$$
\Delta'=(1-\frac{\epsilon}{2})\Sigma'+\Supp f'^*D.
$$ 
Shrinking $Z$ we can assume that $\Supp \Delta'$ coincides with the reduced exceptional divisor of $\phi$ union 
the birational transform of $\Supp (B+G+f^*D)$ (so we get rid of divisors 
which are vertical over $Z$ but do not map to $D$).\\

\emph{Step 2.}
\emph{In this step we study $K_{X'}+\Delta'+2M'$.}
By construction, $(X',\Delta'+2M')$ is generalised lc globally and  generalised $\frac{\epsilon}{2}$-lc 
over  $Z\setminus \{D\}$. Moreover,  
the coefficients of $\Delta'$ belong to $\{1-\frac{\epsilon}{2}, 1\}$ and $\rddown{\Delta'}=\Supp f'^*D$. 
Since $(X,B+M)$ is generalised $\epsilon$-lc, the coefficients of 
$B'$ are at most $1-\epsilon$. Furthermore, $\Supp B'\subseteq \Sigma'+\Supp f'^*D$. Thus we have
$$
B'\le (1-\epsilon)(\Sigma'+\Supp f'^*D)
$$
which in turn gives 
$$
\Delta'-B'\ge \Delta' - (1-\epsilon)(\Sigma'+\Supp f'^*D) = \frac{\epsilon}{2}\Sigma'+\epsilon\Supp f'^*D. 
$$
On the other hand, $\phi^*M-M'$ is effective and exceptional over $X$, so we can write 
$$
\phi^*((\Supp B)+M+ G)=M'+N'
$$
for some $N'\ge 0$ with 
$$
\Supp N'\subseteq \Sigma'+\Supp f'^*D.
$$
In particular, $\Delta'-B'\ge \alpha N'$ for some small $\alpha>0$.  

Now since 
$$
0<\vol((\Supp B)+M+G)|_F)
$$ 
for the general fibres $F$ of $f$, $(\Supp B)+M+G$ is big over $Z$, hence $M'+N'$ is big over $Z$. Then 
since $M'$ is nef over $Z$, 
$$
M'+\alpha N'=(1-\alpha)M'+\alpha(M'+N')
$$ 
is big over $Z$. This in turn implies that $M'+\Delta'-B'$ is big over $Z$ because $\Delta'-B'\ge \alpha N'$.
Therefore, from 
$$
K_{X'}+\Delta'+2M'\sim_\R K_{X'}+\Delta'+2M'-(K_{X'}+B'+M')=M'+\Delta'-B'/Z
$$
we deduce that $K_{X'}+\Delta'+2M'$ is big over $Z$. 
Also note that by assumption $pM'$ is Cartier.\\ 

\emph{Step 3.}
\emph{In this step we show that $(X',\Delta'+2M')$ has a generalised lc model over $Z$, that is, 
an ample model over $Z$.}  
We have 
$$
K_{X'}+\Delta'+2M'\sim_\R K_{X'}+\Delta'-\alpha f'^*D+2M'/Z.
$$
Choosing $\alpha$ to be small enough we can ensure that 
$$
\Theta':=\Delta'-\alpha f'^*D\ge 0.
$$  
Then $(X',\Theta'+2M')$ is generalised klt as $\rddown{\Delta'}=\Supp f'^*D$, and $K_{X'}+\Theta'+2M'$ is big over $Z$.
Thus we can run an MMP on $K_{X'}+\Theta'+2M'$ over $Z$
terminating with a minimal model, say $\tilde{X}''$, on which $K_{\tilde{X}''}+\tilde{\Theta}''+2\tilde{M}''$ is semi-ample 
over $Z$ [\ref{BZh}, Lemma 4.4], hence defining a contraction $\tilde{X}''\to X''/Z$. As
$$
K_{X'}+\Delta'+2M' \sim_\R K_{X'}+\Theta'+2M'/Z,
$$
$X''$ is also the generalised lc model of $({X'},\Delta'+2M')$ over $Z$. In particular, 
$({X''},\Delta''+2M'')$ is generalised lc with nef part being the pullback of $M'$ to some 
common resolution of $X',X''$, and $K_{X''}+\Delta''+2M''$ is ample over $Z$.\\ 

\emph{Step 4.}
\emph{In this step we obtain lower bound for the volume of $K_{X''}+\Delta''+2M''$ restricted to 
components of the fibre of $X''\to Z$ over $D$}.
Let $S$ be the normalisation of a component $T$ of $f''^*D$ where $f''$ is the morphism $X''\to Z$. 
Since $T$ is a component of $\rddown{\Delta''}$, by generalised divisorial adjunction 
[\ref{B-compl}, Subsection 3.1], we can write 
$$
K_S+\Delta_S''+2M_S''\sim_\R (K_{X''}+\Delta''+2M'')|_S
$$
such that $(S,\Delta_S''+2M_S'')$ is a generalised pair data $\overline{S}\to S$ and $M_{\overline{S}}$ 
where the nef part $M_{\overline{S}}$ is the restriction of the nef part of 
$(X'',\Delta''+2M'')$. Since the coefficients of $\Delta''$ are in a fixed finite set and since $pM'$ is Cartier, 
the coefficients of $\Delta_S''$ are in a fixed DCC set $\Psi$ and $M_{\overline{S}}$ is b-Cartier 
[\ref{B-compl}, Lemma 3.3]. 
Moreover, $(S,\Delta_S''+2M_S'')$ is generalised lc and $K_S+\Delta_S''+2M_S''$ is ample. 

Let $\Lambda_{\overline S}$ be the sum of the reduced exceptional divisor of $\overline{S}\to S$ 
and the birational transform of $\Delta_S''$. 
Applying [\ref{BZh}, Theorem 1.3], we find a natural number $m$ depending only on 
$d,p,\Psi$ such that 
$$
|\rddown{m(K_{\overline{S}}+\Lambda_{\overline{S}}+2M_{\overline{S}})}|
$$ 
defines a birational map, hence 
$$
|\rddown{m(K_S+\Delta_S''+2M_S'')}|
$$ 
also defines a birational map. In particular, 
$$
(K_S+\Delta_S''+2M_S'')^{d-1}=\vol(K_S+\Delta_S''+2M_S'')\ge \frac{1}{m^{d-1}}.
$$\

\emph{Step 5.}
\emph{In this step we study the intersection number of $K_{X''}+\Delta''+2M''$ with the fibres of $f''$.}
Write $f''^*D=\sum m_iT_i$ where $T_i$ are irreducible components, and let $S_i$ be the normalisation of $T_i$. 
In later steps we will show that the $m_i$ are bounded from above. As in the previous step we write 
$$
K_{S_i}+\Delta_{S_i}''+2M_{S_i}''\sim_\R (K_{X''}+\Delta''+2M'')|_{S_i}.
$$
Then 
$$
(K_{S_i}+\Delta_{S_i}''+2M_{S_i}'')^{d-1}=(K_{X''}+\Delta''+2M'')^{d-1}\cdot T_i
$$
where to define the latter intersection number we use the fact that $X''\to Z$ is a projective morphism 
over a curve. In particular, 
$$
 \begin{array}{l l}
\sum m_i (K_{S_i}+\Delta_{S_i}''+2M_{S_i}'')^{d-1} & =\sum m_i(K_{X''}+\Delta''+2M'')^{d-1}\cdot T_i\\
&=(K_{X''}+\Delta''+2M'')^{d-1}\cdot (\sum m_iT_i)\\
& =(K_{X''}+\Delta''+2M'')^{d-1}\cdot f''^*D.
\end{array}
$$
Thus  if $F''$ is a general fibre of $f''$, then since $f''^*D\sim F''$ we get 
$$
 \begin{array}{l l}
\sum m_i (K_{S_i}+\Delta_{S_i}''+2M_{S_i}'')^{d-1} & =(K_{X''}+\Delta''+2M'')^{d-1}\cdot {F''}\\
&=((K_{X''}+\Delta''+2M'')|_{F''})^{d-1}\\
&=\vol((K_{X''}+\Delta''+2M'')|_{F''}).
\end{array}
$$\ 

\emph{Step 6.}
\emph{In this step we show that $\vol((K_{X''}+\Delta''+2M'')|_{F''})$ is bounded from above.}
Indeed let $F,F'$ be the fibres of $f,f'$ corresponding to $F''$. Since $(X'',\Delta''+2M'')$ 
is the generalised lc model of $(X',\Delta'+2M')$ and since $F''$ is a general fibre, we have  
$$
\vol((K_{X''}+\Delta''+2M'')|_{F''})\le \vol((K_{X'}+\Delta'+2M')|_{F'}).
$$ 
Actually equality holds but we do not need it.
On the other hand, 
$$
\vol((K_{X'}+\Delta'+2M')|_{F'})\le \vol((K_{X}+\Delta+2M)|_{F})
$$ 
where $\Delta$ is the pushdown of $\Delta'$, because the pushdown of $(K_{X'}+\Delta'+2M')|_{F'}$ 
to $F$ is $(K_{X}+\Delta+2M)|_{F}$ (we are using the assumption that $F$ is a general fibre). 

Let $\Sigma$ be the pushdown of $\Sigma'$, that is, $\Sigma$ is the support of the horizontal$/Z$ 
part of $B+G$. In particular, $\Sigma\le (\Supp B)+G$, hence 
$$
\Sigma +M\le (\Supp B)+M+G
$$
which implies that 
$$
\vol((\Sigma+M)|_F)\le \vol(((\Supp B)+M+G)|_F)<v.
$$
By construction, 
$$
\Delta=(1-\frac{\epsilon}{2})\Sigma+\Supp f^*D,
$$
so
$$
 \begin{array}{l l}
\vol((K_{X}+\Delta+2M)|_{F}) &=\vol((K_{X}+(1-\frac{\epsilon}{2})\Sigma+2M)|_{F})\\
& \le \vol((K_{X}+\Sigma+2M)|_{F})\\
&= \vol((K_{X}+\Sigma+2M -K_X-B-M)|_{F})\\
&=\vol((\Sigma+M-B)|_F)\\
&\le \vol((\Sigma+M)|_F)\\
& <v.
\end{array}
$$
Therefore,  $\vol((K_{X''}+\Delta''+2M'')|_{F''})<v$.\\

\emph{Step 7.}
\emph{In this step we show that the $m_i$ are bounded from above.}
Recall from Step 5 that 
$$
\vol((K_{X''}+\Delta''+2M'')|_{F''})=\sum m_i (K_{S_i}+\Delta_{S_i}''+2M_{S_i}'')^{d-1}.
$$
By the previous step, the left hand side is bounded from above. 
On the other hand, by Step 4,
$$
\sum m_i (K_{S_i}+\Delta_{S_i}''+2M_{S_i}'')^{d-1}\ge \sum \frac{m_i}{m^{d-1}}.
$$
Therefore, the right hand side is bounded from above, hence the $m_i$ are all bounded from above.\\ 

\emph{Step 8.}
\emph{In this final step we finish the proof.} 
We will denote the pushdown of $B'$ to $X''$ by $B''$, etc. By Step 2 we get 
$$
\Delta''-B''\ge \frac{\epsilon}{2}\Sigma''+\epsilon\Supp f''^*D. 
$$
 Thus since $f''^*D=\sum m_iT_i$ with $m_i$ bounded,  
there is a positive real number ${\delta}$ bounded from below away from zero 
such that 
$$
Q'':=\Delta''-B''-\delta f''^*D\ge 0.
$$
 Then 
$$
K_{X''}+\Delta''+2M''-(K_{X''}+B''+\delta f''^*D+M'')=Q''+M''
$$
which in particular means that $Q''+M''$ is $\R$-Cartier. Therefore, since $(X'',\Delta''+2M'')$ is generalised lc,  
$$
(X'',B''+\delta f''^*D+M'')
$$ 
is generalised sub-lc. This implies that $(X,B+\delta f^*D+M)$ is generalised lc because  
the pullbacks of 
$$
K_X+B+\delta f^*D+M
$$
and of 
$$
K_{X''}+B''+\delta f''^*D+M''
$$ 
agree on any common resolution of $X,X''$ as both divisors are $\R$-linearly trivial over $Z$. 
Therefore, 
$t\ge \delta$ where $t$ is the generalised lc threshold of $f^*D$ with respect to $(X,B+M)$, hence 
  $\mu_DB_Z=1-t\le 1-\delta$.

\end{proof}

\begin{proof}(of Theorem \ref{t-cb-conj-sing-bnd-gen-fib})
\emph{Step 1.}
\emph{In this step we introduce a boundary $\Delta'$ on some resolution of $X$.}
We will reduce the theorem to Proposition \ref{l-adj-disc-div-e-lc}. 
Taking a $\Q$-factorialisation we can assume $X$ is $\Q$-factorial.
Assume $D$ is a prime divisor over $Z$. First we reduce the statement to the case when $D$ is a divisor on $Z$. 
Let $Z'\to Z$ be a resolution such that $D$ is a divisor on $Z'$. Pick a log resolution $\phi\colon X'\to X$ of $(X,B)$ such that 
$f'\colon X'\bir Z'$ is a morphism and such that the nef part $M'$ of $(X,B+M)$ is on $X'$.
Write $K_{X'}+B'+M'$ for the pullback of $K_{X}+B+M$.

Let $\Gamma'$ be obtained from $B'$ by removing all the components with negative coefficients, and let 
$\Delta'=\Gamma'+\frac{\epsilon}{2}R'$ where $R'$ is the reduced exceptional divisor of $X'\to X$. Then
$(X',\Delta'+M')$ is generalised $\frac{\epsilon}{2}$-lc, and 
we get 
$$
K_{X'}+\Delta'+M'\sim_\R K_{X'}+\Delta'+M'-(K_{X'}+B'+M')= \Delta'-B'/Z
$$
where $E':=\Delta'-B'\ge \frac{\epsilon}{2}R'$ and $E'$ is exceptional$/X$.\\ 

\emph{Step 2.}
\emph{In this step we consider running MMP on $K_{X'}+\Delta'+M'$ over $Z$.}
Run an MMP on $K_{X'}+\Delta'+M'$ over the main component of 
$X\times_Z{Z'}$ with scaling of some ample divisor. So the MMP is over both $X$ and $Z'$. 
We do not claim that the MMP terminates but 
 we claim that it does terminate over the generic point of $Z$ contracting all the horizontal$/Z$ components of 
 $E'$. Indeed let $U$ be a non-empty open subset of $Z$ over which $Z'\to Z$ is an isomorphism.
Then $X\times_Z{Z'}\to X$ is an isomorphism over $f^{-1}U$, and since $E'$ is exceptional over 
$X$, the MMP terminates over $f^{-1}U$, by Lemma \ref{l-mmp-v-exc}, 
contracting $E'$ over $f^{-1}U$. Thus we reach a model $X''$ on which  $E''=0$ 
over $f^{-1}U$, in particular, $E''$ is vertical over $Z'$. Moreover, since $E''$ contains all the exceptional divisors of 
$X''\to X$ and since $X$ is $\Q$-factorial, we deduce that $X''\to X$ is an isomorphism over $U$. 

Let $G'=\rddown{\phi^*G}$ and let $G''$ be its pushdown on $X''$. Then by the previous paragraph, 
$$
0<\vol(((\Supp \Delta'')+M''+G'')|_{F''})=\vol(((\Supp B)+M+G)|_{F})<v
$$
where $F''$ is a general fibre of $X''\to Z$ and $F$ is the corresponding fibre of $X\to Z$.\\

 \emph{Step 3.}
\emph{In this step we consider running MMP on $K_{X''}+\Delta''+M''$ over $Z'$.}
Now run another MMP on $K_{X''}+\Delta''+M''$ over $Z'$ with scaling of some ample divisor. 
Since $K_{X''}+\Delta''+M''\sim_\R 0$ holds over $U$, the MMP does not do anything over $U$. 
We claim that the 
MMP terminates over the generic point of $D$. We can assume $E''\neq 0$ otherwise the claim holds trivially. 
For the rest of this paragraph we shrink $Z'$ around the generic point of $D$ hence assume that every component of $E''+f''^*D$ 
maps onto $D$  where $f''$ is the induced morphism $X''\to Z'$. 
Let $\alpha$ be the largest real number such that $E''-\alpha f''^*D$ is effective ($\alpha\ge 0$ and $\alpha=0$ is possible). 
There is a component of $ f''^*D$ which is not a component of $E''-\alpha f''^*D$. Thus $E''-\alpha f''^*D$ is very 
exceptional over $Z'$. Therefore, the MMP terminates by Lemma \ref{l-mmp-v-exc} contracting 
$E''-\alpha f''^*D$. In particular, the MMP terminates over the generic point of $D$.\\

\emph{Step 4.}
\emph{In this step we reduce the theorem to the case when $D$ is a divisor on $Z$.}
In the course of the MMP of last step we reach a model $X'''$ on which  
 $(X''',\Delta'''+M''')$ is generalised $\frac{\epsilon}{2}$-lc with nef part being the pullback of $M'$ 
 to some common resolution of $X',X'''$, and that 
$$
K_{X'''}+\Delta'''+M'''\sim_\R 0/Z'
$$
holds over the generic point of $D$.
Moreover, if $K_{X'''}+B'''+M'''$ is the pushdown of $K_{X'}+B'+M'$, then $B'''\le \Delta'''$. 
In particular, the generalised lc threshold of $f'''^*D$ with respect to $(X''',\Delta'''+M''')$ over the 
generic point of $D$ is smaller or equal to the generalised lc threshold with respect to $(X''',B'''+M''')$ 
where $f'''$ is the induced morphism $X'''\to Z'$. Furthermore, by Step 2, 
$$
0<\vol(((\Supp \Delta''')+M'''+G''')|_{F'''})<v
$$
for the general fibres $F'''$ of $f'''$ because $X''\bir X'''$ is an isomorphism over the generic point of $Z'$.
Therefore, shrinking $Z'$ around the generic point of $D$ and replacing $\epsilon$, 
$(X,B+M)\to Z$ with $\frac{\epsilon}{2}$, $(X''',\Delta'''+M''')\to Z'$,
 we can assume that $D$ is a divisor on $Z$.\\

\emph{Step 5.}
\emph{In this step we take a hyperplane section of $Z$.}
In this step assume $\dim Z>1$. Let $H$ be a 
general hyperplane section of $Z$ and let $V=f^*H$. Then by divisorial generalised adjunction 
we can write 
$$
K_V+B_V+M_V\sim_\R (K_X+B+V+M)|_V
$$
where $(V,B_V+M_V)$ is generalised $\epsilon$-lc with nef part $M_{V'}=M'|_{V'}$ where  
$V'\subset X'$ is the pullback of $V$. Let $G_V:=G|_V$. 
By generality of $H$, we have $B_V=B|_V$ and $M_V=M|_V$, hence 
$$
0<\vol(((\Supp B_V)+M_V+G_V)|_{F})=\vol(((\Supp B)+M+G)|_{F})<v
$$
for the general fibres $F$ of $V\to H$ because $F$ is among the general fibres of $f$. 
Moreover, $pM_{V'}$ is Cartier as $pM'$ is Cartier by assumption,  and 
$$
K_V+B_V+M_V\sim_\R 0/H.
$$\ 

\emph{Step 6.}
\emph{In this step we finish the proof by applying induction on dimension.}
If $\dim Z=1$, then we use Proposition \ref{l-adj-disc-div-e-lc}.
Otherwise we apply induction on $\dim Z$ as follows. Let $V,H$ etc be as in the previous step.
Let $C$ be a component of $D\cap H$ 
and let $s$ be the generalised lc threshold of $g^*C$ with respect to $(V,B_V+M_V)$ over the 
generic point of $C$ where $g$ denotes $V\to H$. Applying induction on dimension, 
$s$ is bounded from below away from zero, hence it is enough to show that 
$s\le t$ where $t$ is the generalised lc threshold of $f^*D$ with respect to $(X,B+M)$ over the generic point 
of $D$. 

Shrinking $Z$ we can assume $C=D|_H$, hence $f^*D|_V=g^*C$.
By definition of $s$, 
$$
(V,B_V+sg^*C+M_V)
$$ 
is generalised lc over the 
generic point of $C$. Shrinking $Z$ around the generic point of $C$ we can assume that it is generalised lc everywhere. 
But then by generalised inversion of adjunction [\ref{B-compl}, Lemma 3.2], 
$$
(X,B+V+s f^*D+M)
$$ 
is generalised lc near $V$, hence it is generalised lc over a neighbourhood of $H$ which then implies that it is 
generalised lc over the generic point of $D$ as $H$ intersects $D$. Thus $s\le t$ as required. 

\end{proof}

In the proof just completed we first changed the base $Z$ so that we could assume $D$ is a 
divisor on $Z$. It is worth pointing out that this strategy does not work when dealing with 
Theorem \ref{t-sh-conj-bnd-base-gen-fib} because in this case we need to keep $Z$ 
varying in a bounded family. That is why the proof of \ref{t-sh-conj-bnd-base-gen-fib} 
is different in the sense that we use hyperplane sections of $Z$ only when the centre of 
$D$ on $Z$ is positive-dimensional.

\begin{proof}(of Corollary \ref{cor-cb-conj-sing-bnd-gen-fib})
We want to prove Conjectures \ref{conj-sh-sing-gen-fib} and \ref{conj-cb-sing-gen-fib} 
under the extra assumptions that: any horizontal$/Z$ component of $B$ has coefficient $\ge \tau$ and 
$pM'$ is b-Cartier. We can assume $\tau<1$. First consider \ref{conj-cb-sing-gen-fib}. 
Since $1<\frac{1}{\tau}$ and since $M+G$ is pseudo-effective over $Z$,
$$
 \begin{array}{l l}
0<\vol((B+M+G)|_{F})& \le \vol(((\Supp B)+M+G)|_{F})\\
\le \vol(\left(\frac{1}{\tau}B+M+G\right))|_{F}) &\le \vol(\left(\frac{1}{\tau}(B+M+G\right))|_{F})<\frac{v}{\tau^d}
\end{array}
$$
for the general fibres $F$ of $f$. Thus we can apply Theorem \ref{t-cb-conj-sing-bnd-gen-fib}. 

Now consider \ref{conj-sh-sing-gen-fib}. Since $-K_X$ is big over $Z$, $B+M$ is big over $Z$, so
$$
0<\vol((B+M)|_F)=\vol(-K_X|_F)=\vol(-K_F)
$$ 
for the general fibres $F$ of $f$. Letting $B_F:=B|_F$ and $M_F:=M|_F$, we see that 
$(F,B_F+M_F)$ is generalised $\epsilon$-lc, $K_F+B_F+M_F\sim_\R 0$, and $B_F+M_F$ is big. 
We can then find a big boundary $\Delta_F$ such that $(F,\Delta_F)$ is $\frac{\epsilon}{2}$-lc 
and $K_F+\Delta_F\sim_\R 0$ (this follows from \ref{l-from-gen-fib-to-usual-fib}).
Thus $F$ belongs to a bounded family by [\ref{B-BAB}, Corollary 1.2], 
hence $\vol(-K_F)$ is bounded from above. Then  $\vol((B+M)|_F)$ is bounded from above, so 
taking $G=0$ we are in the situation of \ref{conj-cb-sing-gen-fib}. Thus we are done by the previous 
paragraph.

\end{proof}

\begin{proof}(of Theorem \ref{t-cb-conj-sing-bnd-fib})
This is a special case of Theorem \ref{t-cb-conj-sing-bnd-gen-fib} which was already proved.

\end{proof}

In the rest of this section we prove few other results which are not essential for this paper in the 
sense that they will only be used to give alternatives proofs of \ref{t-towers-of-Fanos}. They will 
likely be useful elsewhere so it is good to write them here for future reference.

\subsection{Comparing singularities on the total space and base}

\begin{lem}\label{l-adj-sub-lc-g-sub-lc}
Let $(X,B)$ be a projective sub-pair and $f\colon X\to Z$ be a contraction such that $(X,B)$ is lc over the 
generic point of $Z$, $K_X+B\sim_\Q 0/Z$, and $B$ is a $\Q$-divisor. 
Let $B_Z,M_Z$ be the discriminant and moduli parts of adjunction and consider $(Z,B_Z+M_Z)$ as a 
generalised pair (as in \ref{rem-base-fib-gen-pair}). Then for any open subset $U\subseteq  Z$, $(X,B)$ is sub-lc over $U$ iff 
$(Z,B_Z+M_Z)$ is generalised sub-lc on $U$.
\end{lem}
\begin{proof}
Choose a log resolution $Z'\to Z$ of $(Z,B_Z)$ such that 
$M_{Z'}$ is nef and ${\bf{M}}_Z$ is the b-divisor determined by $M_{Z'}$, that is, for any higher resolution 
$Z''\to Z'$ the divisor $M_{Z''}$ is the pullback of $M_{Z'}$.  
Pick a log resolution $X'\to X$ of $(X,B)$ such that the induced map $f'\colon X'\bir Z'$ is a 
morphism. Let $K_{X'}+B'$ be the pullback of $K_X+B$.
Let $U\subseteq  Z$ be an open subset. 

Assume that $(X,B)$ is sub-lc over $U$. Let $D$ be a prime divisor on $Z'$ whose 
centre on $Z$ intersects $U$.  Since $(X',B')$ is sub-lc over $U$, 
the lc threshold of $f'^*D$ with respect to $(X',B')$ over the generic point of $D$ is non-negative, hence 
 the coefficient of $D$ in $B_{Z'}$ is at most $1$. Therefore, the coefficients of the components of $B_{Z'}$ 
 whose generic point map to $U$ do not exceed $1$, so $(Z',B_{Z'}+M_{Z'})$ is generalised sub-lc over $U$ 
 which means that $(Z,B_Z+M_Z)$ is generalised sub-lc on $U$. 

Conversely assume that $(Z,B_Z+M_Z)$ is generalised sub-lc on $U$. Assume $(X,B)$ is not sub-lc over $U$. 
Then there is a prime divisor $S$ over $X$ with log discrepancy $a(S,X,B)<0$ whose image on $Z$ intersects $U$. Since 
$(X,B)$ is an lc pair over the generic point of $Z$, $S$ is vertical over $Z$. 
Thus replacing $X',Z'$ we can assume that $S$ is a divisor on $X'$ and that the image of $S$ on $Z'$ is a divisor, say $D$.
Since by assumption the coefficient of $S$ in $B'$ exceeds $1$, 
the lc threshold of $f'^*D$ with respect to $(X',B')$ over the generic point of $D$ is negative, hence 
the coefficient of $D$ in $B_{Z'}$ exceeds $1$, a contradiction. Therefore, 
$(X,B)$ is sub-lc over $U$. 

\end{proof}

\subsection{Composition of contractions}

\begin{lem}\label{l-adj-composition-contractions}
Let $(X,B)$ be a projective sub-pair and $X\overset{f}\to Y\overset{g}\to Z$ be contractions such that $(X,B)$ is lc 
over the generic point of $Z$, $K_X+B\sim_\Q 0/Z$, and $B$ is a $\Q$-divisor. Let 
\begin{itemize}
\item ${\bf{B}}_Y,{\bf{M}}_Y$ (resp. $B_Y,M_Y$)  be the 
discriminant and moduli b-divisors (resp. divisors) of adjunction for $(X,B)$ over $Y$, 

\item  ${\bf{B}}_Z,{\bf{M}}_Z$  be the 
discriminant and moduli b-divisors of adjunction for $(X,B)$ over $Z$, 

\item and ${\bf{C}}_Z$ be the discriminant b-divisor of adjunction for $(Y,B_Y+M_Y)$ over $Z$ 
where we consider $(Y,B_Y+M_Y)$ as a generalised pair (as in \ref{rem-base-fib-gen-pair}). 
\end{itemize}
Then ${\bf{C}}_Z={\bf{B}}_Z$.
\end{lem}
\begin{proof}
Let $D$ be a prime divisor over $Z$, say on some resolution $Z'\to Z$. 
Let $c,b$ be the coefficients of $D$ in ${\bf{C}}_Z,{\bf{B}}_Z$, respectively. We want to show $c=b$.
Pick birational contractions $\psi\colon Y'\to Y$ 
and $\phi\colon X'\to X$ from normal varieties so that $\psi,\phi$ are isomorphisms over the generic point of $Z$ 
and so that the induced maps $g'\colon Y'\bir Z'$ and $f'\colon X'\bir Y'$ are morphisms.  
Let $K_{X'}+B'$ be the pullback of $K_X+B$. Then $(X',B')$ is lc over the generic point of $Z$. 
Moreover, the discriminant and moduli divisors ${{B}}_{Y'}',{{M}}_{Y'}'$ defined for $(X',B')$ 
over $Y'$ coincide with the discriminant and moduli divisors 
${{B}}_{Y'},{{M}}_{Y'}$ on $Y'$ defined for $(X,B)$ over $Y$. 
Similarly,  the discriminant and moduli b-divisors ${\bf{B}}_{Z'}',{\bf{M}}_{Z'}'$ of adjunction for $(X',B')$ over $Z'$ 
coincide with the discriminant and moduli b-divisors ${\bf{B}}_{Z},{\bf{M}}_{Z}$ of adjunction for $(X,B)$ over $Z$,   
and the discriminant b-divisor ${\bf{C}}_{Z'}'$ of adjunction for $(Y',B_{Y'}'+M_{Y'}')$ over $Z'$ 
coincides with the discriminant b-divisor ${\bf{C}}_Z$  of adjunction for $(Y,B_Y+M_Y)$ over $Z$. 
Thus replacing $(X,B),f,g$ with $(X',B'),f',g'$ we can assume $D$ is a divisor on $Z$ and that $Z$ is smooth.

Put $h=gf$. Let $t$ be the lc threshold of $h^*D$ with respect to $(X,B)$ over the generic point of $D$. 
Similarly let $s$ be the generalised lc threshold of $g^*D$ with respect to $(Y,B_{Y}+M_{Y})$ 
over the generic point of $D$. By definition, $b=1-t$ and $c=1-s$. 
On the other hand, for any $\R$-Cartier divisor 
$P_Y$ on $Y$,  $B_{Y}+P_Y$ is the discriminant divisor of adjunction for $(X,B+ f^*P_Y)$ over $Y$. In particular, 
for any real number $u$, $B_Y+ug^*D$ is the discriminant divisor of $(X,B+ uh^*D)$ over $Y$.
Now by Lemma \ref{l-adj-sub-lc-g-sub-lc}, $(X,B+u h^*D)$ is sub-lc over the generic point of $D$ iff  
$$
(Y,B_{Y}+u g^*D+M_{Y})
$$ 
is generalised sub-lc over the generic point of $D$. Applying this to $u=t$ and $u=s$ shows that 
$t=s$ which in turn shows that $b=c$.

\end{proof}

\subsection{DCC property of the discriminant divisor}

\begin{lem}\label{l-fib-adj-dcc}
Let $d,p$ be natural numbers and $\Phi\subset [0,1]$ be a DCC set. Then there is a DCC set 
$\Psi\subset [0,1]$ depending only on $d,p,\Phi$ satisfying the following. 
Assume that $(X,B+M)$ and $X\to Z$ are as in \ref{fib-adj-setup} and that 
\begin{itemize}
\item $(X,B+M)$ is generalised lc of dimension $d$, 

\item the coefficients of $B$ are in $\Phi$, and 

\item $pM'$ is b-Cartier where $M'$ is the nef part of $(X,B+M)$. 

\end{itemize}
Then the discriminant divisor $B_Z$ of adjunction for  $(X,B+M)$ over $Z$ has coefficients in $\Psi$.  
\end{lem}
\begin{proof}

Let $D$ be a prime divisor on $Z$. Let $t$ be the generalised lc threshold of $f^*D$ with respect to $(X,B+M)$ 
over the generic point of $D$. Shrinking $Z$ around the generic point of $D$ we can assume $D$ is Cartier and that 
$t$ is the generalised lc threshold of $f^*D$ with respect to $(X,B+M)$ (that is, globally not just over 
the generic point of $D$). In particular, 
the coefficients of $f^*D$ are natural numbers, hence they belong to $\Phi\cup \N$ which is a DCC set.    
Moreover, we can assume that $\frac{1}{p}$ is in $\Phi$. Then 
by [\ref{BZh}, Theorem 1.5] the generalised lc thresholds $t$ above satisfy the ACC. 
Therefore, $\mu_DB_Z=1-t$ belongs to some 
DCC set $\Psi$ depending only on $d,p,\Phi$. 

\end{proof}

Note that in the proof, unlike some other proofs above, we did not need $M'$ to be nef over $Z$ 
but only used its nefness over $X$ 
(indeed $M'$ is not assumed to be nef over $Z$ in the lemma).


\section{\bf Boundedness of towers of Fano fibrations and of log Calabi-Yau varieties}

In this section we treat Theorems \ref{t-towers-of-Fanos} and \ref{cor-bnd-cy-fib-non-product}.

\begin{proof}(of Theorem \ref{t-towers-of-Fanos})
\emph{Step 1.}
\emph{In this step we do some easy reductions.}
By assumption, $X\to Z$ factors as a sequence 
$$
X=X_1\to \cdots \to Z=X_l
$$ 
of Fano fibrations. If $l=1$, then the statement holds essentially trivially: indeed, $X=Z$ 
and $A^{d=\dim Z}\le r$ means $X$ is bounded; 
also $A-L\sim_\R A-(K_X+B)$ being ample implies $A^{d-1}\cdot (K_X+B)<r $, hence 
$A^{d-1}\cdot B$ is bounded from above which then implies that 
$(X,B)$ is log bounded as the coefficients of $B$ are $\ge \tau$. Thus we can assume $l\ge 2$. 
Moreover, applying induction on $l$ we can assume that $\dim X_{l-1}>0$ otherwise 
we can replace $Z$ with $X_{l-1}$ and decreasing $l$.

On the other hand, by Lemma \ref{l-rational-approximation}, 
we can write $K_X+B=\sum r_i(K_X+B_i)$ for certain real numbers $r_i>0$ with 
$\sum r_i=1$ and rational boundaries $B_i$ such that $(X,B_i)$ is $\frac{\epsilon}{2}$-lc, $K_{X}+B_i\sim_\Q 0/Z$, 
$\Supp B_i=\Supp B$, and the coefficients of $B_i$ are $\ge \frac{\tau}{2}$. 
Furthermore, we can choose $B_i$ so that the coefficients of $B-B_i$ are arbitrarily small, hence 
if $K_X+B_i\sim_\Q f^* L_i$, then we can make sure that $A-L_i$ is ample. Now replacing $\epsilon, \tau, (X,B),L$ with 
$\frac{\epsilon}{2},\frac{\tau}{2},(X,B_i),L_i$ for some $i$ we can assume that $B$ has rational coefficients.\\

\emph{Step 2.}
\emph{In this step we apply adjunction to $(X,B)$ over each $X_i$.}
Let $B_i$ and $M_i$ be the discriminant and moduli divisors 
of adjunction defined for $(X,B)=(X_1,B_1)$ over $X_i$ whenever $\dim X_i>0$. We consider $(X_i, B_i+M_i)$ as a 
generalised pair with data consisting of some 
high resolution $X_i'\to X_i$ and nef part $M_i'$ on $X_i'$ (as in \ref{rem-base-fib-gen-pair}). 

Denote $X\to X_j$ by $g_j$. We claim that for each $i$,
\begin{enumerate}
\item there exists a positive real number $\delta$ depending only on $d,i,\epsilon,\tau$ 
such that $(X_{i},B_{i}+M_{i})$ is generalised $\delta$-lc if $\dim X_i>0$; 

\item there exist natural numbers $n_1,\dots,n_{i-1},v$ depending only on $d,i,\epsilon,\tau$ 
and there exists  an integral divisor $J\ge 0$ on $X$ such that for the general fibres $F$ of $X\to X_i$ 
we have 
$$
\mbox{$J|_F\sim -\sum_{j=1}^{i-1} n_j g_j^*K_{X_j}|_F$ and $0<\vol((B+J)|_F)<v$.}
$$ 
\end{enumerate}

To prove the claim 
we will apply induction on  $d$, so we assume that the claim holds in lower dimension.\\   

\emph{Step 3.}
\emph{In this step we consider the log general fibres of $(X,B)\to X_i$.}
Let $F$ be a general fibre of  $X=X_1\to X_i$, say over a closed point $v$. Then the sequence 
$$
X_1\to  X_2 \to \cdots \to X_i
$$
induces a sequence 
$$
F=G_1\to  G_2 \to \cdots \to G_{i-1}\to G_i=\{v\}
$$
of contractions where $G_j$ is the fibre of $X_j\to X_i$ over $v$. 
Moreover,  $-K_{G_j}$ is ample over $G_{j+1}$ 
as $K_{G_j}=K_{X_j}|_{G_j}$ and $-K_{X_j}$ is ample over $X_{j+1}$. 
In particular, $F\to \{v\}$ factors as a tower of Fano fibrations of length $i$.
Note that $G_{j}$ may consist of only one point 
for some $j<i$: this is the case when $X_j\to X_i$ is birational. In addition, 
$G_j\to G_{j+1}$ may be an isomorphism for some $j$. 

Let 
$$
K_{F}+B_{F}:=(K_X+B)|_{F}.
$$ 
Then $(F,B_F)$ is projective $\epsilon$-lc, $K_F+B_F\sim_\Q 0$, and each non-zero 
coefficient of $B_F$ is $\ge \tau$.\\

\emph{Step 4.}
\emph{In this step we apply induction on dimension.}
In this step assume $\dim X_i>0$ and let $F,G_j$ be as in the previous step. Then $\dim F<d$. 
Therefore, applying induction on dimension for the claim in Step 2 and using the fact that the general fibres 
of $F\to G_i=\{v\}$ are just $F$ itself, there exist natural numbers $n_1,\dots,n_{i-1},v$ 
depending only on $\dim F,i,\epsilon, \tau$ and there is an integral divisor $J_F\ge 0$ such that 
$$
\mbox{$J_F\sim -\sum_{j=1}^{i-1} n_j h_j^*K_{G_j}$ and $0<\vol(B_F+J_F)<v$} 
$$
where $h_j$ denotes $F\to G_j$. In particular, $-\sum_{j=1}^{i-1} n_j h_j^*K_{G_j}$ is an integral 
divisor. On the other hand, from $h_j^*K_{G_j}=g_j^*K_{X_j}|_F$, we get
$$
J_F\sim-\sum_{j=1}^{i-1} n_j h_j^*K_{G_j}= (-\sum_{j=1}^{i-1} n_j g_j^*K_{X_j})|_F.
$$
Since $F$ is a general fibre of $X\to X_i$, we deduce that $-\sum_{j=1}^{i-1} n_j g_j^*K_{X_j}$ is 
an integral divisor over the generic point of $X_i$ (note that since $F$ is a general fibre we have: 
if $P=\sum p_kD_k$ is an $\R$-divisor on $X$ where $D_k$ are distinct irreducible components, 
then $P|_F=\sum p_kD_k|_F$ where $D_k|_F$ are reduced divisors and there is no common component 
for distinct $k$; so the set of coefficients of $P|_F$ coincides with the set of horizontal$/X_i$ coefficients of $P$). Let 
$$
M=\rddown{-\sum_{j=1}^{i-1} n_j g_j^*K_{X_j}}.
$$ 
Then $M|_F\sim J_F\ge 0$, hence there is an integral divisor $0\le J\sim M/X_i$ (this can be seen 
by taking a resolution $W\to X$ and applying base change of cohomology to $W\to X_i$). In particular,  
$$
J|_F\sim M|_F\sim J_F\sim-\sum_1^{i-1} n_j g_j^*K_{X_j}|_F
$$ 
and 
$$
0<\vol((B+J)|_F)=\vol(B_F+J_F)<v.
$$\

\emph{Step 5.}
\emph{In this step we establish claim (1) of step 2.}
As mentioned earlier we can assume that the claim holds in lower dimension. 
If $i=1$, the claim holds trivially. Moreover, if $\dim X_i=0$, then claim (1) holds 
 as it is vacuous in this case. But if $\dim X_i>0$, then claim (1) follows by applying  
Corollary \ref{cor-cb-conj-sing-bnd-gen-fib} to $(X,B)\to X_i$ using the integral divisor $J$ of the previous step. 
We can assume that $\delta$ of claim (1) depends only on  $d,l,\epsilon,\tau$.\\

\emph{Step 6.}
\emph{In this step we work towards establishing claim (2) of step 2.}
We will prove claim (2). If $\dim X_i>0$, then it follows from the previous step. 
So assume $\dim X_i=0$ which means $i=l$ and that $X_{l-1}$ is a Fano variety. 
By claim (1), $(X_{l-1},B_{l-1}+M_{l-1})$ is generalised $\delta$-lc, hence $X_{l-1}$ is 
a $\delta$-lc Fano variety. Thus $X_{l-1}$ is bounded by [\ref{B-BAB}, Theorem 1.1] (=Theorem \ref{t-BAB}), 
so there are natural numbers $n_{l-1},v$
depending only on $d,\delta$ such that $-n_{l-1}K_{X_{l-1}}$ is very ample 
with volume less than $v$. In particular, we are done if $l=2$, so we can assume $l\ge 3$.
We will construct $n_i,\dots,n_{l-2}$ inductively and during the process we modify $n_{l-1},v$.
 
Denote $X_j\to X_k$ by $e_{j,k}$. Assume that  for some $2\le j\le l-1$ 
there exist natural numbers $n_j,\cdots,n_{l-1},v$ depending only on $d,j,\delta$ 
such that 
$$
H_{j}:=-\sum_{k=j}^{l-1} n_k e_{j,k}^*K_{X_k}
$$ 
is very ample on $X_j$ with volume less than $v$.  
By claim (1), $(X_{j-1},B_{j-1}+M_{j-1})$ is generalised $\delta$-lc. 
By assumption, $-K_{X_{j-1}}$ is ample over $X_{j}$. Moreover, since $\dim Z=\dim X_l=0$,  
$K_X+B\sim_\Q 0$ from which we get 
$$
K_{X_{j-1}}+B_{j-1}+M_{j-1}\sim_\Q 0.
$$
Then 
$$
(X_{j-1},B_{j-1}+M_{j-1})\to X_{j}
$$
is a generalised $(\dim X_{j-1}, v, \delta)$-Fano type fibration where we use the assumption that 
$H_{j}^{\dim X_{j}}<v$. Then by Lemma \ref{l-from-gen-fib-to-usual-fib}, there is a boundary $\Delta_{j-1}$ such that 
$(X_{j-1},\Delta_{j-1})\to X_{j}$ is a $(\dim X_{j-1}, v, \frac{\delta}{2})$-Fano type fibration.\\

\emph{Step 7.}
\emph{In this step we establish the claim of step 2 from which we derive the theorem.}
By Theorem \ref{t-log-bnd-cy-fib} (or \ref{t-log-bnd-cy-gen-fib}), 
 $X_{j-1}$ belongs to a bounded family. On the other hand,  by Lemma \ref{l-bnd-cy-fib-v-ampleness},
there exist bounded natural numbers $p,q$ such that the divisor 
$$
H_{j-1}:=p(qe_{j-1,j}^*H_j-K_{X_{j-1}})
$$ 
is very ample. In particular, letting $n_{j-1}:=p$ and replacing $n_j,\cdots,n_{l-1}$ with the numbers 
$qn_{j-1}n_j,\cdots, qn_{j-1}n_{l-1}$, respectively,  we can rewrite 
$$
H_{j-1}=-\sum_{k=j-1}^{l-1} n_k e_{j-1,k}^*K_{X_k}.
$$
 Applying Proposition \ref{l-bnd-cy-bndness-volume}, the volume of $H_{j-1}$ is bounded from above, so replacing $v$ 
we can assume $\vol(H_{j-1})<v$. 

Repeating the above process  gives bounded natural numbers $n_1$,$\cdots$,$n_{l-1}$,$v$ 
depending only on $d,l,\delta$ (hence depending only on $d,l,\epsilon,\tau$) such that 
$$
H_1:=-\sum_{k=1}^{l-1} n_k e_{1,k}^*K_{X_k}
$$ 
is very ample with volume less than $v$.  
In particular, $X=X_1$ belongs to a bounded family which in turn implies that $(X,B)$ is log bounded 
because $H_1^{d-1}\cdot B=-H_1^{d-1}\cdot K_X$ is bounded from above and because 
the coefficients of $B$ are $\ge \tau$. Moreover, we can find 
$$
0\le J\sim H_1=-\sum_{k=1}^{l-1} n_k e_{1,k}^*K_{X_k}=-\sum_{k=1}^{l-1} n_k g_{k}^*K_{X_k}
$$ 
and 
perhaps after replacing $v$  we can assume 
$$
0<\vol(B+J)=\vol(-K_X+H_1)<v.
$$ 
This proves claim (2) and finishes the proof of the theorem.

\end{proof}

\begin{proof}(of Theorem \ref{cor-bnd-cy-fib-non-product})
We follow the strategy in [\ref{DiCerbo-Svaldi}] which reduces the theorem to a special case of \ref{t-towers-of-Fanos}.
Taking a $\Q$-factorialisation we can assume $X$ is $\Q$-factorial. 
Since $B\neq 0$ and $(X,B)$ is not of product type, by [\ref{DiCerbo-Svaldi}, Theorem 3.2](=Theorem \ref{t-tower-of-Mfs}), 
there exist a birational map $\phi\colon X\bir X_1$ and a sequence of contractions 
$$
X_1\to  X_2 \to \cdots \to X_l
$$
such that $\phi^{-1}$ does not contract divisors, each $X_i\to X_{i+1}$ is a Mori fibre space, 
 and $X_l$ is a point. In particular, $l\le d$. Then applying Theorem \ref{t-towers-of-Fanos}, 
we deduce that $(X_1,B_1)$ is log bounded where $B_1=\phi_*B$.

On the other hand, since $K_X+B\sim_\Q 0$, each exceptional prime divisor $D$ of $\phi$ satisfies 
$$
a(D,X_1,B_1)=a(D,X,B)\le 1,
$$
hence there is a birational contraction $\psi \colon  X'\to X_1$ 
from a normal projective variety such that the induced map $X\bir X'$ is an isomorphism in codimension one. 
Let $B'$ on $X'$ be the birational transform of $B$. Then $(X',B')$ is a crepant model of $(X_1,B_1)$. 
Thus $(X',B')$ is log bounded by Theorem 1.3. 

\end{proof}

In the rest of this section we give different proofs of \ref{t-towers-of-Fanos} 
in certain special cases.\\ 

\emph{Alternative proof of \ref{t-towers-of-Fanos} when coefficients of 
$B$ are in a fixed DCC set $\Phi$.}

\emph{Step 1.}
\emph{In this step we apply adjunction to $(X,B)$ over each $X_i$.}
First, as in the previous proof we can assume $l\ge 2$ and that $\dim X_{l-1}>0$. 
Let $B_i$ and $M_i$ (resp. ${\bf{B}}_i$ and ${\bf{M}}_i$) be the discriminant and moduli divisors (resp. b-divisors)
of adjunction defined for $(X_1,B_1)=(X,B)$ over $X_i$ when $\dim X_i>0$. We consider $(X_i, B_i+M_i)$ as a 
generalised pair with data consisting of some 
high resolution $X_i'\to X_i$ and nef part $M_i'$ on $X_i'$ (as in \ref{rem-base-fib-gen-pair}). A crucial point is that, by 
Lemma \ref{l-adj-composition-contractions}, we have the following property: 
\begin{itemize}
\item[$(*)$] $B_i$  (resp. ${\bf{B}}_i$) coincides with the discriminant divisor (resp. b-divisor) of generalised adjunction for 
$$
(X_{i-1},B_{i-1}+M_{i-1}) ~~~\mbox{over $X_i$.}
$$ 
\end{itemize}
\medskip

\emph{Step 2.}
\emph{In this step we investigate the  $(X_{i},B_{i}+M_{i})$.}
We claim that there exist a DCC set $\Psi$, a natural number $p$, and a positive 
real number $\delta$ depending only on $d,l,\Phi,\epsilon$ such that for each $i$ we have: 
\begin{itemize}
\item the coefficients of $B_i$ belong to $\Psi$, 

\item we can choose $M_i'$ in its $\Q$-linear equivalence class so that ${p}M_i'$ is Cartier, and 

\item $(X_{i},B_{i}+M_{i})$ is generalised $\delta$-lc. 
\end{itemize}
For $i=1$ the claim holds trivially by taking $\Psi=\Phi$, $p=1$, and $\delta=\epsilon$.
Assuming that we have already found $\Psi,p,\delta$ which satisfy the claim up to $i-1\ge 1$, we 
prove the claim for $i$ (where we assume $\dim X_i>0$). 
By Lemma \ref{l-fib-adj-dcc},  the coefficients of $B_i$ belong to 
some DCC set $\tilde{\Psi}$ depending only on $d,\Phi$.\\   

\emph{Step 3.}
\emph{In this step we show that we can choose $M_i'$ with bounded Cartier index.}
Let $F$ be a general fibre of  $X_1\to X_i$ over a point $v$. 
Let 
$$
K_{F}+B_{F}:=(K_X+B)|_{F}.
$$ 
By induction on dimension, the set of such $(F,B_{F})$ form a log bounded family.
Thus there is a bounded natural number $\tilde{p}$ such that we can choose $M_i'$ 
in its $\Q$-linear equivalence class so that $\tilde{p}M_i'$ is Cartier: this follows from the same arguments as in the 
proof of [\ref{HX}, Claim 3.2 (3)] (note that in [\ref{HX}, Claim 3.2 (3)] it is implicitly assumed that $B_F$ 
is big but this is not needed in the proof once we know that $(F,B_F)$ is log bounded).
Alternatively, as in [\ref{B-compl}, Proposition 6.3], we can use boundedness of relative complements 
for $K_{X_{i-1}}+B_{i-1}+M_{i-1}$ over $X_i$ (similar to that of usual relative complements [\ref{B-compl}, Theorem 1.8])
to show that $\tilde{p}$ exists.\\

\emph{Step 4.}
\emph{In this step we finish the proof of the claim of Step 3.}
By $(*)$ above and by Corollary \ref{cor-cb-conj-sing-bnd-gen-fib} applied to $(X_{i-1},B_{i-1}+M_{i-1})$ 
over $X_i$, the b-divisor ${\bf{B}}_i$ has coefficients $\le 1-\tilde{\delta}$ for some 
positive real number $\tilde{\delta}$ depending only on $d,p,\Psi, \delta$. In particular,  
 $(X_i,B_i+M_i)$ is generalised $\tilde{\delta}$-lc. 
Now replace $\Psi,p,\delta$ 
with $\Psi\cup \tilde{\Psi}$, $p\tilde{p}$, $\delta\tilde{\delta}$, respectively. 
Then we can assume that the coefficients of $B_i$ are in $\Psi$, $pM_i'$ is Cartier, and that 
$(X_{i},B_{i}+M_{i})$ is generalised $\delta$-lc. This proves the above claim inductively.\\

\emph{Step 5.}
\emph{In this step we show that $X_{l-1}$ is bounded.}
Denote $X_{l-1}\to X_l$ by $h$. Pick $0\le \Delta_{l-1}\sim_\Q h^*A$ with coefficients in a fixed finite set so 
that 
$$
(X_{l-1},B_{l-1}+\Delta_{l-1}+M_{l-1})
$$ 
is still generalised $\delta$-lc with nef part $M_{l-1}'$. By assumption, $-(K_{X_{l-1}}+\Delta_{l-1})$ is ample over $X_{l}$. 
By construction, 
$$
K_{X_{l-1}}+B_{l-1}+\Delta_{l-1}+M_{l-1}\sim_\Q h^*(L+A).
$$
Then 
$$
(X_{l-1},B_{l-1}+\Delta_{l-1}+M_{l-1})\to X_{l}=Z
$$
is a generalised $(\dim X_{l-1}, r2^{\dim X_l}, \delta)$-Fano type fibration. 
Thus the pairs $(X_{l-1},\Delta_{l-1})$ form a log bounded family, by Lemma \ref{l-from-gen-fib-to-usual-fib} 
and Theorem \ref{t-log-bnd-cy-fib}. 
In particular, we can find a very ample divisor $H$ on $X_{l-1}$ such that $H^{\dim X_{l-1}}$ 
is bounded from above and $H-h^*A\sim_\Q H-\Delta_{l-1}$ is ample.\\ 

\emph{Step 6.}
\emph{In this step we finish the proof.}
Denote $X=X_1\to X_{l-1}$ by $g$. Then $K_X+B\sim_\Q g^*h^*L$ and 
$$
H-h^*L=H-h^*A+h^*(A-L)
$$ 
is ample. 
Therefore, if $l>2$, then we can apply induction on $l$. If $l=2$, 
then $X=X_{l-1}$ and $h=f$, in particular, $X$ belongs to a bounded family.
Moreover,
$$
H-(K_X+B)\sim_\Q H-f^*L
$$ 
is ample, hence  $H^{d-1}\cdot B$ is bounded from above.
This implies that $(X,B)$ is log bounded.\\

\emph{Alternative proof of \ref{t-towers-of-Fanos} when $K_X+B\sim_\Q 0$ 
and the coefficients of $B$ are in a fixed DCC set $\Phi$.}
The previous proof can be simplified in the case $K_X+B\sim_\Q 0$ 
in the sense that we do not need \ref{cor-cb-conj-sing-bnd-gen-fib}.
The proof goes along the same lines  except that we can show that $(X_i,B_i+M_i)$ is 
generalised $\delta$-lc  by a different argument rather than 
applying \ref{cor-cb-conj-sing-bnd-gen-fib} to the generalised Fano type 
fibration 
$$
(X_{i-1},B_{i-1}+M_{i-1}) \to X_i.
$$
Indeed the generalised pair $(X_i,B_i+M_i)$ is generalised klt satisfying the following properties 
with a fixed DCC set $\Psi$ and natural number $p$: 
\begin{itemize}
\item $B_i$ has coefficients in $\Psi$, and 

\item $pM_i'$ is Cartier.
\end{itemize}
When $K_X+B\sim_\Q 0$, we also have
\begin{itemize}
\item  
$
K_{X_{i}}+B_{i}+M_{i}\sim_\Q 0.
$
\end{itemize} 
But then $(X_i,B_i+M_i)$ is generalised $\delta$-lc for some fixed $\delta>0$, by Lemma \ref{l-bnd-sing-gen-cy-pairs}.
The rest of the proof is as before.\\



\vspace{2cm}

\textsc{DPMMS, Centre for Mathematical Sciences} \endgraf
\textsc{University of Cambridge,} \endgraf
\textsc{Wilberforce Road, Cambridge CB3 0WB, UK} \endgraf
\email{c.birkar@dpmms.cam.ac.uk\\}

\end{document}